 \documentclass[preprint]{elsarticle}

\def\norm#1{\|#1\|} 

\usepackage{algorithm}
\usepackage{algorithmic}
\usepackage{amsmath}
\usepackage{amsthm}
\usepackage{amssymb}
\usepackage{caption}
\usepackage{comment}
\usepackage{graphicx}
\usepackage{float}
\usepackage[latin1]{inputenc}
\usepackage[pdfborder={0 0 0}, colorlinks=true, linkcolor=blue, citecolor=red]{hyperref}\hypersetup{pdfborder=0 0 0}
\usepackage{mathrsfs}
\usepackage{multicol}
\usepackage{multirow}
\usepackage{natbib}
\setcitestyle{numbers}
\usepackage{subfigure}
\usepackage{wrapfig}
\usepackage{xcolor}
\usepackage{lineno,hyperref}
\modulolinenumbers[5]
\usepackage[textsize=footnotesize]{todonotes}
\usepackage{siunitx} 

\usepackage{tikz}
\usetikzlibrary{shapes}

\definecolor{cred}{RGB}{139,37,0}
\definecolor{cblue}{rgb}{0.0, 0.45, 0.73}
\definecolor{cgreen}{rgb}{0.2, 0.8, 0.2}
\definecolor{cyellow}{rgb}{1.0, 0.97, 0.0}
\definecolor{corange}{rgb}{1.0, 0.8, 0.0}

\newcommand{\sqboxs}{1.2ex}
\newcommand{\sqbox}[1]{\textcolor{#1}{\rule{\sqboxs}{\sqboxs}}}

\graphicspath{{./figures/}}

\newcommand{\beq} {\begin{equation}}
\newcommand{\eeq} {\end{equation}}
\newcommand{\bdm} {\begin{displaymath}}
\newcommand{\edm} {\end{displaymath}}
\newcommand{\bit}{\begin{itemize}}
\newcommand{\eit}{\end{itemize}}
\newcommand{\bde}{\begin{description}}
\newcommand{\ede}{\end{description}}
\newcommand{\bce}{\begin{center}}
\newcommand{\ece}{\end{center}}
\newcommand{\ben} {\begin{enumerate}}
\newcommand{\een} {\end{enumerate}}
\newcommand{\bea} {\begin{eqnarray}}
\newcommand{\eea} {\end{eqnarray}}
\newcommand{\barr} {\begin{array}}
\newcommand{\earr} {\end{array}}
\newcommand{\bean} {\begin{eqnarray*}}
\newcommand{\eean} {\end{eqnarray*}}
\newcommand{\edoc} {\end{document}}

\newcommand{\dd}[2] {\ensuremath{\frac{\partial {#1}}{\partial {#2}}}}
\newcommand{\DD}[2] {\ensuremath{\frac{d {#1}}{d {#2}}}}
\newcommand{\Grad} {\ensuremath{\nabla}}

\newcommand{\Div} {\ensuremath{\nabla\cdot}}

\newcommand{\eval}[2][\right]{\relax
  \ifx#1\right\relax \left.\fi#2#1\rvert}

\providecommand{\norm}[1]{\left\lVert#1\right\rVert}

\newcommand{\half} {\ensuremath{\frac{1}{2}}}

\newcommand{\mc}[1]{\mathcal{#1}}

\newcommand{\mb}[1]{\mathbf{#1}}

\newcommand{\nort}[1]{{\left\vert\kern-0.25ex\left\vert\kern-0.25ex\left\vert #1 
    \right\vert\kern-0.25ex\right\vert\kern-0.25ex\right\vert}}
\newcommand{\snor}[1]{\left| #1 \right|}
\newcommand{\LRp}[1]{\left( #1 \right)}
\newcommand{\LRs}[1]{\left[ #1 \right]}

\newcommand{\LRc}[1]{\left\{ #1 \right\}}

\newcounter{tablecell}

\newcommand{\figlab}[1]{\label{fig:#1}}
\newcommand{\eqnlab}[1]{\label{eq:#1}}
\newcommand{\theolab}[1]{\label{theo:#1}}

\newcommand{\tablab}[1]{\label{tab:#1}}

\newcommand{\figref}[1]{\ref{fig:#1}}

\newcommand{\eqnref}[1]{\eqref{eq:#1}}
\newcommand{\alglab}[1]{\label{alg:#1}}
\newcommand{\algref}[1]{\ref{alg:#1}}
\newcommand{\seclab}[1]{\label{sect:#1}}
\newcommand{\secref}[1]{\ref{sect:#1}}
\newcommand{\tabref}[1]{\ref{tab:#1}}

\renewcommand{\u}{u}

\newcommand{\q}{q}

\newcommand{\qb}{{\mb{\q}}}

\newcommand{\ub}{\mb{\u}}

\newcommand{\R}{{\mathbb{R}}}

\newcommand{\f}{f}
\newcommand{\g}{g}
\newcommand{\fb}{\mb{\f}}
\newcommand{\gb}{\mb{\g}}

\newcommand{\pOmega}{\partial \Omega}

\newcommand{\K}{K}

\newcommand{\pK}{{\partial \K}}

\newcommand{\Rb}{\mb{R}}

\renewcommand{\sb}{\mb{s}}

\newcommand{\Lb}{\mb{L}}
\newcommand{\Gb}{\mb{G}}

\newcommand{\tb}{{\bf t}}

\newcommand{\Ne}[1]{{N_{E_#1}}}

\newcommand{\rhouAn}[2]{{ (\rho u_1)^{#2}_{({#1})} }}

\newcommand{\rhowAn}[2]{{ (\rho w_1)^{#2}_{({#1})} }}

\newcommand{\Nze}[1]{{ N_{ze_{#1}} }}
\newcommand{\Nxe}[1]{{ N_{xe_{#1}} }}
\newcommand{\Nye}[1]{{ N_{ye_{#1}} }}

\newcommand{\jmhalf}{{j-\half}}
\newcommand{\jphalf}{{j+\half}}

\newcommand{\imhalf}{{i-\half}}
\newcommand{\iphalf}{{i+\half}}

\newcommand{\dz}{{\triangle z}}

\newcommand{\tila}{\tilde{a}}
\newcommand{\tilb}{\tilde{b}}
\newcommand{\tilc}{\tilde{c}}
\newcommand{\xb}{{\bf x}}

\newcommand{\Fb}{{\bf F}}

\newcommand{\Fbs}{{{\bf F}^*}}

\newcommand{\Gbs}{{{\bf G}^*}}

\newcommand{\nb}{{\bf n}}

\newcommand{\Fcal}{\mathcal{F}}

\newcommand{\dtt}{\triangle t}
\newcommand{\Qb}{{\bf Q}}

\newcommand{\mass}{\text{mass}}

\newcommand{\dt}{{\triangle t}}
\newcommand{\dx}{{\triangle x}}

\newcommand{\pres}{{{p}}}

\newcommand{\Ical}{\mathcal{I}}

\newcommand{\Nb}{{\bf N }}

\newtheorem{theorem}{Theorem}[section]

\newtheorem{proposition}[theorem]{Proposition}

\begin{document}

\begin{frontmatter}

\title{Mass-Conserving Implicit-Explicit Methods for Coupled Compressible Navier--Stokes Equations}
\author[Argonne]{Shinhoo Kang\corref{mycorrespondingauthor}}
\cortext[mycorrespondingauthor]{Corresponding author}
\ead{shinhoo.kang@anl.gov}
\author[Argonne]{Emil M. Constantinescu}
\ead{emconsta@anl.gov}
\author[Argonne]{Hong Zhang}
\ead{hongzhang@anl.gov}
\author[Argonne-ESD]{Robert L. Jacob}
\ead{jacob@anl.gov}

\address[Argonne]{Mathematics and Computer Science Division, Argonne National Laboratory, Lemont, IL, USA}
\address[Argonne-ESD]{Environmental Science Division, Argonne National Laboratory, Lemont, IL, USA}

\begin{abstract}
  Earth system models are composed of coupled components that separately model systems such as the global atmosphere, ocean, and land surface.
  While these components are well developed, coupling them in a single system can be a
  significant challenge. Computational efficiency, accuracy, and
  stability are principal concerns.   
  In this study we focus on these issues.  
  In particular, implicit-explicit (IMEX) tight and loose coupling strategies 
  are explored for handling different time scales. 
  For a simplified model for the air-sea interaction problem, 
  we consider coupled compressible Navier--Stokes equations with an interface condition. 
  Under the rigid-lid assumption, horizontal momentum and heat flux are exchanged through the interface. 
  Several numerical experiments are presented to demonstrate the stability of the coupling schemes. 
  We  show both numerically and theoretically  that our IMEX coupling methods are mass conservative  
  for a coupled compressible Navier--Stokes system with the rigid-lid condition. 
\end{abstract}


\end{frontmatter}


\section{Introduction}

The rapid development of high-performance computational resources and 
state-of-art Earth system models (ESMs), 
for example, the Energy Exascale Earth System Model (E3SM) \cite{golaz2019doe}
and EC-Earth \cite{hazeleger2010ec}, 
enhance 
our understanding of important processes in the Earth system, 
such as the global water cycle \cite{gentine2019coupling} or global biogeochemical cycles \cite{burrows2020doe}),  
the predictability of the climate system \cite{collins2006community} including
global warming \cite{anderson2016co2}, and sea level rise \cite{hoffman2019effect}. 
  
ESMs are composed of models of various components such as atmosphere, ocean, ice, river, and land. 
Each component has its own temporal and spatial scale, and  thus each can have its own computational grid and timestep. 
Moreover, the governing equations, while related, are not identical; and a
proper synchronization, interpolation, or projection of the state
variables is necessary between components.  
In order to address such issues, 
several couplers are used in ESMs to control 
the interactions between these components \cite{golaz2019doe}; examples include CPL7 with the Model Coupling Toolkit \cite{craig2012new, jacob2005m} for E3SM and the Community Earth Systems Model, 
and 
OASIS3 \cite{valcke2013oasis3} for the EC-Earth ESM. These couplers control the overall time integration of the ESM and are carefully constructed to allow hundreds of simulated years of integration.

However, challenges still exist.
For some component pairings and choices of coupling frequency, the coupling procedure can trigger numerical instability
\cite{hallberg122014numerical,lemarie2015analysis, beljaars2017numerical}.
The overall accuracy and stability of the coupling schemes are not widely explored 
in the context of Earth system models because of their complexity and computational demands 
\cite{zhang2020stability}. Recent efforts target  
various coupling strategies  such as 
synchronous partitioned schemes for convection-diffusion equations
\cite{peterson2019explicit,sockwell2020interface}, 
partitioned coupling algorithms for diffusion equations
\cite{zhang2020stability}, global-in-time Schwarz methods for
diffusion equations \cite{lemarie2014sensitivity}, operator splitting
methods for incompressible Navier--Stokes (INS) equations
\cite{bresch2006operator}, 
subdomain iteration methods for coupled 3D and 1D INS equations
\cite{formaggia2001coupling},
and sequential and concurrent coupling
approaches for the Boussinesq convection model
\cite{connors2019stability}. 
In \cite{carpenter2014entropy}, 
    entropy stable compressible Navier--Stokes (CNS) systems is coupled with synchronous explicit Runge-Kutta methods in discontiuous Galerkin spatial discretization. 
    
There are different dimensions in terms of how we partition and couple the different systems. 
For example, with implicit-implicit time-stepping, 
both the atmospheric and the ocean models are advanced implicitly.
However, in our study, we focus on implicit-explicit time stepping. 
Implicit-explicit (IMEX) methods are an important class of methods
developed to efficiently handle problems that have both stiff and nonstiff
components  by solving stiff terms implicitly and nonstiff terms explicitly 
 \cite{ascher1997implicit,pareschi2005implicit,Constantinescu_A2010a}. Several IMEX Runge--Kutta (IMEX-RK) methods have been proposed 
\cite{Kennedy2003additive,roldan2013efficient, boscarino2017unified}.
IMEX methods are successfully adapted 
to handle geometrically induced stiffness \cite{kanevsky2007application},
to relax scale-separable stiffness in shallow water equations 
\cite{kang2019imex} and non-hydrostatic systems \cite{restelli2009conservative, giraldo2010high, giraldo2013implicit, gardner2018implicit, vogl2019evaluation,Abdi_2019}. 
In particular, the authors in \cite{gardner2018implicit} 
studied various IMEX methods for evolving acoustic waves implicitly to enable larger time step sizes in a global non-hydrostatic atmospheric model. 
IMEX methods are also applied to 
coupling different systems including  fluid-structure interaction \cite{froehle2014high} and 
particle-laden flow \cite{huang2019high}.
 
We propose a set of IMEX-based tight and loose coupling methods
for coupled compressible Navier--Stokes (CNS) equations 
with a rigid-lid coupling condition in finite volume (FV) spatial discretization. 
IMEX tight coupling schemes treat the ocean implicitly whereas the atmosphere explicitly by using the partitioned IMEX method.
IMEX loose coupling schemes handle the ocean semi-implicitly by using the standard IMEX method whereas the atmosphere explicitly.
In particular, we choose additive Runge--Kutta (ARK) methods 
because the schemes  not only support high-order accuracy in time but also 
provide simple embedded schemes that can be used for interpolation or
extrapolation of stage values in time \cite{Kennedy2003additive}.  
Furthermore, we demonstrate the mass-conserving property of the proposed coupling methods and 
show the performance of IMEX coupling schemes through numerical examples. 
  
In this study, we use a setup that is simplified when compared with operational
models but that still maintains the same temporal challenges associated
with coupling air-sea interaction. The physical domains of the ocean and of the
atmosphere are different, one has side boundaries and the other does not, but both are dry ideal gases. 
We do not consider gravitational forcing.
Incorporating the gravity term is desirable for atmospheric and ocean modelers, but we limit this study to the simplest problem possible to focus on the coupling the heat and horizontal momentum across the interface using IMEX approaches, which is not directly affected by gravity. 
We also consider conformal grids across
both domains and the  interface. These simplifications allow us to
focus on the effect of the jumps in temperature and velocity across
the interface, for which we use a linear bulk flux formula. This 
provides a baseline for the development of coupling methods through
such interfaces and illustrates challenges associated broadly with coupled systems.

  
This paper is organized as follows.
We begin in Section \secref{model-problems}
by describing the coupled CNS systems 
with rigid-lid interface condition as well as bulk flux and 
FV spatial discretization on a uniform mesh. 
In Section \secref{IMEXCoupling} 
we illustrate IMEX-RK tight and loose coupling methods and their mass-conserving property.
In Section \secref{NumericalResults} we discuss the performance of IMEX-RK coupling schemes 
through numerical simulations. 
In Section \secref{Conclusion} we present our conclusions.

\section{Model problems}
\seclab{model-problems}

We consider two ideal gas fluids governed by CNS. 
The domain $\Omega$ consists of two subdomains $\Omega_1$ and $\Omega_2$, 
which are vertically separated by 
the interface $\Gamma = \overline{\Omega}_1 \cap \overline{\Omega}_2$, as shown  
in Figure \figref{coupledmodel_cns}.


%
%
%

The homogeneous CNS equations in $\Omega_m$ ($m\in\{1,2\}$) are described by
\begin{subequations}
\eqnlab{cns-gov}
\begin{align}
  \eqnlab{cns-mass}
  \dd{\rho_m}{t}     + \Div \LRp{\rho \ub_m} &= 0,\\
  \eqnlab{cns-momentum}
  \dd{\rho \ub_m}{t} + \Div \LRp{\rho \ub_m \otimes \ub_m + \Ical \pres_m} &= \Div \sigma_m,\\
  \eqnlab{cns-energy}
  \dd{\rho E_m}{t}   + \Div \LRp{\rho \ub_m H_m} &= \Div (\sigma_m \ub_m) - \Div \Pi_m,
\end{align}
\end{subequations}
\begin{wrapfigure}{r}{0.25\textwidth}
  \centering
  \includegraphics[width=0.24\textwidth]{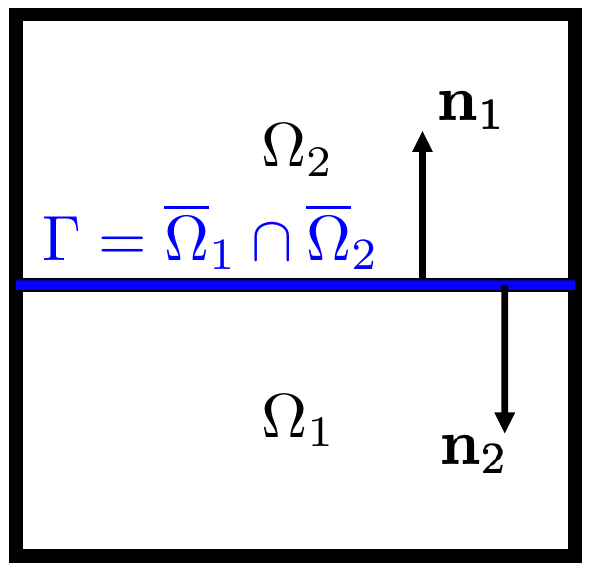}  
  \caption{Schematic of a coupled model.} 
  \figlab{coupledmodel_cns}
\end{wrapfigure}
where $\rho_m$ is the density $[\si{\kilo\gram\per\cubic\meter}]$;
$\ub_m$ is the velocity vector $[\si{\meter\per\second}]$;\footnote{
  $\ub_m=(u_m,w_m)^T$ in two-dimensions.
}
$\pres_m$ is the pressure $[\si{\newton\per\meter\squared}]$;
$\rho E_m = \rho_m e_m + \half \rho_m \norm{\ub_m}^2 $ is the total energy $[\si{\joule\per\cubic\meter}]$;
$e_m=\frac{\pres_m}{\rho_m(\gamma_m -1)}$ is the internal energy $[\si{\joule\per\kilogram}]$;
$H_m = E_m + \frac{\pres_m}{\rho_m} = \frac{a_m^2}{\gamma_m -1} + \frac{1}{2}\norm{\ub_m}^2$ is the total specific enthalpy $[\si{\joule\per\kilogram}]$;
$\sigma_m = \mu_m \LRp{\Grad \ub_m + \Grad (\ub_m)^T - \frac{2}{3} \Ical \Div \ub_m } $ is the viscous stress tensor; 
$\Pi_m = -\kappa_m \Grad T_m$ is the heat flux; 
$T_m$ is the temperature; 
$\kappa_m = \mu_m (c_p)_m Pr_m^{-1}$ is the heat conductivity $[\si{\watt\per\meter\per\kelvin}]$;
$\mu_m$ is the dynamic viscosity $[\si{\pascal\second}]$;
$Pr_m$ is Prantl number; 
$a_m = \LRp{\gamma_m\frac{\pres_m}{\rho_m}}^\half$ is the sound speed $[\si{\meter\per\second}]$ for ideal gas;
$\gamma_m = \LRp{\frac{c_p}{c_v}}_m$ is the ratio of the specific heats; 
and $(c_p)_m$ and $(c_v)_m$ are the specific heat capacities at constant pressure and at constant volume $[\si{\joule\per\kilogram \per\kelvin}]$, respectively. 
In a compact form, \eqnref{cns-gov} can be written as
\begin{align}
\eqnlab{cns-gov-compact}
    \dd{\qb_m}{t}  + \Div \Fcal_m^I(\qb_m) = \Div \Fcal_m^V(\qb_m),
\end{align}
with the conservative variable $\qb_m = (\rho_m, \rho \ub_m^T, \rho E_m)^T$, 
the inviscid flux tensor 
$\Fcal_m^I=(\rho \ub_m, \rho \ub_m \otimes \ub_m + \Ical \pres_m, \rho \ub_m H_m)^T$, and 
the viscous flux tensor 
$\Fcal_m^V=(0, \sigma_m, \sigma_m \ub_m - \Pi_m)^T$  for $\Omega_m$.

\subsection{Interface conditions}
 
The interface $\Gamma$ allows the exchange of heat and horizontal momentum
fluxes between the two domains. To implement this, 
we adapt the rigid-lid assumption used in many oceanography models  \cite{washington2005introduction,muller2006equations}.
\footnote{While many modern global ocean models no longer use the rigid-lid assumption, the free surface variations are ignored in coupling and in this study we want the treatment of the ocean surface to be consistent.}
In the rigid-lid coupling condition \cite{bresch2006operator,connors2012decoupled},
the normal velocity component and normal traction vector at the interface are set to zero, 
but the continuity of the tangential traction vector and heat flux are enforced: 
\begin{subequations}
\eqnlab{rigid-lid-cond}
\begin{align}
  \nb_m \cdot \ub_m  &= 0,\\
  \nb_m \cdot \LRp{\sigma_m \nb_m} &= 0,\\
  \tb \cdot \LRp{\sigma_1 \nb_1 - \sigma_2 \nb_2}  &= 0,\\
  \nb_1 \cdot \Pi_1 + \nb_2 \cdot \Pi_2 &=0.
\end{align}
\end{subequations}
Here, $\nb_m$ is the outward unit normal vectors on $\pOmega_m$ toward the adjacent domain of $\Omega_m$,
and $\tb$ is the tangent vector on $\Gamma$.
In two dimensions, the rigid-lid interface condition \eqnref{rigid-lid-cond} is simplified as\footnote{
Since $v=0$ at the interface, we have $\dd{v}{x}=0$ on the interface.

}  
\begin{align*}
  \hat{w}_1 &= \hat{w}_2 = 0,\\
  \hat{\sigma}_{1zz} &= \hat{\sigma}_{2zz} = 0,\\
  \mu_1 \dd{u_1}{z} &= \mu_2 \dd{u_2}{z} =: \hat{\sigma}_{xz} ,\\ 
   - \kappa_1 \dd{T_1}{z} &= - \kappa_2 \dd{T_2}{z} =: \hat{\Pi}_z,
\end{align*}
with $\nb_1 = (0,1)^T$, $\nb_2=(0,-1)^T$ and $\tb=(1,0)^T$. 

The next question is how to parameterize the physical heat and horizontal momentum fluxes across the interface.
According to previous work in \cite{liu1979bulk,smith1988coefficients,fairall1996bulk,bao2000numerical},
the bulk flux approach is widely used, 
where the physical fluxes are related to the measurable quantities such as wind-ocean current speed and lower air-sea surface temperature.
For example, by introducing the bulk coefficients $b_u$ and $b_T$, 
we can associate temperature and velocity with the heat and the horizontal momentum fluxes: 
\begin{align}
  \eqnlab{bulk-flux}
  \hat{\sigma}_{xz} := b_u (u_2 - u_1) \text{ and }
  \hat{\Pi}_{z} := - b_T (T_2 - T_1).
\end{align}
Several forms of the bulk coefficients have been proposed in earlier work \cite{panosfsky1984atmospheric,fairall1996bulk,vickers2015formulation}, 
but in this study we use the linear bulk coefficients (with constant $\mu_1$,$\mu_2$,$\kappa_1$ and $\kappa_2$): 
\begin{align*}
  b_u = \frac{2 \mu_1 \mu_2}{\dz_2\mu_1 + \dz_1\mu_2}, \text{ and }
  b_T = \frac{2 \kappa_1 \kappa_2}{\dz_2\kappa_1 + \dz_1\kappa_2}, 
\end{align*}
which come from the finite difference (FD) approximation of heat and horizontal momentum fluxes.\footnote{
  Assume there exist a common velocity $\hat{u}$ and temperature $\hat{T}$ at the interface. 
  From the FD approximation of the heat and horizontal momentum fluxes, 
  \begin{align*}
    \mu_1 \frac{\hat{u} - u_1}{\dz_1/2} 
    = \mu_2 \frac{u_2 - \hat{u}}{\dz_2/2}, \text{ and }
    - \kappa_1 \frac{\hat{T} - T_1}{\dz_1/2} 
    = - \kappa_2 \frac{T_2 - \hat{T}}{\dz_2/2},
    \end{align*}
  we have the weighted velocity $\hat{u}$ and temperature $\hat{T}$ by 
  \begin{align*}
    \hat{u} = \frac{\dz_2\mu_1 u_1     + \dz_1\mu_2 u_2}{   \dz_2\mu_1    + \dz_1\mu_2}, \text{ and }
    \hat{T} = \frac{\dz_2\kappa_1 T_1  + \dz_1\kappa_2 T_2}{\dz_2\kappa_1 + \dz_1\kappa_2}. 
  \end{align*}
  Substituting $\hat{u}$ and $\hat{T}$ in the FD form above, we obtain the linear bulk coefficients. 
}

Once we have computed the heat and the momentum fluxes at the interface by using \eqnref{bulk-flux}, 
we estimate the isothermal wall boundary states of $u_w$ and $T_w$ for $\Omega_1$ and $\Omega_2$, respectively,
\begin{align*}
  u_{w_1} = u_1 + \frac{\hat{\sigma}_{xz} \dz_1}{2\mu_1 } \text{ and } T_{w_1} = T_1 - \frac{\hat{\Pi}_z \dz_1}{2\kappa_1 },\\
  u_{w_2} = u_2 - \frac{\hat{\sigma}_{xz} \dz_2}{2\mu_2 } \text{ and } T_{w_2} = T_2 + \frac{\hat{\Pi}_z \dz_2}{2\kappa_2 },
\end{align*}
by using FD approximation.\footnote{
  \begin{align*}
    \hat{\sigma}_{xz}
      = \mu n_z \dd{u}{z} 
      \approx  \mu n_z \frac{u_w - u}{\dz/2},\\
    \hat{\Pi}_z 
      = - \kappa n_z \dd{T}{z} 
      \approx - \kappa n_z \frac{T_w - T}{\dz/2}
  \end{align*}  
}
Then, we apply the isothermal wall boundary treatment \cite{jacobs2007conservative} to the interface. 

\subsection{Nondimensionalization}

By the nondimensional variables, 
$$ 
\rho^*_m = \frac{\rho_m}{\rho_r}, 
\pres^*_m = \frac{\pres_m}{\rho_r u_r^2},
\ub^*_m = \frac{\ub_m}{u_r},
x^* = \frac{x}{L},
t^* = \frac{t}{L/u_r},
\mu^*_m = \frac{\mu_m}{\mu_r},
\text{ and }
T^*_m = \frac{T_m}{T_r},
$$

\noindent
we rewrite the governing equation \eqnref{cns-gov} as  
\begin{subequations}
\eqnlab{cns-gov-nondimension}
\begin{align}
  \dd{\rho_m^*}{t^*}     + \Grad^*\cdot \LRp{\rho \ub_m^*} &= 0,\\
  \dd{\rho \ub_m^*}{t^*} + \Grad^*\cdot \LRp{\rho \ub_m^* \otimes \ub_m^* + \Ical \pres_m^*} &= \Grad^* \cdot \sigma_m^*,\\
  \dd{\rho E_m^*}{t^*}   + \Grad^*\cdot \LRp{\rho \ub_m^* H_m^*} &= \Grad^*\cdot (\sigma_m^* \ub_m^*) - \Grad^* \cdot \Pi_m^*,
\end{align}
\end{subequations}
where 
$\sigma_m^* = \tilde{\mu}_m \LRp{\Grad^* \ub^*_m + \Grad^* (\ub^*_m)^T -\frac{2}{3}\Ical \Grad^* \cdot \ub^*_m } $,  
$\Pi^*_m = \frac{\tilde{c}_p\tilde{\mu}_m}{Pr} \Grad^* T^*_m $,  $\Grad^*=\frac{1}{L}\Grad$,
$\tilde{\mu}_m := \frac{\mu^*_m}{Re_r}$, 
$Re_r := \frac{\rho_r u_r L}{\mu_r}$
and $\tilde{c}_p := \frac{T_r c_p}{u_r^2}$.


We take $\rho_r = \rho_{\infty_2}$, $\mu_r=\mu_{\infty_2}$, $T_r=T_{\infty_2}$, and the speed of sound as a reference velocity, $u_r = a_{\infty_2} = \sqrt{\gamma R T_{\infty_2}}$.\footnote{
  This leads to $\tilde{c}_p = \frac{c_p}{\gamma R} = \frac{1}{\gamma-1}$ 
  and $ Re_r = \frac{\rho_{\infty_2} a_{\infty_2} L}{\mu_{\infty_2}} 
  = \frac{\rho_{\infty_2} u_{\infty_2} L}{\mu_{\infty_2}} \frac{a_{\infty_2}}{u_{\infty_2}} = \frac{Re_{\infty_2}}{M_{\infty_2}}$.
}
The normalized equation of state (EOS) for an ideal gas is $\pres^*_m = \gamma^{-1} \rho^*_m T^*_m = \rho^*_m e^*_m(\gamma-1)$.

In this study, we use the nondimensionalized form \eqnref{cns-gov-nondimension}, 
and we omit the superscript ($*$) hereafter.

\subsection{Finite volume discretization}

We denote by $\Omega_{m_h} := \cup_{\ell=1}^\Ne{m} K_{m_{\ell}}$ the mesh containing a finite
collection of non-overlapping elements, $K_{m_{\ell}}$, that partition $\Omega_m$.
For example, in a two-dimensional Cartesian coordinates system, we have 
$\Ne{m} = \Nxe{m} \times \Nze{m}$.
For clarity, we abbreviate the subscript $m$ in this section.

Integrating \eqnref{cns-gov-compact} over elements, applying the divergence theorem,
and introducing a numerical flux, 
$\nb\cdot \Fcal^*$, we obtain a cell-centered FV scheme for the $\ell$ element,
\begin{align}
\eqnlab{cns-fv}
  \DD{\overline{\qb}_\ell}{t} 
  = - \frac{1}{\snor{\K_\ell}}\int_{\pK_\ell} \nb \cdot \Fcal^* d\pK
  = - \frac{1}{\snor{\K_\ell}}\sum_{f\in\pK_\ell}\int_{f} \nb \cdot \Fcal^* df ,
\end{align}
where $\nb$ is the outward unit normal vector on the boundary $\pK$ of the element $\K$,
$f$ is the elemental face of $\pK$, 
 $\overline{\qb}_\ell = \snor{\K_\ell}^{-1} \int_{\K_\ell} \DD{\qb}{t} d\K$ is the average state variable 
in $\K_\ell$, and $\snor{\K_\ell}$ is the Lebesgue measure of element $\K_\ell$. 

The numerical flux $\Fcal^*=\Fcal^{I^*} - \Fcal^{V^*}$ is composed of two parts: inviscid and viscous. 
For the inviscid part, we employ the upwind flux based on Roe's approximation, \cite{roe1986characteristic}:  
\begin{align*}
\nb \cdot \Fcal^{I^*} 
 &= \half\LRp{\Fcal^{I}(\qb^l)  + \Fcal^{I}(\qb^r )} \cdot \nb
 + \frac{|A|}{2}\LRp{\qb^r \nb^r + \qb^l \nb^l} \cdot \nb ,
\end{align*}
where 
$\qb^l$ and $\qb^r$ are the reconstructed values from the left and the right sides of the elemental face $f$; 
$|A|:=\mc{R} |\Lambda| \mc{R}^{-1}$;
$A:=\dd{\Fcal^I\cdot\nb}{\qb} = \mc{R} \Lambda \mc{R}^{-1}$ is the flux Jacobian; 
and $\mc{R}$ and $\Lambda$ are eigenvectors and eigenvalues of the flux Jacobian (details can be found in \cite{toro2013riemann}).
For the viscous part, inspired by \cite{nishikawa2011two},
we first compute the common velocity ($\hat{\ub}$), 
common velocity gradient ($\widehat{\Grad \ub}$), and 
common temperature gradient ($\widehat{\Grad T}$) at the elemental face $f$ 
and then evaluate the viscous flux 
$$
\nb \cdot \Fcal^{V^*} := \nb \cdot \Fcal^{V}\LRp{ \hat{\ub},\widehat{\Grad \ub},\widehat{\Grad T} }.
$$
 
\subsubsection*{Two-dimensional $(x,z)$ uniform grid}
With a two-dimensional uniform mesh, \eqnref{cns-fv} for the $\ell=(i,j)$ element becomes 
\begin{align}
  \DD{\overline{\qb}_{(i,j)}}{t} 
  = 
   - \frac{1}{\dx}\LRp{\Fbs_{(\iphalf,j)}  - \Fbs_{(\imhalf,j)}} 
   - \frac{1}{\dz}\LRp{\Gbs_{(i,\jphalf)}  - \Gbs_{(i,\jmhalf)}},
  \eqnlab{euler-fv-cartesian-2d}   
\end{align}
where 
$\Fbs = \Fb^{I^*} - \Fb^{V^*}$,
$\Gbs = \Gb^{I^*} - \Gb^{V^*}$,
$\Fb^{I^*} = 
  \LRp{\rho u, \rho u u + \pres, \rho u w, \rho u H}^{*^T}$, 
$\Gbs^{I^*} = 
  \LRp{\rho w, \rho w u, \rho w w + \pres, \rho w H}^{*^T}$,
$\Fb^{V^*} = 
  \LRp{0, \sigma_{xx}, \sigma_{xz}, \sigma_{xx}u + \sigma_{xz} w - \Pi_x}^{*^T}$, and 
$\Gbs^{V^*} = 
  \LRp{0, \sigma_{zx}, \sigma_{zz}, \sigma_{zx}u + \sigma_{zz} w - \Pi_z}^{*^T}$.  

For the cell-centered second-order FV, 
we compute the numerical fluxes in \eqnref{euler-fv-cartesian-2d} using linearly reconstructed variables. 
Since each element has only cell-averaged values, 
the reconstruction requires information from adjacent elements. 
For the inviscid flux ($\Fb^{I^*}\text{ and }\Gb^{I^*}$), 
we compute the cell-centered gradients $\overline{\Grad \qb}$ using the least-square (LS) scheme \cite{syrakos2017critical}.\footnote{
  In uniform mesh, LS schemes approximate $x$ and $z$ directional gradients by 
  \begin{align*}
    \overline{\dd{\qb}{x}}_{(i,j)} & \approx \frac{\overline{\qb}_{(i+1,j)} - \overline{\qb}_{(i-1,j)}  }{2\dx} + \mathcal{O}((\dx)^2),\\ 
    \overline{\dd{\qb}{z}}_{(i,j)} & \approx \frac{\overline{\qb}_{(i,j+1)} - \overline{\qb}_{(i,j+1)}  }{2\dz} + \mathcal{O}((\dz)^2).
  \end{align*}
} 
With the gradients, the left and the right conserved variables ($\qb^l$ and $\qb^r$) 
are obtained by 
\begin{align*}
  \qb_{(i+\half,j)}^l &= \overline{\qb}_{(i,j)}   + \overline{\dd{\qb}{x}}_{(i,j)} \frac{\dx}{2}, \quad
  \qb_{(i+\half,j)}^r = \overline{\qb}_{(i+1,j)} - \overline{\dd{\qb}{x}}_{(i+1,j)} \frac{\dx}{2}  
\end{align*}
at the $x$-face of $(i+\half,j)$ 
and 
\begin{align*}
  \qb_{(i,j+\half)}^l &= \overline{\qb}_{(i,j)}   + \overline{\dd{\qb}{z}}_{(i,j)} \frac{\dz}{2}, \quad
  \qb_{(i,j+\half)}^r = \overline{\qb}_{(i,j+1)} - \overline{\dd{\qb}{z}}_{(i,j+1)} \frac{\dz}{2} 
\end{align*}
at the $z$-face of $(i,j+\half)$. 

For the viscous flux ($\Fb^{V^*}\text{ and }\Gb^{V^*}$), 
we compute the cell-centered gradients of velocity and temperature ($\overline{\Grad \ub}$ and $\overline{\Grad T}$) by using the LS scheme, and  we compute
the common velocity of $\hat{\ub}$,
\begin{align*}
  \hat{\ub}_{(i+\half,j)} = \half\LRp{ \overline{\ub}_{(i,j)} + \overline{\ub}_{(i+1,j)} }, \quad
  \hat{\ub}_{(i,j+\half)} = \half\LRp{ \overline{\ub}_{(i,j)} + \overline{\ub}_{(i,j+1)} },
\end{align*}
by taking an arithmetic average. 
As for the common gradients of $\widehat{\Grad \ub}$ and $\widehat{\Grad T}$, 
we obtain normal gradients by finite difference approximation
\begin{align*}
  \widehat{\dd{u}{x}}_{(i+\half,j)} &= \frac{1}{\dx}\LRp{ \overline{u}_{(i+1,j)} - \overline{u}_{(i,j)} },& 
  \widehat{\dd{u}{z}}_{(i,j+\half)} &= \frac{1}{\dz}\LRp{ \overline{u}_{(i,j+1)} - \overline{u}_{(i,j)} }, \\
  \widehat{\dd{w}{x}}_{(i+\half,j)} &= \frac{1}{\dx}\LRp{ \overline{w}_{(i+1,j)} - \overline{w}_{(i,j)} },&   
  \widehat{\dd{w}{z}}_{(i,j+\half)} &= \frac{1}{\dz}\LRp{ \overline{w}_{(i,j+1)} - \overline{w}_{(i,j)} },  \\
  \widehat{\dd{T}{x}}_{(i+\half,j)} &= \frac{1}{\dx}\LRp{ \overline{T}_{(i+1,j)} - \overline{T}_{(i,j)} },&
  \widehat{\dd{T}{z}}_{(i,j+\half)} &= \frac{1}{\dz}\LRp{ \overline{T}_{(i,j+1)} - \overline{T}_{(i,j)} }, 
\end{align*}
and tangential gradients by averaging two adjacent cell-centered gradients,
\begin{align*}
  \widehat{\dd{u}{z}}_{(i+\half,j)} &= \half\LRp{ \overline{\dd{u}{z}}_{(i,j)} + \overline{\dd{u}{z}}_{(i+1,j)} },&
  \widehat{\dd{u}{x}}_{(i,j+\half)} &= \half\LRp{ \overline{\dd{u}{z}}_{(i,j)} + \overline{\dd{u}{z}}_{(i,j+1)} }, \\
  \widehat{\dd{w}{x}}_{(i,j+\half)} &= \half\LRp{ \overline{\dd{w}{z}}_{(i,j)} + \overline{\dd{w}{z}}_{(i,j+1)} }, &
  \widehat{\dd{w}{z}}_{(i+\half,j)} &= \half\LRp{ \overline{\dd{w}{z}}_{(i,j)} + \overline{\dd{w}{z}}_{(i+1,j)} }, \\
  \widehat{\dd{T}{x}}_{(i,j+\half)} &= \half\LRp{ \overline{\dd{T}{z}}_{(i,j)} + \overline{\dd{T}{z}}_{(i,j+1)} },&
  \widehat{\dd{T}{z}}_{(i+\half,j)} &= \half\LRp{ \overline{\dd{T}{z}}_{(i,j)} + \overline{\dd{T}{z}}_{(i+1,j)} }.
\end{align*}

\section{IMEX coupling framework}
\seclab{IMEXCoupling}


In this section we introduce IMEX coupling methods 
in which we treat one model implicitly and the other model explicitly. 
IMEX coupling schemes naturally support two-way coupling 
for all stages. 
This means that at every time stage 
the continuity of the heat and the horizontal momentum fluxes is ensured by construction. For example, at the $i$th stage, 
$$ \hat{\sigma}_{xz}^{(i)} = b_u (u^{(i)}_2 - u^{(i)}_1), 
\text{ and } \hat{\Pi}_z^{(i)} = -b_T (T^{(i)}_2 - T^{(i)}_1). $$ 
We denote two-way coupling at every stage as tight coupling (TC),
which is represented by double arrows in Figure \figref{coupling-diagram}(a). 

With two domains that exhibit different stiffness properties,\footnote{
  For example, the atmospheric model and the ocean model have $110 \si{\kilo\metre}$ and $30\si{\kilo\metre} \sim 60 \si{\kilo\metre}$ grid sizes \cite{golaz2019doe}. 
  Typical acoustic wave speed is about $1500 \si{\metre\per\second}$ for seawater \cite{cristini2012some} and $340 \si{\metre\per\second}$ for the atmosphere. 
} IMEX methods
represent a suitable alternative to monolithic implicit methods by
alleviating the cost of applying an implicit method on the less stiff partition.   
Moreover, the IMEX tight coupling approach supports a high-order solution in time,
reduces the computational cost arising from an implicit monolithic approach,
and allows a large timestep size compared with that of a fully explicit coupling approach.
However, the tight coupling requires frequent communications between the two models across the interface 
and also sequentially advances each model (e.g., runs the atmosphere model first and then updates the ocean model). 
Furthermore, each model is advanced with the same timestep size.

To confer more timestepping granularity, one may consider loosening the tight coupling  
so that each model can march with a different timestep size \cite{constantinescu2007multirate,schlegel2009multirate,constantinescu2013extrapolated,seny2013multirate,sandu2019class}, for example, 
$\dt_{ocn} = N_s\dt_{atm}$, or can run simultaneously. 
Therefore, we consider two loose coupling (LC)  schemes: concurrent and sequential.\footnote{ 
  Here, we borrow the terms concurrent and sequential modes used by \cite{connors2019stability}.
} 

For concurrent coupling (CC), 
the two models run simultaneously
by two-way coupling at each colocated time level. For example, the
 atmospheric model and ocean model are two-way coupled at every $\dt_{ocn} = 2\dt_{atm}$ in Figure \figref{coupling-diagram}(b).
For sequential coupling (SC), we run one model first and then advance the other model.
For example, 
as can be seen in Figure \figref{coupling-diagram}(c),
after two-way coupling has occurred at the beginning of each colocated time level, 
we advance the ocean model first while freezing the interface data. 
Then we update the atmospheric model by providing the updated solution from the ocean model at each stage.

Obviously, in both the concurrent and the sequential coupling approaches, 
the heat and the horizontal momentum fluxes are not always continuous at every stage in general. For example,  
at the  $i(>1)$th stage, 
\begin{align*}
\hat{\sigma}_{1xz}^{(i)} = b_u (u^{(i)}_2 - u^{n}_1) &\ne b_u (u^{n}_2 - u^{(i)}_1) = \hat{\sigma}_{2xz}^{(i)}, \\
\hat{\Pi}_{1z}^{(i)} = - b_T (T^{(i)}_2 - T^{n}_1) &\ne -b_T (T^{n}_2 - T^{(i)}_1) = \hat{\Pi}_{2z}^{(i)}  
\end{align*}
for concurrent coupling, and
\begin{align*}
\hat{\sigma}_{1xz}^{(i)} = b_u (u^{(i)}_1 - u^{(i)}_2) &\ne b_u (u^{n}_1 - u^{(i)}_2) = \hat{\sigma}_{2xz}^{(i)}, \\
\hat{\Pi}_{1z}^{(i)} = - b_T (T^{(i)}_1 - T^{(i)}_2) &\ne -b_T (T^{n}_1 - T^{(i)}_2) = \hat{\Pi}_{2z}^{(i)} 
\end{align*}
for sequential coupling. 
The interface data is not always available during the time level.
interval. Consequently, compared with tight coupling, 
 loose coupling leads to a degraded rate of convergence in time 
but can increase the solution accuracy when dynamics rapidly changes because of substeps (see ARK3 in Table \tabref{khi-re5000-imex-coupling}).

\begin{figure}[h!t!b!]
  \centering
  \centering
  \subfigure[Tight coupling (TC)]{
    \includegraphics[trim=2.5cm 6.6cm 3.2cm 6.0cm,clip=true,
      width=0.82\textwidth]{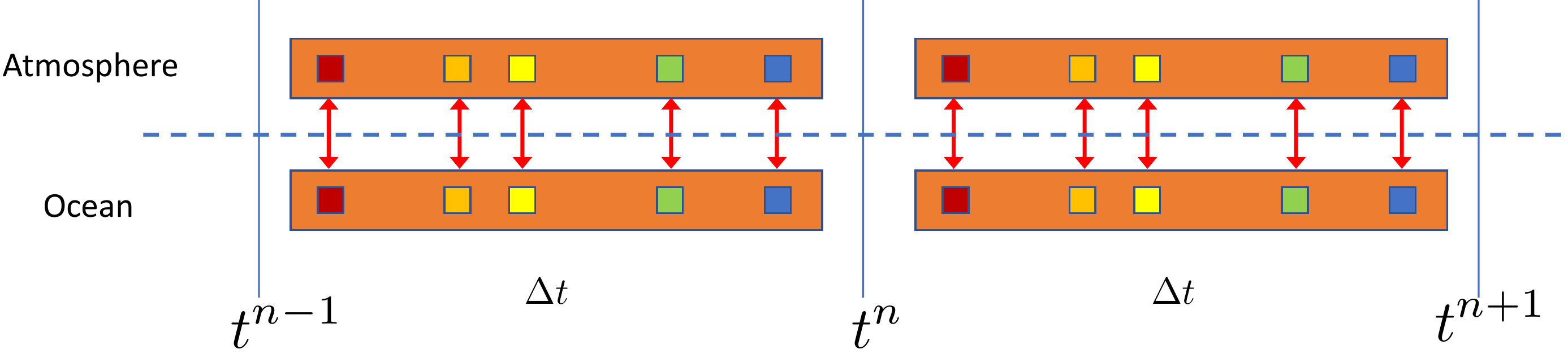}  
  }
  \subfigure[Concurrent coupling with two substeps (CC2)]{
    \includegraphics[trim=2.5cm 6.0cm 3.2cm 4.7cm,clip=true,
      width=0.82\textwidth]{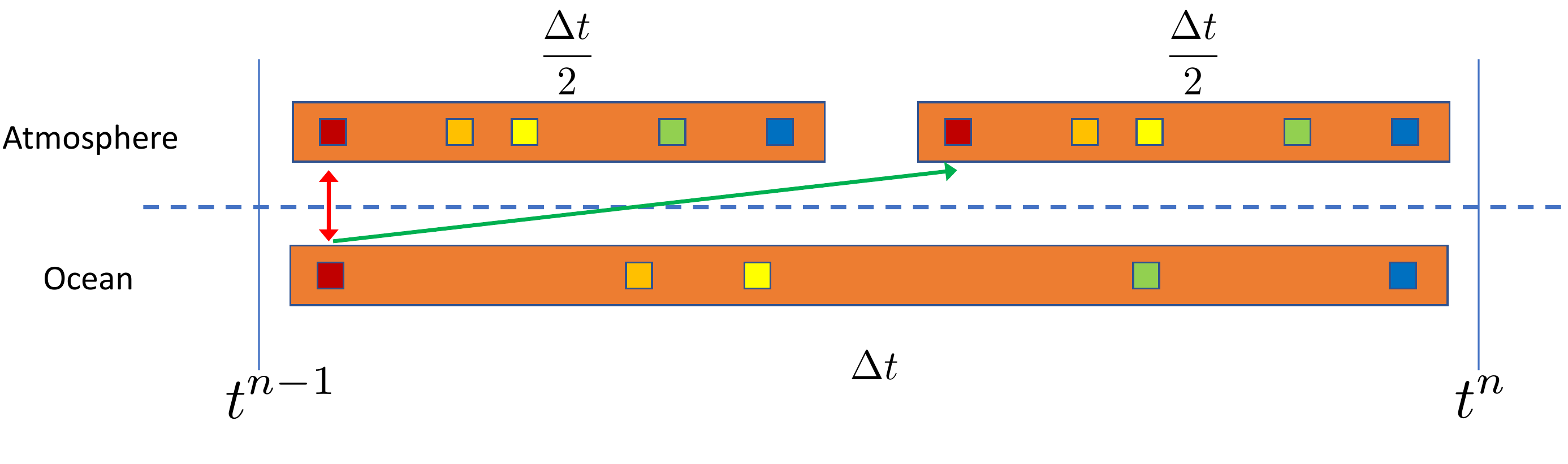}  
  }
  \subfigure[Sequential coupling with two substeps (SC2)]{
    \includegraphics[trim=2.5cm 6.6cm 3.2cm 4.7cm,clip=true,
      width=0.82\textwidth]{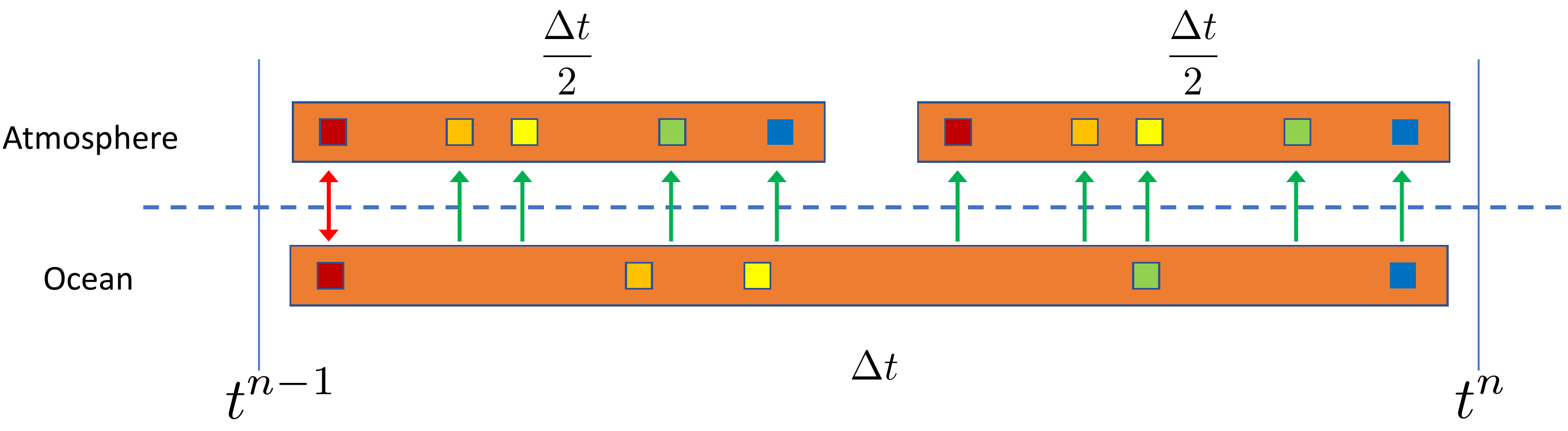}  
  }
  \caption{Coupling diagrams: (a) tight coupling, 
  (b) concurrent with two substeps (CC2), and 
  (c) sequential coupling with two substeps (SC2).
  Five-time stages are illustrated by 
  small (\sqbox{cred}, \sqbox{corange}, \sqbox{cyellow}, \sqbox{cgreen}, and \sqbox{cblue}) boxes.
  The single (green) and the double (red) arrows describe one-way and two-way couplings, respectively. 
  For tight coupling, two-way coupling occurs at every stage, 
  whereas for loose coupling, 
  two-way coupling is only performed 
  at each colocated time level.
  After the colocated time level, 
  for concurrent coupling, 
  the two models run simultaneously without updating interface data; and 
  for sequential coupling, 
  the ocean model run first while freezing the interface data, and then 
  the atmospheric model is marched by providing the updated solution from the ocean model at each stage.
}

  \figlab{coupling-diagram}
\end{figure}


Without loss of generality, 
in this study we refer to the model on $\Omega_1$ as the ocean model 
and to the model on $\Omega_2$ as the atmospheric model for convenience.
For TC, we treat the ocean implicitly and the atmosphere explicitly 
within the IMEX coupling framework. 
For LC, we advance the ocean model using IMEX methods 
but march the atmospheric model explicitly using the explicit part of the IMEX methods. The interface is treated explicitly. 
Detailed algorithms are shown in the following subsections.

\subsection{IMEX-RK tight coupling  methods}
\seclab{IMEX-RK-coupling}

We discretize \eqnref{cns-gov-nondimension} by using the FV method in \eqnref{euler-fv-cartesian-2d},
which yields the following semi-discretized coupled systems,
\begin{subequations}
\begin{align}
  \DD{\qb_1}{t} & = \Rb_1(\qb_1,\qb_2) \text{ on } \Omega_1, \\  
  \DD{\qb_2}{t} & = \Rb_2(\qb_1,\qb_2) \text{ on } \Omega_2,
\end{align}
\eqnlab{semi-discretized-coupled-model}
\end{subequations}
with $\qb_m=(\qb_{m_1},\qb_{m_2},\cdots,\qb_{m_\Ne{m} })^T$ and 
$\Rb_m=(\Rb_{m_1},\Rb_{m_2},\cdots,\Rb_{m_\Ne{m} })^T$ for $m=1,2$.
Here, we omit the overline (which stands for cell-averaged quantity) for brevity. 

Under the assumption that 
the ocean model is stiffer than the atmospheric model,
we cast \eqnref{semi-discretized-coupled-model} into 
$s$-stage IMEX-RK methods \cite{ascher1997implicit,pareschi2005implicit,Kennedy2003additive, giraldo2013implicit},  
\begin{subequations}
   \eqnlab{IMEXRK}
  \begin{align}
   \eqnlab{IMEXRK-qi}
     \Qb^{(i)} &= \Qb^n
       + \dtt\sum_{j=1}^{i-1}a_{ij} \fb_j
       + \dtt\sum_{j=1}^{i}\tilde{a}_{ij} \gb_j,\quad i=1,\hdots,s,\\
   \eqnlab{IMEXRK-qn}
     \qb^{n+1} &= \qb^n
       + \dtt\sum_{i=1}^{s}b_{i} \fb_i
       + \dtt\sum_{i=1}^{s}\tilde{b}_{i} \gb_i,
  \end{align}
\end{subequations}
where $\qb=(\qb_1,\qb_2)^T$, 
$\Qb=(\Qb_1,\Qb_2)^T$, 
$\fb=(0,\Rb_2)^T$,
$\gb=(\Rb_1,0)^T$,
$\fb_i = \fb\LRp{t^n+c_i\dtt, \Qb^{(i)}}$, 
$\gb_i = \gb\LRp{t^n+\tilde{c}_i\dtt,\Qb^{(i)}}$,
$\qb^n = \qb(t^n)$;
$\Qb^{(i)}$ is the $i$th intermediate state;
and $\dtt$ is the timestep size.
The scalar coefficients $a_{ij}$, $\tila_{ij}$, $b_i$, $\tilb_i$,
$c_i$, and $\tilc_i$ determine all the properties of a given IMEX-RK
scheme.

The intermediate state $\Qb^{(i)}$ in \eqnref{IMEXRK-qi} is 
  \begin{align}
   \eqnlab{IMEXRK-qi-AO}
     \begin{pmatrix}
       \Qb_1^{(i)} \\
       \Qb_2^{(i)} 
     \end{pmatrix}
     -\dtt \tilde{a}_{ii}
     \begin{pmatrix}
      \Rb_1 (\Qb^{(i)}) \\
       0
     \end{pmatrix}
       &= 
     \begin{pmatrix}
       \qb_1^n \\
       \qb_2^n  
     \end{pmatrix}
       + 
     \dtt \sum_{j=1}^{i-1}
     \begin{pmatrix}
       a_{ij} \Rb_{1j}\\
       \tilde{a}_{ij} \Rb_{2j}
     \end{pmatrix}
     =: 
     \begin{pmatrix}
       \check{\Qb}_1\\
       \check{\Qb}_2
     \end{pmatrix}
  \end{align}
  with $\Rb_{ij}:=\Rb_i(\Qb^{(j)})$. 
Thanks to the semi-implicit structure,\footnote{
This can be seen as a block Gauss--Seidel structure.} we explicitly update $\Qb_2^{(i)}$ from the second row in \eqnref{IMEXRK-qi-AO}
and then solve for $\Qb_1^{(i)}$ implicitly. The next timestep solution is obtained by \eqnref{IMEXRK-qn} after solving for all the intermediate stages: 
\begin{subequations}
\begin{align}
  \eqnlab{imex-rk-update-a}
  {\qb}_1^{n+1} &= \qb_1^n + \dtt\sum_{i=1}^{s} \tilde{b}_{i} \Rb_{1i}, \\
  \eqnlab{imex-rk-update-b}
  {\qb}_2^{n+1} &= \qb_2^n + \dtt\sum_{i=1}^{s} b_{i} \Rb_{2i}. 
\end{align} 
\eqnlab{imex-rk-update}
\end{subequations}

These steps are summarized in Algorithm \algref{IMEXRK-Coupling}.
Note that at the time of computing the right-hand sides, $\Rb_1$ and $\Rb_2$, 
both the solution $\Qb_1$ and the solution $\Qb_2$ are available at the same time; hence,
 the bulk fluxes in \eqnref{bulk-flux} are conservative at each stage. 


\begin{algorithm}[h!t!b!]
  \begin{algorithmic}[1]
    \ENSURE Given solution state $\qb^n$, compute its next solution state $\qb^{n+1}$ 
    under the assumption that $\Rb_1$ is stiffer than $\Rb_2$. 
    Let $\Rb_{1i}:=\Rb_1(\Qb_1^{(i)},\Qb_2^{(i)})$
    and $\Rb_{2i}:=\Rb_2(\Qb_1^{(i)},\Qb_2^{(i)})$.
      \FOR{$i=1$ to $s$}
        \IF {$\tilde{a}_{ii} = 0$}
          \STATE $\Qb_1^{(i)} \gets \qb_1^n$
          \STATE $\Qb_2^{(i)} \gets \qb_2^n$
        \ELSE
          \STATE ${\Qb}_2^{(i)} \gets \qb_2^n + \dtt\sum_{j=1}^{i-1} a_{ij} \Rb_{2j}$
          \STATE $\check{\Qb}_1^{(i)} \gets \qb_1^n + \dtt\sum_{j=1}^{i-1} \tilde{a}_{ij} \Rb_{1j}$
          \STATE Solve for $\Qb_1^{(i)}$ according to \eqref{eq:IMEXRK-qi-AO} 
        \ENDIF
        \STATE Send/Receive interface information from $\Omega_1$ to $\Omega_2$, vice versa.
        \STATE Compute stage right-hand sides, $\Rb_{1i}$ and $\Rb_{2i}$ 
      \ENDFOR
      \STATE ${\qb}_1^{n+1} \gets \qb_1^n + \dtt\sum_{i=1}^{s} \tilde{b}_{i} \Rb_{1i}$
      \STATE ${\qb}_2^{n+1} \gets \qb_2^n + \dtt\sum_{i=1}^{s} {b}_{i} \Rb_{2i}$
  \end{algorithmic}
  \caption{Partitioned IMEX-RK Coupling}
  \alglab{IMEXRK-Coupling}
\end{algorithm}



To avoid nonlinear solves for $\Qb_1$, we can linearize $\Rb_1$ and use the same IMEX methods \cite{giraldo2013implicit,kang2019imex} to solve only linear systems at each stage. We do so by choosing $\Lb$, a linear operator containing stiff components of $\Rb_1$, and then constructing a term by subtracting $\Lb$ from $\Rb_1$ in the hope that this term is nonstiff, that is,
$\Nb:= \Rb_1 - \Lb$.  
By taking 
$\fb=(\Nb,\Rb_2)^T$ and
$\gb=(\Lb,0)^T$, 
the intermediate state $\Qb^{(i)}$ in \eqnref{IMEXRK-qi} becomes 
\begin{align}
 \eqnlab{IMEXRK-qi-AO-imex}
   \begin{pmatrix}
     \Qb_1^{(i)} \\
     \Qb_2^{(i)} 
   \end{pmatrix}
   -\dtt \tilde{a}_{ii}
   \begin{pmatrix}
     \Lb (\Qb^{(i)}) \\
     0  
   \end{pmatrix}
     &= 
   \begin{pmatrix}
     \qb_1^n \\
     \qb_2^n  
   \end{pmatrix}
     + 
   \dtt \sum_{j=1}^{i-1}
   \begin{pmatrix}
     a_{ij}\Nb_{j} + \tilde{a}_{ij} \Lb_{j} \\
     a_{ij} \Rb_{2j}
   \end{pmatrix}
\end{align}
with $\Lb_{j}:=\Lb(\Qb^{(j)})$ and $\Nb_{j}:=\Rb_{1j} - \Lb_j$. 
This form requires only one linear solve to update $\Qb_1^{(i)}$.
We update the next timestep solution as follows:  
\begin{subequations}
\begin{align}
  \eqnlab{imex-rk-imex-update-a}
  {\qb}_1^{n+1} &= \qb_1^n + \dtt\sum_{i=1}^{s} \LRp{b_{i} \Nb_{i} + \tilde{b}_{i} \Lb_{i}},\\
  \eqnlab{imex-rk-imex-update-b}  
  {\qb}_2^{n+1} &= \qb_2^n + \dtt\sum_{i=1}^{s} b_{i} \Rb_{2i}.
\end{align} 
\eqnlab{imex-rk-imex-update}
\end{subequations}
The algorithm is summarized in Algorithm \algref{IMEXRK-Coupling-imex}.

\begin{algorithm}[h!t!b!]
  \begin{algorithmic}[1]
    \ENSURE Given solution state $\qb^n$, compute its next solution state $\qb^{n+1}$ 
    under the assumption that $\Rb_1$ is stiffer than $\Rb_2$,
     and $\Lb$ contains stiff components of $\Rb_1$. Let $\Rb_{1i}:=\Rb_1(\Qb_1^{(i)},\Qb_2^{(i)})$,
        $\Rb_{2i}:=\Rb_2(\Qb_1^{(i)},\Qb_2^{(i)})$,
        $\Lb_{i}:=\Lb(\Qb_1^{(i)},\Qb_2^{(i)})$, and
        $\Nb_i = \Rb_{1i} - \Lb_i$.
      \FOR{$i=1$ to $s$}
        \IF {$\tilde{a}_{ii} = 0$}
          \STATE $\Qb_1^{(i)} \gets \qb_1^n$
          \STATE $\Qb_2^{(i)} \gets \qb_2^n$
        \ELSE
          \STATE ${\Qb}_2^{(i)} \gets \qb_2^n + \dtt\sum_{j=1}^{i-1} a_{ij} \Rb_{2j}$
          \STATE $\check{\Qb}_1^{(i)} \gets \qb_1^n 
                + \dtt\sum_{j=1}^{i-1} \LRp{a_{ij} \Nb_j + \tilde{a}_{ij} \Lb_j }$
          \STATE Linear solve for $\Qb_1^{(i)}$ 
        \ENDIF
        \STATE Send/Receive interface information from $\Omega_1$ to $\Omega_2$, vice versa.
        \STATE Compute stage right-hand-sides, $\Rb_{2i}$, $\Nb_{i}$ and $\Lb_i$ 
      \ENDFOR
      \STATE ${\qb}_1^{n+1} \gets \qb_1^n + \dtt\sum_{i=1}^{s} \LRp{b_i \Nb_i + \tilde{b}_{i} \Lb_i }$
      \STATE ${\qb}_2^{n+1} \gets \qb_2^n + \dtt\sum_{i=1}^{s} b_{i} \Rb_{2i}$
  \end{algorithmic}
  \caption{Partitioned IMEX-RK Coupling with Linearized Stiff Flux }
  \alglab{IMEXRK-Coupling-imex}
\end{algorithm}

\subsubsection{The linear operator $\Lb$}

We construct a stiff linear operator $\Lb$ so that the numerical stiffness of  $\Nb(=\Rb_1 - \Lb)$ is relaxed.
Since the Jacobian of the right-hand side $\dd{\Rb_1}{\qb}$ contains the stiffness, 
a linear operator can be chosen as an approximation of the Jacobian to some degree. For example, with a two-dimensional uniform mesh, 
we can define the linear operator for $(i,j)$ element by 
\begin{align}
  \eqnlab{linear-rhs}
  \Lb (\tilde{\qb},\qb)_{(i,j)}:= 
  - \frac{1}{\dx}\LRp{\Fbs_{L,(\iphalf,j)}  - \Fbs_{L,(\imhalf,j)}}  
  - \frac{1}{\dz}\LRp{\Gbs_{L,(i,\jphalf)}  - \Gbs_{L,(i,\jmhalf)}},
\end{align}
where $\tilde{\qb}$ is the reference state and 
$ \Fbs_L := \dd{\Fbs}{\qb}\bigg|_{\tilde{\qb}} \qb $ and 
$ \Gbs_L := \dd{\Gbs}{\qb}\bigg|_{\tilde{\qb}} \qb $ 
are $x$ and $z$ directional linear fluxes.

If the inviscid part is stiffer than the viscous part, then  the linearized viscous part can be dropped from \eqnref{linear-rhs},
 leading to 
\begin{align}
  \eqnlab{linear-inviscid}
  \Lb^I (\tilde{\qb},\qb)_{(i,j)}:= 
  - \frac{1}{\dx}\LRp{\Fb^{I^*}_{L,(\iphalf,j)}  - \Fb^{I^*}_{L,(\imhalf,j)}}  
  - \frac{1}{\dz}\LRp{\Gb^{I^*}_{L,(i,\jphalf)}  - \Gb^{I^*}_{L,(i,\jmhalf)}} 
\end{align}
with 
$ \Fb^{I^*}_L := \dd{\Fb^{I^*}}{\qb}\bigg|_{\tilde{\qb}} \qb $ and 
$ \Gb^{I^*}_L := \dd{\Gb^{I^*}}{\qb}\bigg|_{\tilde{\qb}} \qb $.

From a computational point of view,
both \eqnref{linear-rhs} and \eqnref{linear-inviscid} 
can be expensive 
without a proper preconditioner 
because of their large number of degrees of freedom and global communication in $(x,z)$ domains.
To reduce the computational cost, 
we can adapt the dimensional splitting approach, 
used in horizontally explicit and vertically implicit (HEVI) methods \cite{giraldo2013implicit,weller2013runge,gardner2018implicit}.  
For HEVI, we define the linear operator by 
\begin{align}
  \eqnlab{linear-inviscid-hevi}
  \Lb^{z} (\tilde{\qb},\qb)_{(i,j)}:= 
  - \frac{1}{\dz}\LRp{\Gb^{I^*}_{L,(i,\jphalf)}  - \Gb^{I^*}_{L,(i,\jmhalf)}} 
\end{align}
with only the linear inviscid flux in the  vertical direction.

\subsection{IMEX-RK loose coupling  methods}
\seclab{IMEX-RK-loose-coupling}

Now we consider a practical situation where there is a fixed time-scale difference between the two domains. For example, the timestep size of the ocean model, $\dt_1$, is $N_s$ times larger than that of the atmosphere, $\dt_2$, namely, 
$\dt_1 = N_s \dt_2$. 
We let $\qb^{n+1}_1$ and $\qb^{n+\frac{k}{N_s}}_2$ be numerical approximations of $\qb_1(t^n+\dt_1)$
 and $\qb_2(t^n+k\dt_2)$, respectively. 

We advance the ocean model by using $s$-stage IMEX-RK schemes, which can also provide 
the dense output formulas of order $p^*$ \cite{Kennedy2003additive}, 
\begin{align*}
\Qb^*(t^n+\theta \dt) := \qb^n + \dt \sum_{i=1}^{s} B^*_i(\theta) \Nb_i + \widehat{B}^*_i(\theta) \Lb_i.
\end{align*}
Here,  
%
$\Nb_i:= \Nb(t^n + c_i \dt,Q^{(i)})$, 
$\Lb_i:=\Lb(t^n + \tilde{c}_i \dt,Q^{(i)})$, 
$\theta \in [0,1]$, 
$B^*_i(\theta)= \sum_{j=1}^{p^*} b^*_{ij} \theta^j$, and $b^*_{ij}$ is a matrix of coefficients of size $s \times p^*$.
With the ARK schemes satisfying  $\widehat{B}^*_i(\theta)= B^*_i(\theta)$ and $\tilde{c}_i=c_i$, 
we have the time polynomial for ARK methods,\footnote{
  For example, ARK2 ($s=3$, $p^*=2$) has 
  \begin{align*}
  b^*=
  \begin{pmatrix}
    \frac{1}{\sqrt{2}}& -\frac{1}{2\sqrt{2}} \\
    \frac{1}{\sqrt{2}}& -\frac{1}{2\sqrt{2}}\\
    1-\sqrt{2} & \frac{1}{\sqrt{2}}
  \end{pmatrix}.
  \end{align*}
}\cite{giraldo2013implicit}
\begin{align}
  \eqnlab{DenseOutput-ARKt}
  \Qb^*(t^n+\theta \dt) := \qb^n + \dt \sum_{i=1}^{s} B^*_i(\theta) \Rb_i \,.
\end{align}

For sequential coupling (SC), 
we integrate the ocean model using $s$-stage IMEX-RK 
and the atmospheric model using s-stage RK (explicit part of IMEX-RK).
We construct the time polynomial \eqnref{DenseOutput-ARKt} and 
evaluate $\Qb^*$ at every RK stage during the substeps.
For example, with two substeps, $\Qb^*$ is interpolated 
at $t^n + (k-1+c_i)\dt_2$ for $k=1,2$ and $i=1,2\cdots, s$.
For concurrent coupling (CC), 
we simply take $\Qb^*(t^n+\theta \dt)=\qb^n$. 
The algorithm is summarized in Algorithm \algref{IMEXRK-Coupling-imex-sc}.

We call the sequential and the concurrent couplings with $N_s$ substeps by SC$\{N_s\}$ 
and CC$\{N_s\}$, respectively.  

%

\begin{algorithm}[h!t!b!]
  \begin{algorithmic}[1]
    \ENSURE Given solution state $\qb^n_1$ and $\qb^n_2$, 
      compute their next solution states $\qb^{n+1}_1$ and $\qb^{n+1}_2$ 
      under the assumption of $\dt_1 = N_s\dt_{2}$  
      and $\Lb$ containing stiff components of $\Rb_1$. 
      Let $\Rb_{1i}:=\Rb_1(\Qb_1^{(i)},\qb^n_2)$,
        $\Rb_{2i}:=\Rb_2(\Qb_1^{*,(i)},\Qb_2^{(i)})$,
        $\Qb_1^{*,(i)}:= \Qb_1^{*}(t^n + (k-1+c_i)\dt_2)$,
        $\Lb_{i}:=\Lb(\Qb_1^{(i)},\qb^n_2)$, and
        $\Nb_i = \Rb_{1i} - \Lb_i$.
      \STATE Exchange interface information from atmospheric model to ocean model
      \FOR{$i=1$ to $s$}
        \IF {$\tilde{a}_{ii} = 0$}
          \STATE $\Qb_1^{(i)} \gets \qb_1^n$
        \ELSE
          \STATE $\check{\Qb}_1^{(i)} \gets \qb_1^n 
                + \dtt_1\sum_{j=1}^{i-1} \LRp{a_{ij} \Nb_j + \tilde{a}_{ij} \Lb_j }$
          \STATE Linear solve for $\Qb_1^{(i)}$ 
        \ENDIF        
        \STATE Compute stage right-hand sides, i.e., $\Nb_{i}$ and $\Lb_i$ 
      \ENDFOR
      \STATE ${\qb}_1^{n+1} \gets \qb_1^n + \dtt_1\sum_{i=1}^{s} \LRp{b_i \Nb_i + \tilde{b}_{i} \Lb_i }$
      \FOR{$k=1$ to $N_s$}
        \FOR{$i=1$ to $s$}
          \IF {$i=1$}
            \STATE $\Qb_2^{(i)} \gets \qb_2^{n+\frac{k}{N_s}}$
          \ELSE
            \STATE $\Qb_2^{(i)} \gets \qb_2^n + \dtt_2 \sum_{j=1}^{i-1} a_{ij} \Rb_{2i}$
          \ENDIF
          \STATE Interpolate $\Qb_1^{*,(i)}$
          \STATE Compute stage right-hand sides, $\Rb_{2i}$ 
        \ENDFOR
        \STATE ${\qb}_2^{n+\frac{k}{N_s}} \gets \qb_2^{n+\frac{k-1}{N_s}} + \dtt\sum_{i=1}^{s} b_{i} \Rb_{2i}$
      \ENDFOR
      
  \end{algorithmic}
  \caption{IMEX-RK Loose Coupling Methods}
  \alglab{IMEXRK-Coupling-imex-sc}
\end{algorithm}

\subsection{Mass-conserving IMEX-RK coupling}
  
Conservation of mass is the most fundamental conservation property 
because it is related to many other conservation properties such as tracer and energy.
Any imperfection in the conservation of mass will affect long-time integration,
which can generate a superficial pressure field via the equation of states 
and eventually lead to unwanted modes.~\cite{thuburn2008some} 
Thus, 
we examine the mass conservation property in the IMEX-RK coupling framework. 

We define the total mass in $\Omega=\Omega_1 \cup \Omega_2$ by
\begin{align*}
  \mass 
  &= \int_{\Omega_1} \rho d\Omega + \int_{\Omega_2} \rho d\Omega 
  = \sum_{\ell=1}^{\Ne{1}} \int_{\K_{1_\ell}} \rho d\K
  + \sum_{\ell=1}^{\Ne{2}} \int_{\K_{2_\ell}} \rho d\K \\
  & = \sum_{\ell=1}^{\Ne{1}} \bar{\rho}_{1_\ell} \snor{\K_{1_\ell}} 
        + \sum_{\ell=1}^{\Ne{2}} \bar{\rho}_{2_\ell} \snor{\K_{2_\ell}} 
    = \mass_1 + \mass_2 ,
\end{align*}
where $\mass_m = \sum_{\ell=1}^{\Ne{m}} \bar{\rho}_{m_\ell} \snor{\K_{m_\ell}}$ ; 
$\bar{\rho}_{m_\ell} = \snor{\K_{m_\ell}}^{-1} \int_{\K_{m_\ell}} \rho d\K$ 
is the mean density on $\K_{m_\ell}$.

\begin{proposition}
\theolab{imex-mass-conservation}
The IMEX-RK coupling methods with $b_i = \tilde{b}_i$ in \eqnref{imex-rk-update} 
and \eqnref{imex-rk-imex-update}
are mass conservative for system \eqnref{cns-gov} with the FV scheme, the numerical flux \eqnref{euler-fv-cartesian-2d},
and the interface condition \eqnref{rigid-lid-cond}.
\end{proposition}

\begin{proof}
  Without loss of generality, 
  we assume the $x$-periodic boundary condition and 
  isothermal wall condition at the top and bottom boundary.
  For simplicity, we focus on two-dimensional Cartesian coordinates. 
  We denote the right-hand side of the mass conservation equation in \eqnref{euler-fv-cartesian-2d} 
  for $\ell=(i,j)$ element on $\Omega_m$ by 
  $(\Rb_m)_{\rho,(i,j)} = (\Rb_m)^x_{\rho,(i,j)} + (\Rb_m)^z_{\rho,(i,j)}$, where
  \begin{align*}
    (\Rb_m)^x_{\rho,(i,j)} &= - \frac{1}{\dx_m}\LRp{(\rho u_m)^*_{(\iphalf,j)}  - (\rho u_m)^*_{(\imhalf,j)}},\\
    (\Rb_m)^z_{\rho,(i,j)} &= - \frac{1}{\dz_m}\LRp{(\rho w_m)^*_{(i,\jphalf)}  - (\rho w_m)^*_{(i,\jmhalf)}}.
  \end{align*}
  Summing $(\Rb_m)_{\rho,(i,j)}$ over all elements gives
  \begin{multline*}
  \sum_{i=1}^{\Nxe{m}}\sum_{j=1}^{\Nze{m}} 
      \LRp{(\Rb_m)^x_{\rho,(i,j)} + (\Rb_m)^z_{\rho,(i,j)}} \\
  = - \frac{1}{\dx_m} \sum_{j=1}^{\Nze{m}}  \LRp{ \rhouAn{j,\Nxe{m}+\half}{*,(k)} - \rhouAn{j,\half}{*,(k)} } \\
    - \frac{1}{\dz_m} \sum_{i=1}^{\Nxe{m}}  \LRp{ \rhowAn{\Nze{m}+\half,i}{*,(k)} - \rhowAn{\half,i}{*,(k)} } = 0  
  \end{multline*}
  due to $x$-periodicity and the wall boundary condition.


Thus, by taking the dot products of 
$\LRs{1}^T = (1, 1, \cdots, 1) \in \R^{\Ne{1}}$ 
 with \eqnref{imex-rk-update-a} and 
 $\LRs{1}^T = (1, 1, \cdots, 1) \in \R^{\Ne{2}}$ 
 with \eqnref{imex-rk-update-b},
  \begin{align*}
  \frac{1}{\dt} \LRp{ \mass_1(t^{n+1}) - \mass_1(t^n) }  
  &= \sum_{k=1}^{s} b_k \sum_{i=1}^{\Nxe{1}}\sum_{j=1}^{\Nze{1}} 
      \LRp{(\Rb_1)^x_{\rho,(i,j)} + (\Rb_1)^z_{\rho,(i,j)}} = 0,\\
  \frac{1}{\dt} \LRp{ \mass_2(t^{n+1}) - \mass_2(t^n) }  
  &= \sum_{k=1}^{s} \tilde{b}_k \sum_{i=1}^{\Nxe{2}}\sum_{j=1}^{\Nze{2}} 
      \LRp{(\Rb_2)^x_{\rho,(i,j)} + (\Rb_2)^z_{\rho,(i,j)}} = 0,
  \end{align*}
we have the mass conservation property on each subdomain; 
 hence, the total mass is conserved.

Similarly, by taking the dot products of 
$\LRs{1}^T = (1, 1, \cdots, 1) \in \R^{\Ne{1}}$ 
 with \eqnref{imex-rk-imex-update-a} and 
 $\LRs{1}^T = (1, 1, \cdots, 1) \in \R^{\Ne{2}}$ 
 with \eqnref{imex-rk-imex-update-b}, and 
 using the condition of $b_i=\tilde{b}_i$,  we have
  \begin{align*}
  \frac{1}{\dt} \LRp{ \mass_1(t^{n+1}) - \mass_1(t^n) }  
  &= \sum_{k=1}^{s} b_k \sum_{i=1}^{\Nxe{1}}\sum_{j=1}^{\Nze{1}} 
      \LRp{(\Rb_1)^x_{\rho,(i,j)} + (\Rb_1)^z_{\rho,(i,j)}} = 0,\\
  \frac{1}{\dt} \LRp{ \mass_2(t^{n+1}) - \mass_2(t^n) }  
  &= \sum_{k=1}^{s} b_k \sum_{i=1}^{\Nxe{2}}\sum_{j=1}^{\Nze{2}} 
      \LRp{(\Rb_2)^x_{\rho,(i,j)} + (\Rb_2)^z_{\rho,(i,j)}} = 0.
  \end{align*}
\end{proof}

Mass conservation is a critical component of the proposed strategy and will be checked numerically as well in the next section.
\footnote{
  Note that we investigate the mass conserving property for coupled CNS systems. 
  We conjecture that linear invariants will be preserved even with other couplings/ formulations 
  provided that the spatial discretization is conservative, 
  but a future investigation is needed. 
}
For IMEX numerical simultions,
 we limit ourselves to the 
 ARK2 method in \cite{giraldo2013implicit}, 
 and the ARK3 and the ARK4 in \cite{Kennedy2003additive}.

\section{Numerical results}
\seclab{NumericalResults}
 
We denote a single compressible Navier--Stokes model by CNS1 and 
a coupled compressible Navier--Stokes model by CNS2. 
We first perform spatial convergence studies for CNS1 using two examples:
density wave advection in Section \secref{sec-dwa} and Taylor--Green vortex in Section \secref{sec-tgv}.
Then, we conduct a temporal convergence study of IMEX coupling methods for CNS2 
using two moving vortices. 
We  compare the performance of tight and loose coupling methods  
through a wind-driven current and Kelvin--Helmholtz instability example.    

In the following examples, 
we measure the $L_2$ error of $q$ by 
$$ \left\Vert q  - q_r   \right\Vert := 
\LRp{\sum_{\ell=1}^{N_E} \snor{\K_{\ell}} (q_\ell-q_{\ell r})^2 }^\half $$,
where $q_r$ can be an exact solution or a reference solution.
We will take the fourth-order Runge--Kutta (RK4) methods with TC as a reference solution. 

We denote an IMEX coupling method by 
[ARK method]([linear operator type],[tight or loose coupling]).
For example, ARK2 ($\Lb^z$,TC) means an ARK2 tight coupling method whose implicit solver is HEVI, 
and ARK2 ($\Lb$,SC8) describes an ARK2 sequential coupling method with eight substeps in the explicit part 
whose implicit solver is IMEX.
  
\subsection{CNS1: Density wave advection}
\seclab{sec-dwa}
One-dimensional density wave advection is simulated \cite{ghosh2016semi} 
with zero viscosity. 
The initial density shape is advected with the constant velocity and pressure field.
The initial condition is given as 
\begin{align*}
\rho &= \rho_\infty + \half \sin(2\pi x) \cos(2\pi z),\\
u &= w = u_\infty,\\
\pres &= \pres_\infty
\end{align*}
with $(\rho_\infty,u_\infty,w_\infty,\pres_\infty)=(1,1,1,1)$.
The computational domain is taken as $\Omega = [0,1]^2$, 
and periodic boundary conditions are applied at all the boundaries. 

For spatial convergence studies,
we take uniform nested meshes with $h=\dx=\dz=1/\LRc{20,40,80,160,320,640,1280}$,
 and the RK4 time integration with $\dt=6.25\times 10^{-5}$. 
Since this example has an exact solution,
\begin{align*}
  \rho_e &= \rho_\infty + \half \sin(2\pi \tilde{x}) \cos(2\pi \tilde{z}),
\end{align*}
with $\tilde{x} = x - u_\infty t$ and $\tilde{z} = z - w_\infty t$,
we compute the $L_2$ errors at $t=0.1$
 and report them in Table \tabref{dwa-fixedDt-exact}.
 Here, $q_e$ can be $\rho_e$, $\rho \ub_e$, and $\rho E_e$. 
We observe the second-order convergence rate as expected.

 \begin{table}[t] 
  \caption{Spatial convergence for Denstiy wave advection: 
  we take uniform nested meshes with $h=\dx=\dz=1/\LRc{20,40,80,160,320,640,1280}$, 
 and the RK4 time integration with $\dt=6.25\times 10^{-5}$. 
  The error is measured by using an exact solution at $t=0.1$.} 
  \tablab{dwa-fixedDt-exact} 
  \begin{center} 
    \begin{tabular}{*{1}{c}|*{2}{c}|*{2}{c}|*{2}{c}} 
    \hline 
    \multirow{2}{*}{$h$}
    & \multicolumn{2}{c}{$ \left\Vert \rho   - \rho_e   \right\Vert  $} 
    & \multicolumn{2}{c}{$ \left\Vert \rho {\bf u} - \rho {\bf u}_e \right\Vert  $} 
    & \multicolumn{2}{c}{$ \left\Vert \rho E - \rho E_e \right\Vert  $} \tabularnewline 
    & error & order &error & order &error & order \tabularnewline 
    \hline\hline 
  1/   20&       3.142E-03 & $-$&       4.444E-03 & $-$&       3.142E-03 & $-$\tabularnewline
  1/   40&       5.609E-04 &    2.486&       7.933E-04 &    2.486&       5.609E-04 &    2.486\tabularnewline
  1/   80&       1.213E-04 &    2.209&       1.715E-04 &    2.209&       1.213E-04 &    2.209\tabularnewline
  1/  160&       2.900E-05 &    2.064&       4.102E-05 &    2.064&       2.900E-05 &    2.064\tabularnewline
  1/  320&       7.166E-06 &    2.017&       1.013E-05 &    2.017&       7.166E-06 &    2.017\tabularnewline
  1/  640&       1.786E-06 &    2.004&       2.526E-06 &    2.004&       1.786E-06 &    2.004\tabularnewline
  1/ 1280&       4.462E-07 &    2.001&       6.310E-07 &    2.001&       4.462E-07 &    2.001\tabularnewline
    \hline\hline 
    \end{tabular} 
  \end{center}     
\end{table}

\subsection{CNS1: Taylor--Green vortex}
\seclab{sec-tgv}

The Taylor--Green vortex flow \cite{taylor1937mechanism} is simulated by
using compressible Navier--Stokes equations at $Mach = 0.1$.
We solve the flows on a uniform grid of $\Omega = [0,1]^2$ with periodic boundary conditions.
The initial condition is
\begin{align*}
\rho &= \rho_\infty,\\
u &=   u_\infty \cos (2\pi x) \sin (2\pi z),\\
w &= - u_\infty \sin (2\pi x) \cos (2\pi z),\\
\pres &= \pres_\infty + \frac{\rho_\infty u_\infty^2}{4} \LRp{ \cos(4\pi x) + \cos (4\pi z) }
\end{align*}
with $(\rho_\infty,u_\infty,\pres_\infty)=(1,0.1,\gamma^{-1})$.
We take $\gamma=1.4$, $Pr=0.72$, and $Re=100$. 
  
Since this example does not have an exact solution,
we take the RK4 solution with $\dt = 1\times 10^{-6}$ and $h=1/640$ as a ground truth 
and measure the $L_2$ errors.
Table \tabref{tgv-fixedDt-reference} shows the second-order convergence rates for the conservative variables. 
 
\begin{table}[t] 
  \caption{Spatial convergence for TGV: 
  we take uniform nested meshes with $h=\dx=\dz=1/\LRc{20,40,80,160,320}$,
  and the RK4 time integration with $\dt=1\times 10^{-6}$. 
   The error is measured by using a reference solution   
  with $\dt = 1\times 10^{-6}$ and $h=1/640$ at $t=10^{-3}$.}
  \tablab{tgv-fixedDt-reference} 
  \begin{center} 
  \begin{tabular}{*{1}{c}|*{2}{c}|*{2}{c}|*{2}{c}} 
  \hline 
  \multirow{2}{*}{$h$}
  & \multicolumn{2}{c}{$ \left\Vert \rho   - \rho_r   \right\Vert  $} 
  & \multicolumn{2}{c}{$ \left\Vert \rho {\bf u} - \rho {\bf u}_r \right\Vert  $} 
  & \multicolumn{2}{c}{$ \left\Vert \rho E - \rho E_r \right\Vert  $} \tabularnewline 
  & error & order &error & order &error & order \tabularnewline 
  \hline\hline 
  1/   20&       2.261E-08 & $-$&       5.754E-04 & $-$&       1.077E-04 & $-$\tabularnewline
  1/   40&       6.558E-09 &    1.786&       1.442E-04 &    1.996&       2.728E-05 &    1.981\tabularnewline
  1/   80&       1.679E-09 &    1.965&       3.570E-05 &    2.014&       6.789E-06 &    2.007\tabularnewline
  1/  160&       4.034E-10 &    2.057&       8.509E-06 &    2.069&       1.622E-06 &    2.066\tabularnewline
  1/  320&       8.088E-11 &    2.319&       1.702E-06 &    2.321&       3.248E-07 &    2.320\tabularnewline
  \hline\hline 
  \end{tabular} 
  \end{center} 
\end{table}

\subsection{CNS2: Two moving vortices}
\seclab{cns2-tmv}
Ocean and atmosphere have different thermodynamic properties because one is a gas and other is a liquid.
\footnote{
For example, in the standard atmosphere, 
the standard pressure and temperature at sea surface level are given as 
$\pres=1013.25 \si{hPa}$ and $T=15 \si{\celsius}$.
With the universal gas constant for dry air, $R=287 \si{J.Kg^{-1}.K^{-1}}$,
we obtain the density $\rho = 1.226 \si{\kg.\meter^{-3}}$ from the equation of state for the ideal gas law. 
For the ocean surface, however, 
pressure, temperature, and density are given
$\pres=1013.25 \si{hPa}$, $T=24 \si{\celsius}$ and $\rho=1024 \si{kg.m^{-3}}$  
for salinity $S=35$ according to the example in \cite{young2010dynamic}.
}
For coupling under the rigid-lid assumption in \eqnref{rigid-lid-cond},
the jumps of velocity and temperature are of interest for specifying the bulk form \eqnref{bulk-flux}. 
To make the coupling problem simple, 
we consider two fluids as ideal gases with the same density 
but jumps of temperature and velocity across the interface. 
We assume that the temperature variation across the interface is within $10 \si{\percent}$ of the atmospheric temperature.
We also allow for a $10 \si{\percent}$ jump of the pressure at the material interface 
since an ocean model can have a different pressure from that of an atmospheric model at the interface.

We add two moving vortices  
to the uniform mean flow, apply periodic boundary conditions to $x$-direction, 
and impose isothermal boundary conditions at the top and the bottom walls. 
The top and bottom walls horizontally move with $u_{1_w}=0.1$ and $u_{2_w}=0.05$, respectively.
The whole domain is $\Omega =(-5,5) \times (-5,5)$
comprising two subdomains:   
$\Omega_1 =(-5,5) \times (-5,0)$ and 
$\Omega_2 =(-5,5) \times (0,5)$.  
The superposed flow is given as 
\begin{align*}
  \rho_m & = \LRp{ 1 - \frac{(\gamma-1)\beta^2_m}{8 \alpha\gamma \pi^2} e^{\alpha\LRp{1-r^2}} }^{\frac{1}{\gamma-1}},\\
  u_m & = u_{\infty_m} + \frac{\beta_m}{2\pi}\tilde{z}_m e^{\frac{\alpha}{2}\LRp{1-r^2} }, \\
  w_m & = - \frac{\beta_m}{2\pi}\tilde{x}_m e^{\frac{\alpha}{2}\LRp{1-r^2} }, \\
  \pres_m & = T_{\infty_m} \gamma^{-1} \rho_m^\gamma
\end{align*} 
with $\tilde{x}_m=x_m - x_{c_m}$ and $\tilde{z}_m=z_m - z_{c_m}$ 
for $\Omega_m$. 
We take 
$u_{\infty_1} = 0.05$,
$u_{\infty_2} = 0.1$,
$T_{\infty_1} = 1.1$,
$T_{\infty_2} = 1.0$,
$\alpha=2$,
$\beta_1=0.1$,
$\beta_2=0.5$,
$(x_{c_1},z_{c_1})=(0,-2.5)$, and 
$(x_{c_2},z_{c_2})=(0,2.5)$.
We choose the fluid parameters of $\gamma=1.4$, $Pr=0.72$, 
$\tilde{c}_p=(\gamma-1)^{-1}$,
and $\tilde{\mu}_m=5000^{-1}$.
\begin{figure}[h!t!b!]
  \centering
  \includegraphics[trim=2.6cm 1.5cm 5.5cm 2.2cm,clip=true,height=0.26\textwidth]{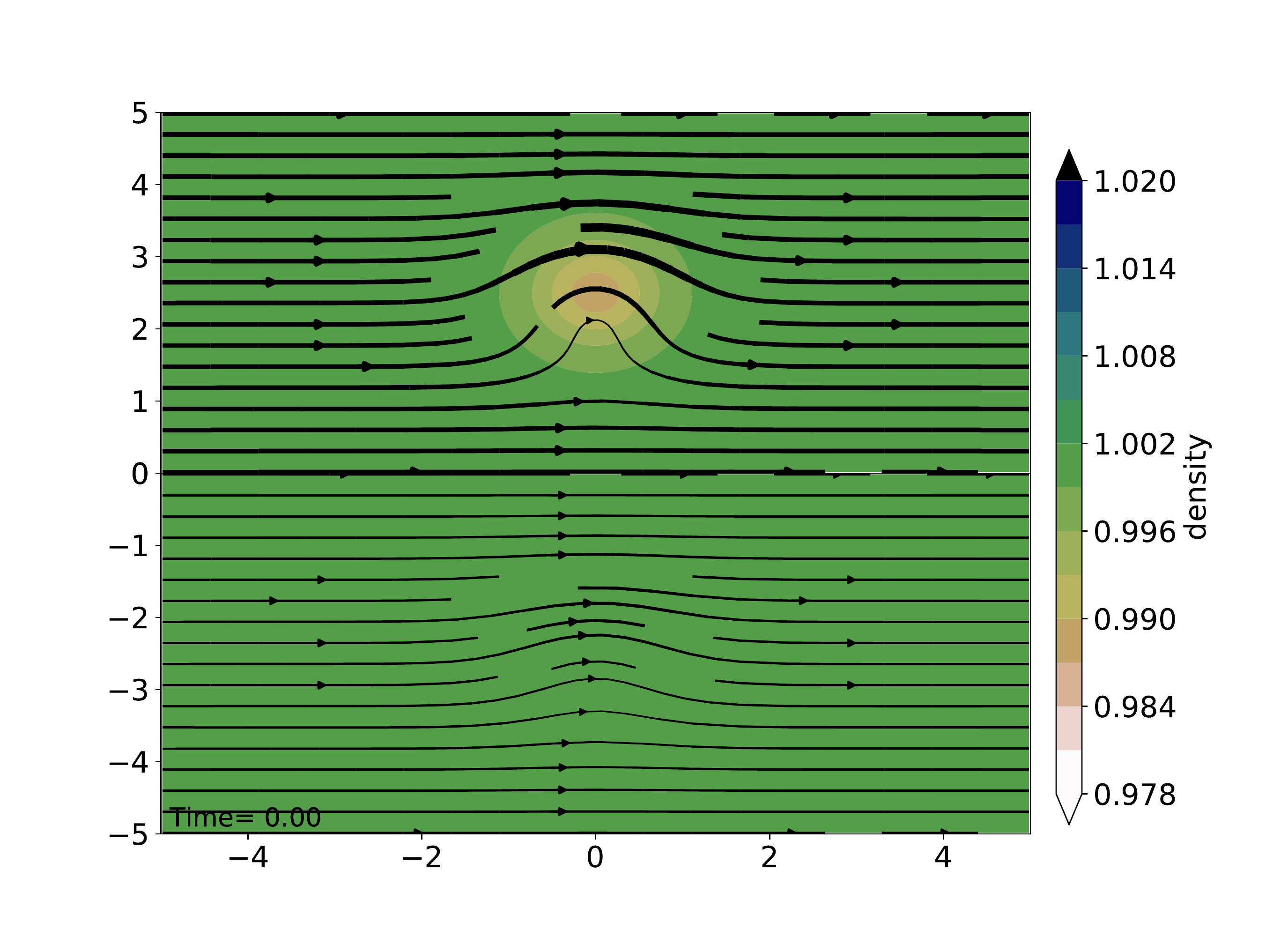}
  \includegraphics[trim=3.3cm 1.5cm 5.5cm 2.2cm,clip=true,height=0.26\textwidth]{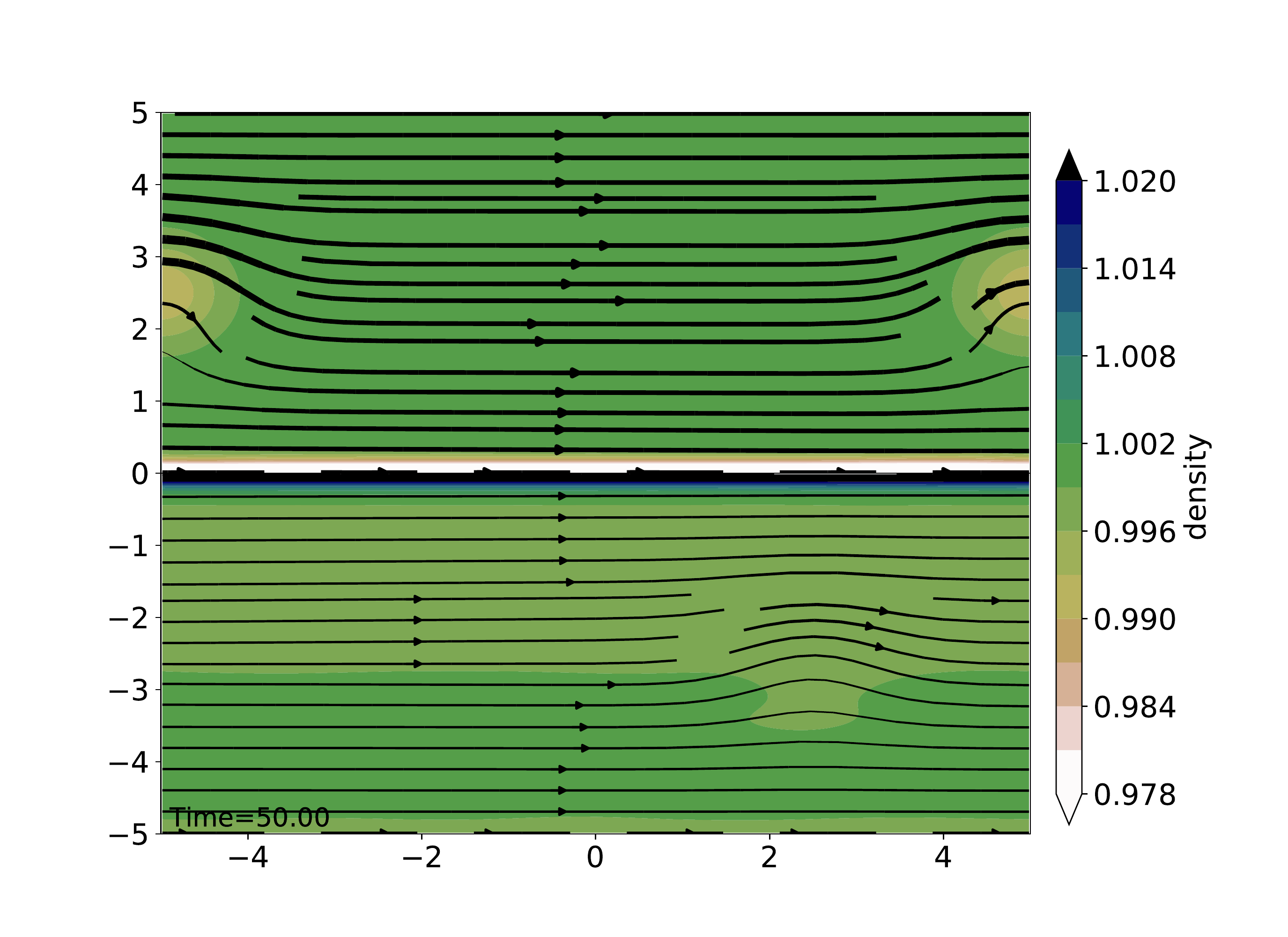}
  \includegraphics[trim=3.3cm 1.5cm 2.8cm 2.2cm,clip=true,height=0.26\textwidth]{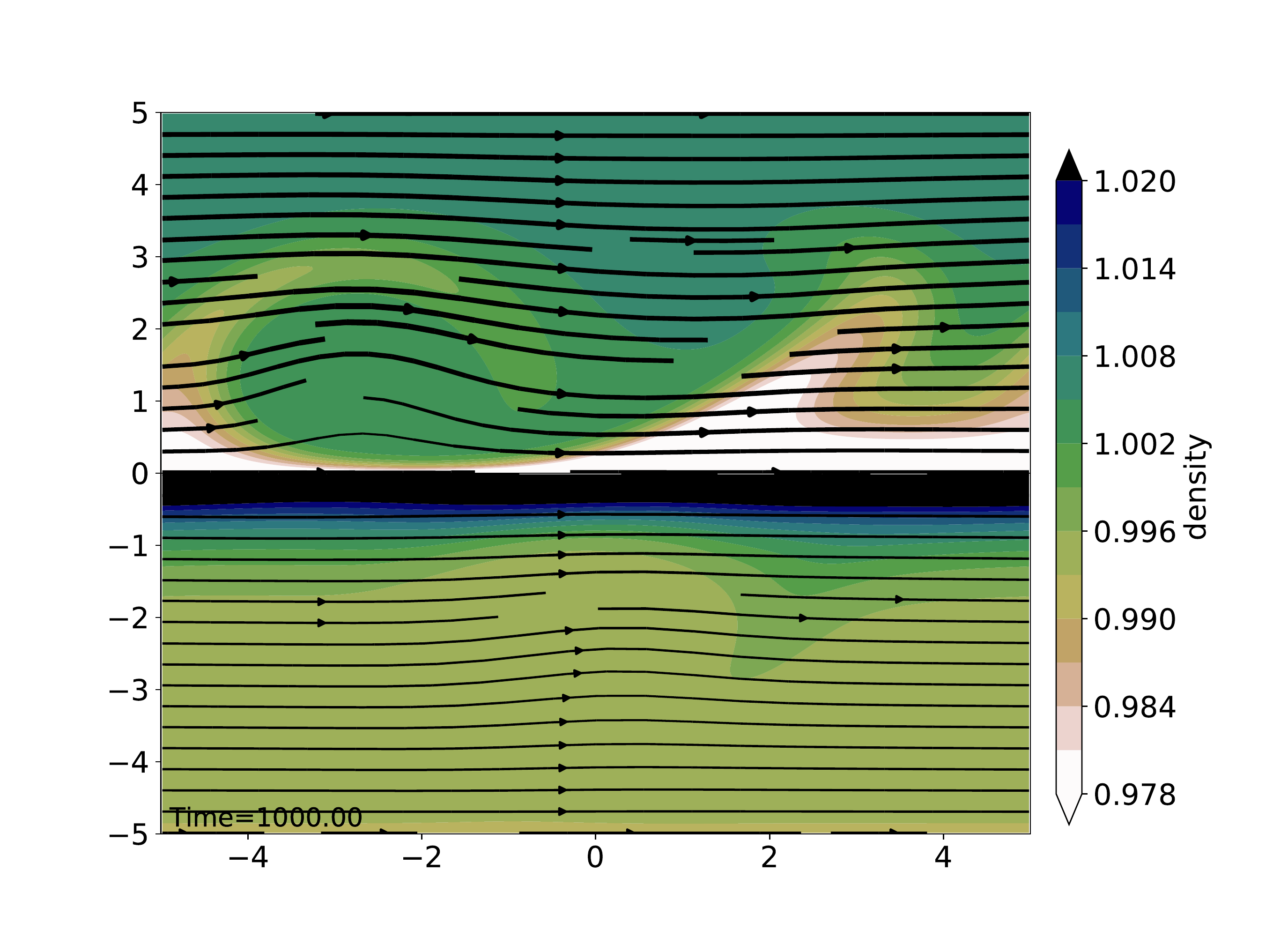}

  \caption{Snapshots of numerical solutions for two moving vortices at $t=\LRc{0,50,1000}$: 
  two vortices propagate to the positive $x$-direction with a mean velocity of $0.05$ on $\Omega_1$ 
  and $0.1$ on $\Omega_2$, respectively.}
  \figlab{ccns-tmv-ss}
\end{figure}

We integrate the coupled model by using RK4 (TC)
with $\dt=0.025$ for $t\in \LRs{0,1000}$
over an $\Nxe{_m}(320)\times\Nye{_m}(160)$ mesh.
Figure \figref{ccns-tmv-ss} shows the snapshots of the density field at $t=\LRc{0,50,1000}$
 with streamlines. 
 Initially, two vortices are located at the center of each domain (left panel). 
 Then they propagate to the positive $x$-direction with a different speed,
 so that the vortex on $\Omega_2$ moves two times faster than that on $\Omega_1$ (center panel).
 The velocity difference at the interface makes  
 for the horizontal momentum to be transferred from the top (atmosphere) to the bottom (ocean) models.
 At the same time, since the bottom (ocean) is $10 \si{\percent}$ hotter than the top (atmosphere), 
 heat transfer occurs from the bottom (ocean) to the top (atmosphere).
 Given the heat and horizontal momentum fluxes at the interface, 
 we estimate the wall temperature and velocity for atmosphere and ocean models at the interface, respectively. 
 The interface at the top (atmosphere) is warmed, but the bottom (ocean) counterpart is cooled. 
 This situation leads to the sharp gradient of density at the interface.
 As time passes, the vortices get diffused because of viscosity 
 and roll the fluids near the interface (right panel).

\subsubsection{IMEX tight coupling methods}
We now perform a temporal convergence study for IMEX coupling methods 
to demonstrate their stability (with $Cr > 1$\footnote{
  We define the Courant number $Cr:=a+\norm{\ub} = a + \sqrt{u^2+w^2}$.
}) and accuracy.
Since both models have similar acoustic wave speed,\footnote{
  The difference in the acoustic waves is about $0.05$, 
  e.g., $a_{\infty_2} - a_{\infty_1} = \sqrt{1.1} - \sqrt{1}$. 
} it is difficult to relax the scale-separable stiffness with IMEX tight coupling methods. 
Instead, we add $z$-directional geometric stiffness to $\Omega_1$.
The computational domain is discretized with $80 \times 800$ on $\Omega_1$ and $80 \times 80$ on $\Omega_2$. 
We treat $\Omega_1$ implicitly using the HEVI approach,\footnote{
Tighter tolerance is required to avoid solver errors that affect the temporal asymptotic analysis.
  We take \texttt{1e-4} 
  tolerance for the Krylov subspace methods. 
}, whereas  we treat $\Omega_2$ explicitly so that  
the timestep size is restricted by the acoustic wave speed in $\Omega_2$.

\begin{figure}[h!t!b!]
  \centering
    \includegraphics[trim=0.0cm 0.0cm 0.0cm 0.0cm,
      clip=true,width=0.8\textwidth]{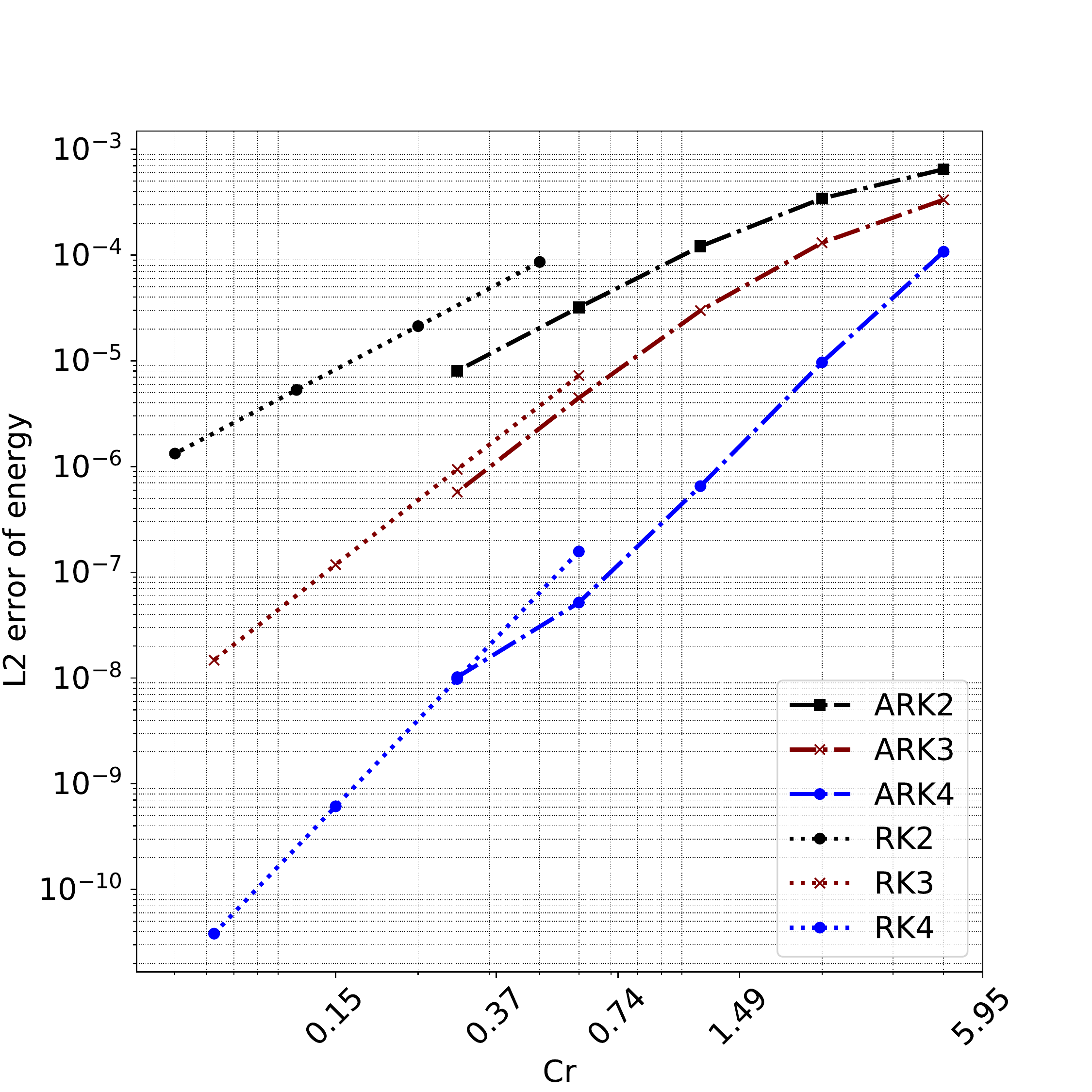}
 
  \caption{Two moving vortices: temporal convergence study for IMEX tight coupling methods.
  We take $\texttt{1e-4}$  tolerance for the Krylov subspace solver.}
  \figlab{tmv-tconv-imex-tc}
\end{figure}

We perform numerical simulations for IMEX tight coupling methods and 
measure the $L_2$ error at $t=2$ with the RK4 solution of $\dt=5\times 10^{-4}$.
For comparison, we also conduct numerical simulations for RK tight coupling methods. 
The results are summarized in Table \tabref{tmv-tempJump-tconv-reference} 
and Figure \figref{tmv-tconv-imex-tc}. 
In general, the second-, third-, and fourth-order convergence rates are observed for both RK and ARK coupling methods in asymptotic regimes as expected.
However, we observe that the rate of convergence of ARK4 decreases near $\mc{O}(10^{-8})$ error level. 
This decrease might be because we construct the linear operator $\Lb^z$ in \eqnref{linear-inviscid} 
based on both analytical Jacobian and finite difference (FD) approximation  
\footnote{
  For example, we use FD approximation for computing the linearized Roe flux,
   where the perturbation of absolute values of eigenvalues is approximated 
   $ \delta \snor{\Lambda} \approx \frac{\snor{\Lambda (\qb+\epsilon\delta \qb)} - \snor{\Lambda (\qb)} }{\epsilon}$ 
   with $\epsilon = 10^{-8}$. 
} in our implementation.

 
\begin{table}[t] 
  \caption{Temporal convergence study  conducted for RK and IMEX tight coupling methods.
  The domain is discretized with $80 \times 800$ on $\Omega_1$ and $80 \times 80$ on $\Omega_2$. 
  We use the RK4 solution with $\dt=5\times 10^{-4}$ as the reference solution, and we 
  measure 
  the $L_2$ error at $t=2$.
  In general, the expected second-order,  third-order, and  fourth-order convergence rates are observed for both the RK and ARK ($\Lb^z$,TC) methods. 
  }
  \tablab{tmv-tempJump-tconv-reference} 
  \begin{center} 
  \begin{tabular}{*{1}{c}|*{1}{c}|*{2}{c}|*{2}{c}|*{2}{c}} 
  \hline 
  \multirow{2}{*}{ }
  & \multirow{2}{*}{$ \triangle t (Cr_1,Cr_2)$}
  & \multicolumn{2}{c}{$ \left\Vert \rho   - \rho_r   \right\Vert  $} 
  & \multicolumn{2}{c}{$ \left\Vert \rho {\bf u} - \rho {\bf u}_r \right\Vert  $} 
  & \multicolumn{2}{c}{$ \left\Vert \rho E - \rho E_r \right\Vert  $} \tabularnewline 
  & & error & order &error & order &error & order \tabularnewline 
  \hline\hline 
  \multirow{5}{*}{RK2}
  & 0.004 (0.72,0.08) &       3.128E-05 & $-$&       3.285E-05 & $-$&       8.606E-05 & $-$\tabularnewline
  & 0.002 (0.36,0.04) &       7.730E-06 &    2.017&       8.117E-06 &    2.017&       2.126E-05 &    2.017\tabularnewline
  & 0.001 (0.18,0.02) &       1.929E-06 &    2.003&       2.026E-06 &    2.003&       5.306E-06 &    2.003\tabularnewline
  & 0.0005(0.09,0.01) &       4.821E-07 &    2.000&       5.063E-07 &    2.000&       1.326E-06 &    2.000\tabularnewline
  \multicolumn{8}{c}{} \tabularnewline
  \multirow{5}{*}{RK3}
  & 0.005    (0.89,0.10) &       2.633E-06 & $-$&       2.765E-06 & $-$&       7.244E-06 & $-$\tabularnewline
  & 0.0025   (0.45,0.05) &       3.413E-07 &    2.948&       3.585E-07 &    2.948&       9.390E-07 &    2.948\tabularnewline
  & 0.00125  (0.22,0.02) &       4.285E-08 &    2.994&       4.500E-08 &    2.994&       1.179E-07 &    2.994\tabularnewline
  & 0.000625 (0.11,0.01) &       5.358E-09 &    3.000&       5.627E-09 &    3.000&       1.474E-08 &    3.000\tabularnewline
  \multicolumn{8}{c}{} \tabularnewline
  \multirow{5}{*}{RK4}
  & 0.005    (0.89,0.10) &       5.709E-08 & $-$&       5.996E-08 & $-$&       1.571E-07 & $-$\tabularnewline
  & 0.0025   (0.45,0.05) &       3.558E-09 &    4.004&       3.737E-09 &    4.004&       9.789E-09 &    4.004\tabularnewline
  & 0.00125  (0.22,0.02) &       2.223E-10 &    4.000&       2.333E-10 &    4.002&       6.111E-10 &    4.002\tabularnewline
  & 0.000625 (0.11,0.01) &       1.420E-11 &    3.969&       1.461E-11 &    3.997&       3.817E-11 &    4.001\tabularnewline
  \multicolumn{8}{c}{} \tabularnewline
  \multirow{5}{*}{ARK2 ($\Lb^z$,TC)}  
  &0.04(7.15,0.76)   &       2.352E-04 &      $-$&       2.447E-04 &      $-$&       6.483E-04 &      $-$\tabularnewline
  &0.02(3.57,0.38)   &       1.248E-04 &    0.914&       1.310E-04 &    0.901&       3.433E-04 &    0.917\tabularnewline
  &0.01(1.79,0.19)   &       4.395E-05 &    1.506&       4.616E-05 &    1.505&       1.209E-04 &    1.506\tabularnewline
  &0.005(0.89,0.10)  &       1.163E-05 &    1.919&       1.221E-05 &    1.919&       3.198E-05 &    1.919\tabularnewline
  &0.0025(0.45,0.05) &       2.919E-06 &    1.994&       3.065E-06 &    1.994&       8.030E-06 &    1.994\tabularnewline
  \multicolumn{8}{c}{} \tabularnewline
  \multirow{5}{*}{ARK3 ($\Lb^z$,TC)}  
  &0.04(7.15,0.76)   &       1.214E-04 &      $-$&       1.275E-04 &      $-$&       3.341E-04 &      $-$\tabularnewline
  &0.02(3.57,0.38)   &       4.758E-05 &    1.352&       4.997E-05 &    1.352&       1.309E-04 &    1.352\tabularnewline
  &0.01(1.79,0.19)   &       1.086E-05 &    2.132&       1.140E-05 &    2.132&       2.987E-05 &    2.132\tabularnewline
  &0.005(0.89,0.10)  &       1.621E-06 &    2.744&       1.702E-06 &    2.744&       4.459E-06 &    2.744\tabularnewline
  &0.0025(0.45,0.05) &       2.082E-07 &    2.960&       2.187E-07 &    2.960&       5.729E-07 &    2.960\tabularnewline
  \multicolumn{8}{c}{} \tabularnewline
  \multirow{5}{*}{ARK4 ($\Lb^z$,TC)} 
  &0.04(7.15,0.76)   &       3.911E-05 &      $-$&       4.108E-05 &      $-$&       1.076E-04 &      $-$\tabularnewline
  &0.02(3.57,0.38)   &       3.515E-06 &    3.476&       3.692E-06 &    3.476&       9.670E-06 &    3.476\tabularnewline
  &0.01(1.79,0.19)   &       2.371E-07 &    3.890&       2.507E-07 &    3.881&       6.521E-07 &    3.890\tabularnewline
  &0.005(0.89,0.10)  &       1.888E-08 &    3.650&       2.162E-08 &    3.535&       5.166E-08 &    3.658\tabularnewline
  &0.0025(0.45,0.05) &       3.782E-09 &    2.320&       5.073E-09 &    2.092&       1.018E-08 &    2.343\tabularnewline
  \hline\hline 
  \end{tabular} 
  \end{center} 
\end{table}

\subsubsection{IMEX loose coupling (concurrent and sequential coupling) methods }

The IMEX tight coupling approach can achieve high-order convergence in time, 
but it requires communication at each stage. 
Moreover, both the models advance with the same timestep size. 
We may relax the tight coupling condition 
by using concurrent and sequential couplings, as shown in Figure \figref{coupling-diagram}(b) and Figure \figref{coupling-diagram}(c).
In IMEX loose coupling methods, 
we solve the bottom (ocean) model implicitly using ARK time integrators 
but treat the top (atmosphere) model explicitly using the explicit part of ARK methods.
The computational domain is discretized with $80 \times 400$ elements on $\Omega_1$ and $80 \times 160$ elements on $\Omega_2$. 

We measure the $L_2$ error at $t=2$ with the reference RK4 solution of $\dt=5\times 10^{-4}$ 
and report the results in Table \tabref{tmv-tconv-imex-lc}.
Since the top (atmosphere) model has two substeps, 
the Courant number at each substep should be understood as $\half Cr_2$ in Table \tabref{tmv-tconv-imex-lc}. 
Unlike IMEX tight coupling methods, 
both the concurrent (CC2) and the sequential (SC2) coupling methods show 
 first-order convergence rates for density, momentum, and total energy. 
The difference between two error levels of CC2 and SC2 is within $\mathcal{O}(10^{-5})$.

\begin{table}[t] 
  \caption{Temporal convergence study  conducted for IMEX loose coupling methods.
  The domain is discretized with $80 \times 400$ on $\Omega_1$ and $80 \times 160$ on $\Omega_2$. 
  We use the RK4 solution with $\dt=5\times 10^{-4}$ as the reference solution, and we
  measure the $L_2$ error at $t=2$.
  }
  \tablab{tmv-tconv-imex-lc} 
  \begin{center} 
  \begin{tabular}{*{1}{c}|*{1}{c}|*{2}{c}|*{2}{c}|*{2}{c}} 
  \hline 
  \multirow{2}{*}{ }
  & \multirow{2}{*}{$ \triangle t (Cr_1,Cr_2)$}
  & \multicolumn{2}{c}{$ \left\Vert \rho   - \rho_r   \right\Vert  $} 
  & \multicolumn{2}{c}{$ \left\Vert \rho {\bf u} - \rho {\bf u}_r \right\Vert  $} 
  & \multicolumn{2}{c}{$ \left\Vert \rho E - \rho E_r \right\Vert  $} \tabularnewline 
  & & error & order &error & order &error & order \tabularnewline 
  \hline\hline 
  \multirow{5}{*}{\shortstack{ARK2 ($\Lb^z$,CC2)}}
  &0.05    (4.47,1.91)&       3.849E-04 & $-$&       3.969E-04 & $-$&       1.033E-03 & $-$\tabularnewline
  &0.025   (2.23,0.95)&       1.781E-04 &    1.112&       1.832E-04 &    1.115&       4.764E-04 &    1.117\tabularnewline
  &0.0125  (1.12,0.48)&       5.539E-05 &    1.685&       5.513E-05 &    1.733&       1.413E-04 &    1.753\tabularnewline
  &0.00625 (0.56,0.24)&       1.696E-05 &    1.707&       1.522E-05 &    1.857&       3.702E-05 &    1.933\tabularnewline
  &0.003125(0.28,0.12)&       6.371E-06 &    1.413&       4.839E-06 &    1.653&       1.030E-05 &    1.846\tabularnewline
  &0.001563(0.14,0.06)&       2.865E-06 &    1.153&       1.929E-06 &    1.327&       3.446E-06 &    1.579\tabularnewline
  &0.000781(0.07,0.03)&       1.390E-06 &    1.043&       8.944E-07 &    1.109&       1.441E-06 &    1.258\tabularnewline
  \multicolumn{8}{c}{} \tabularnewline
  \multirow{5}{*}{\shortstack{ARK2 ($\Lb^z$,SC2)}}  
  &0.05    (4.47,1.91)&       3.835E-04 & $-$&       3.954E-04 & $-$&       1.032E-03 & $-$\tabularnewline
  &0.025   (2.23,0.95)&       1.773E-04 &    1.113&       1.823E-04 &    1.117&       4.754E-04 &    1.118\tabularnewline
  &0.0125  (1.12,0.48)&       5.468E-05 &    1.697&       5.430E-05 &    1.747&       1.403E-04 &    1.760\tabularnewline
  &0.00625 (0.56,0.24)&       1.636E-05 &    1.741&       1.443E-05 &    1.911&       3.597E-05 &    1.964\tabularnewline
  &0.003125(0.28,0.12)&       5.962E-06 &    1.456&       4.187E-06 &    1.785&       9.289E-06 &    1.953\tabularnewline
  &0.001563(0.14,0.06)&       2.635E-06 &    1.178&       1.498E-06 &    1.483&       2.624E-06 &    1.824\tabularnewline
  &0.000781(0.07,0.03)&       1.272E-06 &    1.051&       6.556E-07 &    1.192&       9.070E-07 &    1.532\tabularnewline
  \multicolumn{8}{c}{} \tabularnewline
  \multirow{5}{*}{\shortstack{ARK3 ($\Lb^z$,CC2)}}
  &0.05    (4.47,1.91)&       2.074E-04 & $-$&       2.040E-04 & $-$&       5.208E-04 & $-$\tabularnewline
  &0.025   (2.23,0.95)&       7.514E-05 &    1.464&       6.949E-05 &    1.554&       1.720E-04 &    1.598\tabularnewline
  &0.0125  (1.12,0.48)&       2.467E-05 &    1.607&       1.807E-05 &    1.943&       3.688E-05 &    2.222\tabularnewline
  &0.00625 (0.56,0.24)&       1.112E-05 &    1.150&       7.116E-06 &    1.344&       1.133E-05 &    1.703\tabularnewline
  &0.003125(0.28,0.12)&       5.513E-06 &    1.012&       3.490E-06 &    1.028&       5.375E-06 &    1.075\tabularnewline
  \multicolumn{8}{c}{} \tabularnewline
  \multirow{5}{*}{\shortstack{ARK3 ($\Lb^z$,SC2)}}  
  &0.05    (4.47,1.91)&       2.068E-04 & $-$&       2.008E-04 & $-$&       5.161E-04 & $-$\tabularnewline
  &0.025   (2.23,0.95)&       7.494E-05 &    1.465&       6.708E-05 &    1.582&       1.682E-04 &    1.617\tabularnewline
  &0.0125  (1.12,0.48)&       2.458E-05 &    1.608&       1.565E-05 &    2.099&       3.219E-05 &    2.386\tabularnewline
  &0.00625 (0.56,0.24)&       1.109E-05 &    1.148&       5.521E-06 &    1.504&       6.927E-06 &    2.216\tabularnewline
  &0.003125(0.28,0.12)&       5.506E-06 &    1.010&       2.677E-06 &    1.044&       2.986E-06 &    1.214\tabularnewline
  \multicolumn{8}{c}{} \tabularnewline
  \multirow{5}{*}{\shortstack{ARK4 ($\Lb^z$,CC2)}}
  &0.05    (4.47,1.91)&       9.853E-05 & $-$&       7.092E-05 & $-$&       1.417E-04 & $-$\tabularnewline
  &0.025   (2.23,0.95)&       4.462E-05 &    1.143&       2.839E-05 &    1.321&       4.447E-05 &    1.671\tabularnewline
  &0.0125  (1.12,0.48)&       2.215E-05 &    1.011&       1.403E-05 &    1.016&       2.170E-05 &    1.035\tabularnewline
  &0.00625 (0.56,0.24)&       1.105E-05 &    1.004&       6.997E-06 &    1.004&       1.080E-05 &    1.006\tabularnewline
  &0.003125(0.28,0.12)&       5.515E-06 &    1.002&       3.493E-06 &    1.002&       5.390E-06 &    1.003\tabularnewline
  \multicolumn{8}{c}{} \tabularnewline
  \multirow{5}{*}{\shortstack{ARK4 ($\Lb^z$,SC2)}} 
  &0.05    (4.47,1.91)&       9.884E-05 & $-$&       6.078E-05 & $-$&       1.207E-04 & $-$\tabularnewline
  &0.025   (2.23,0.95)&       4.505E-05 &    1.134&       2.204E-05 &    1.463&       2.565E-05 &    2.234\tabularnewline
  &0.0125  (1.12,0.48)&       2.243E-05 &    1.006&       1.089E-05 &    1.016&       1.217E-05 &    1.076\tabularnewline
  &0.00625 (0.56,0.24)&       1.120E-05 &    1.002&       5.441E-06 &    1.002&       6.069E-06 &    1.004\tabularnewline
  &0.003125(0.28,0.12)&       5.598E-06 &    1.001&       2.719E-06 &    1.001&       3.030E-06 &    1.002\tabularnewline
  \hline\hline 
  \end{tabular} 
  \end{center} 
\end{table}

\subsection{CNS2: Wind-driven flow}

The ocean current is driven by the wind at the sea surface. 
Under the assumption of the rigid-lid interface, 
we mimic the wind-driven flows for a coupled compressible Navier--Stokes equation. 
The computational domain is the same as the one in Section \secref{cns2-tmv}. 
We take $T_{\infty_1} = 1.1$, $T_{\infty_2}=1$,
and $\rho_{\infty_1}=\rho_{\infty_2}=1$
with $Pr=0.72$ and $\gamma=1.4$. 
We impose a horizontally moving isothermal wall condition at the top boundary, for example, 
$T_{w_2}=0.9$ and $u_{w_2}=0.1$, and an isothermal no-slip condition at the bottom wall, for example,
$T_{w_1}=1$ and $u_{w_1}=0$. 
We apply periodic boundary conditions laterally on $\Omega_2$ 
and adiabatic no-slip wall boundary conditions at the left and the right walls on $\Omega_1$.
We set uniform horizontal flows $u_{\infty_2}=0.1$ on $\Omega_2$ 
and zero velocity on $\Omega_1$. 
Similar to the lid-driven cavity flow \cite{bruneau20062d},
the viscous drag force exerted by the velocity difference at the interface induces circular motion on $\Omega_1$.
 
We perform simulations for 
$\tilde{\mu}_m = \LRc{10^{-1},500^{-1},1000^{-1},5000^{-1}}$
  on the grid with $80\times80$ elements on $\Omega_1$ 
  and $80 \times 80$ elements on $\Omega_2$ using RK4 methods.
In Figure \figref{wdf-re-ss}, 
temperature fields for 
$\tilde{\mu}_m = \LRc{10^{-1},500^{-1},1000^{-1},5000^{-1}}$
are contoured 
with streamlines at $t=500$. 
Since the heat and momentum fluxes are 
proportional to $\tilde{\mu}_m$, 
as shown  for 
$\tilde{\mu}_m=100^{-1}$, 
the temperature $T_1=1.1$ quickly cools  to $T_1=1.0$, compared with other cases 
(for $\tilde{\mu}_m=\LRc{500^{-1},1000^{-1},5000^{-1}}$). 
As a result, the horizontal momentum becomes 
the main driving source for the flows on $\Omega_1$. 
As we 
decrease $\tilde{\mu}_m$,
the heat transfer becomes weaker.
Thus, sharp gradients of temperature fields are observed near the interface and the boundaries at the top and the bottom.  
The temperature gradients induce the vertical motion of the fluid $\Omega_1$, 
which pushes the center of the circulation toward the right wall, as shown for $\tilde{\mu}_m=5000^{-1}$. 
 
Since this case clearly shows how cooled fluid moves away from the interface, 
we choose it for the comparison of IMEX coupling methods in the following section. 


\begin{figure}[h!t!b!]
  \centering
    \subfigure[$\tilde{\mu}_m=100^{-1}$]{
    \includegraphics[trim=2.5cm 1.5cm 2.65cm 2.2cm,
      clip=true,width=0.47\columnwidth]{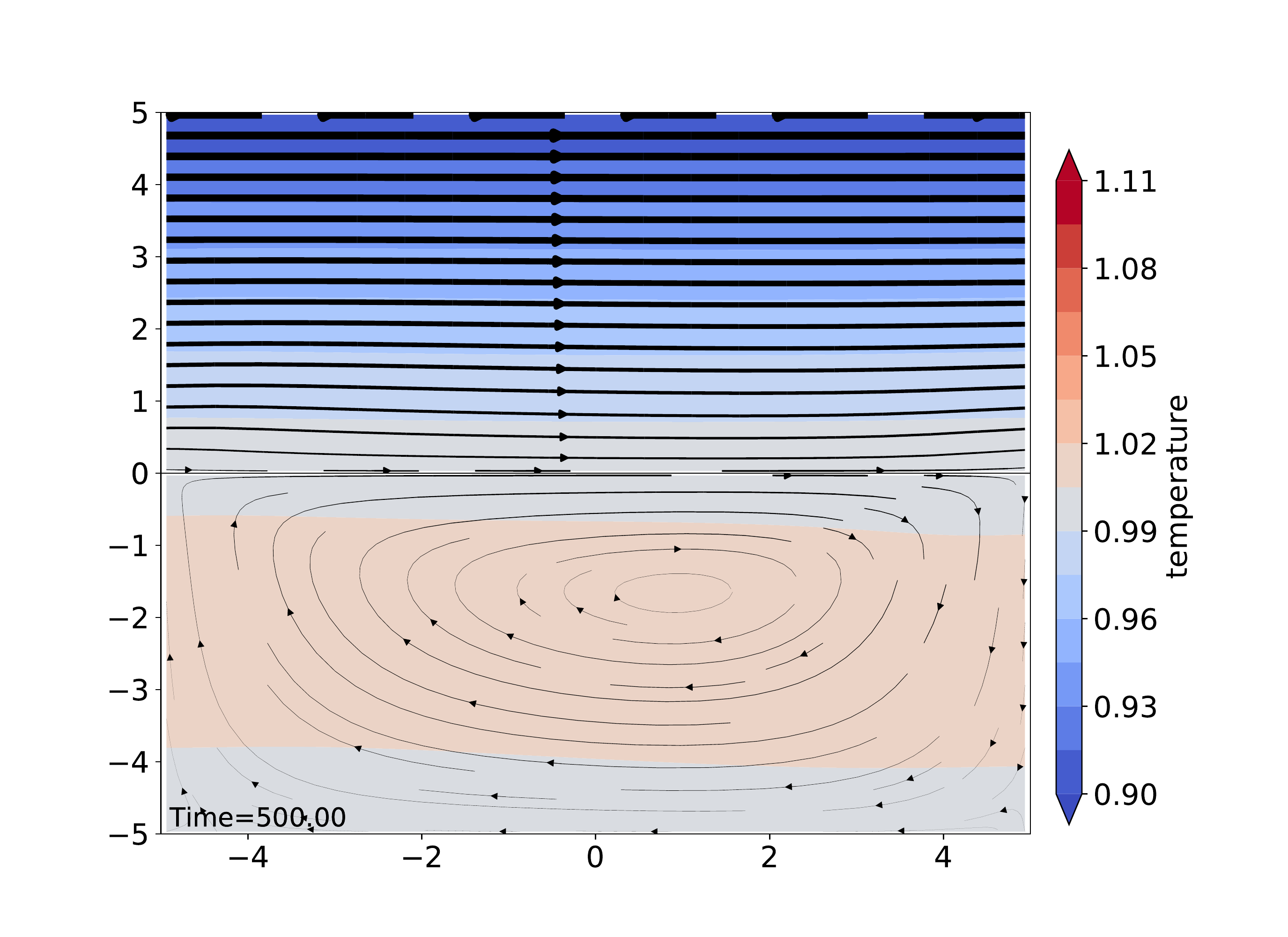}
  }
  \subfigure[$\tilde{\mu}_m=500^{-1}$]{
    \includegraphics[trim=2.5cm 1.5cm 2.65cm 2.2cm,
      clip=true,width=0.47\columnwidth]{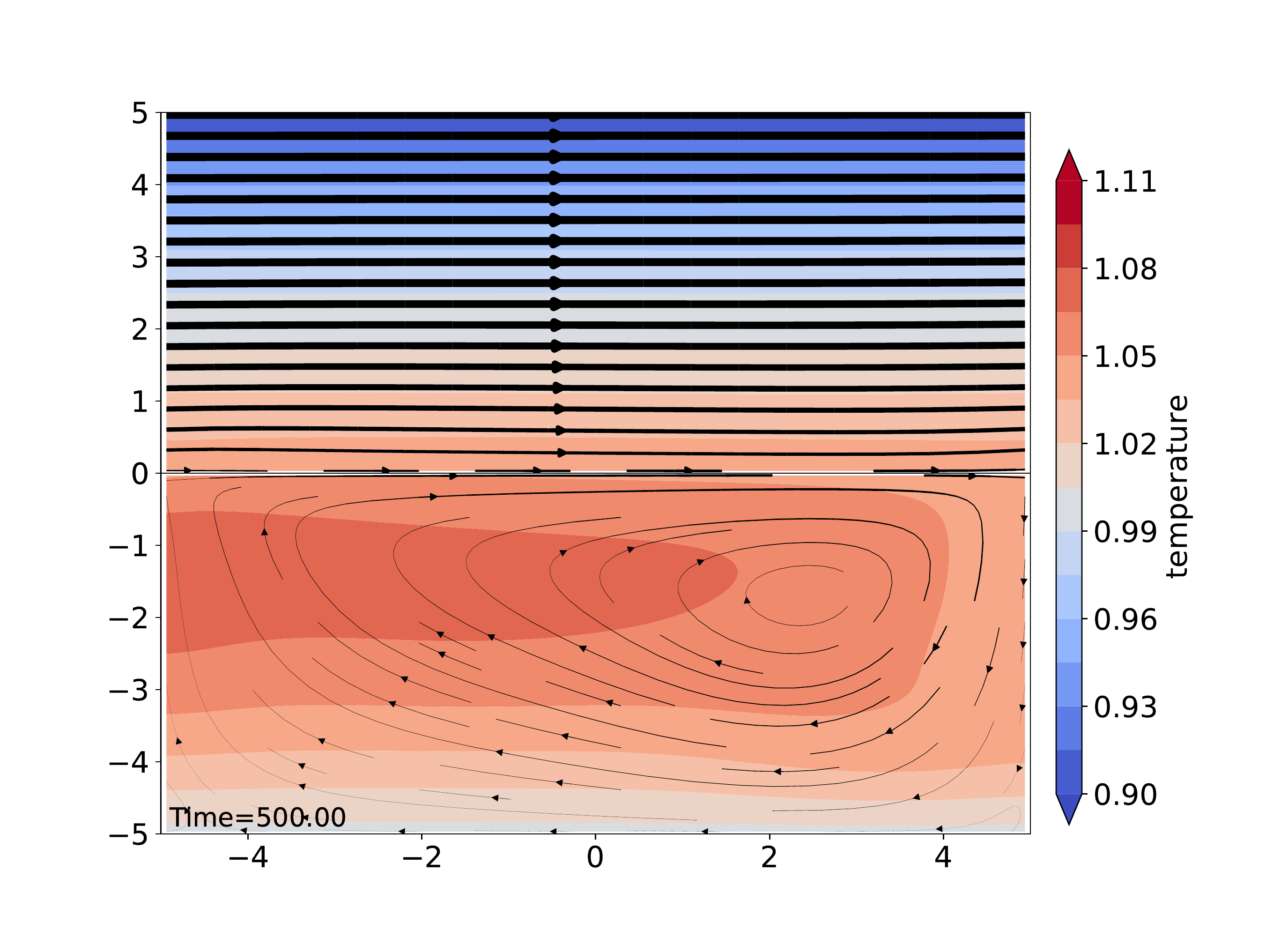}
  }\\
  \subfigure[$\tilde{\mu}_m=1000^{-1}$]{
    \includegraphics[trim=2.5cm 1.5cm 2.65cm 2.2cm,
      clip=true,width=0.47\columnwidth]{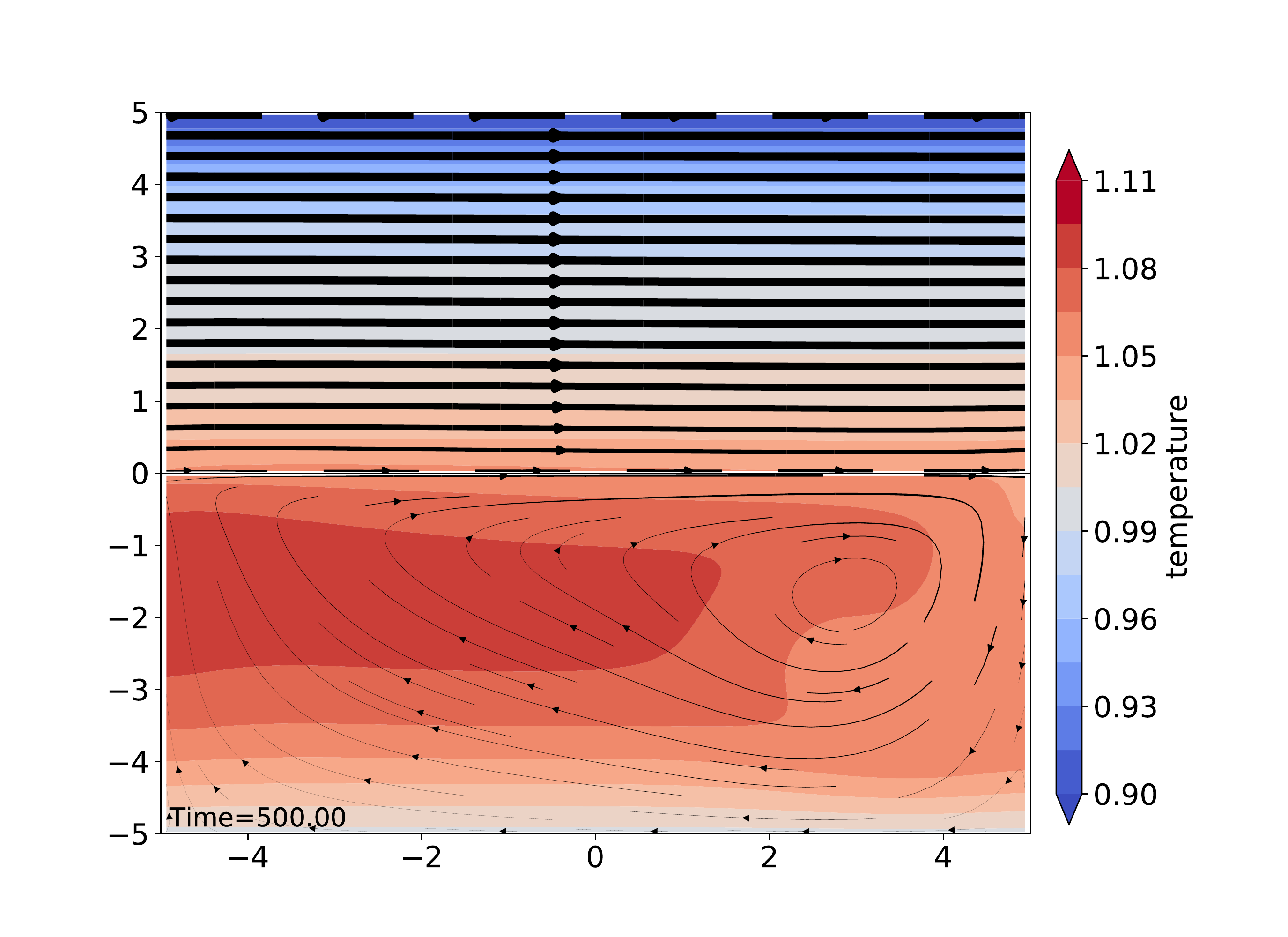}
  }
  \subfigure[$\tilde{\mu}_m=5000^{-1}$]{
    \includegraphics[trim=2.5cm 1.5cm 2.65cm 2.2cm,
      clip=true,width=0.47\columnwidth]{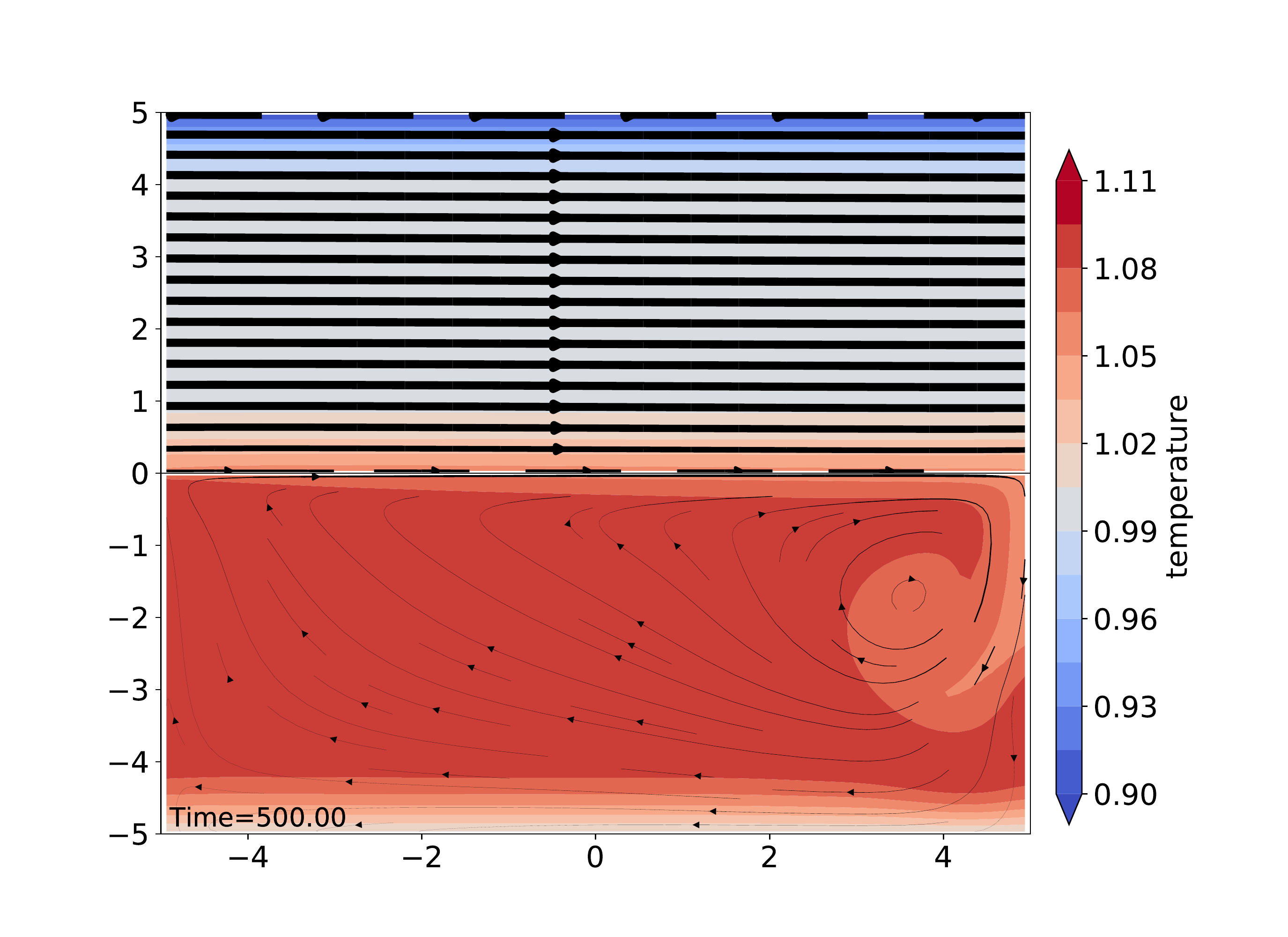}
  }
  \caption{Wind-driven flows: the snapshots of the temperature field at $t=500$.
  Simulations are performed 
  on the grid with $80\times80$ elements on $\Omega_1$ 
  and $80 \times 80$ elements on $\Omega_2$ using RK4 methods.}
  \figlab{wdf-re-ss}
\end{figure}

\subsubsection{IMEX coupling methods for $\tilde{\mu}_m=5000^{-1}$} 

We first conduct numerical experiments for ARK4($\Lb^z$,TC)
on the grid with $100\times500$ elements on $\Omega_1$ and $100 \times 80$ elements on $\Omega_2$ for $t\in\LRc{0,500}$.
The timestep size is taken as $\dt=0.05$, which corresponds to $Cr_1=5.38$ and $Cr_2=0.88$. 
In Figure \figref{wdf-re5000-hevi-ark4-tc-ss} we plot the temperature fields with streamlines (black solid lines) over the simulation period.
We see the temperature near the top and the bottom boundaries decreases 
because of the cold wall boundary condition. At the interface, 
the temperature is getting hotter on $\Omega_2$ and colder on $\Omega_1$ through the heat exchange. 
The cooled fluid moves along the right wall and rolls up following the circulation at $t=500$. 
The temperature field as well as the streamline at $t=500$ shows good agreement with the RK4 solution 
in Figure \figref{wdf-re5000-diff}(a).
The difference between RK4 and ARK4 ($\Lb^z$,TC) is within $\mathcal{O}(10^{-5})$ in Figure \figref{wdf-re5000-diff}(b).

%
Now we compare the performance of IMEX coupling methods. 
We choose the largest timestep size 
for each IMEX coupling method. 
\footnote{
  For example, 
  ARK2($\Lb^z$,TC) with $\dt=0.05$ leads to unstable solutions, thus,  
 we take $\dt=0.04$ for ARK2($\Lb^z$,TC). 
} 
We take the RK4 (with $\dt=0.01$) solution as a reference and 
measure the $L_2$ relative errors.
In Table \tabref{wdfh-re5000-coupling} we summarize the relative errors and wall-clock times for IMEX coupling methods.
The relative errors of TC, SC2, and CC2 coupling methods for ARK3 have 
the same order of accuracy for density, momentum, and total energy. 
For example, the order of relative error for density is $\mathcal{O}(10^{-4})$. 
Similarly, the relative errors of ARK2 ($\Lb^z$, SC2) and ARK2 ($\Lb^z$, CC2) 
have $\mathcal{O}(10^{-4})$ order of accuracy for density, momentum, and total energy.
Compared with ARK3 ($\Lb^z$, TC), ARK4 ($\Lb^z$, TC) is closer to the RK4 solution.
However, ARK2 ($\Lb^z$, SC2/CC2) has a smaller relative error compared with that of ARK3 ($\Lb^z$, SC2/CC2).
Compared to ARK4 ($\Lb^z$, SC2/CC2), ARK2 ($\Lb^z$, SC2/CC2) has slightly smaller relative errors of density and total energy. 
From an accuracy point of view, 
IMEX tight coupling benefits from using high-order methods; however, 
IMEX loose coupling does not. 
The reason may be that the heat and the horizontal momentum fluxes are not continuous at the interface for IMEX loose coupling schemes. 
From a stability viewpoint, 
ARK2 ($\Lb^z$, SC2/CC2) is more stable than ARK2 ($\Lb^z$, TC). 
This is because 
the Courant number at each substep for ARK2 ($\Lb^z$, SC2/CC2) becomes 
smaller than ARK2 ($\Lb^z$, TC) counterpart. 
By increasing subcycles in the explicit part. 
the acoustic mode in atmospheric model is resolved with smaller timestep size. 
As for the computational cost, IMEX ARK2 and ARK3 coupling schemes are comparable to the RK4 tight coupling method. 
\footnote{
  The relative errors with \texttt{1e-4} 
  tolerance for the Krylov solver
  are similar to those with \texttt{1e-2}
  in this example. 
  To enhance computational cost, we use 
  \texttt{1e-2}
  tolerance.
}


\begin{figure}[h!t!b!]
  \centering
    \includegraphics[trim=2.5cm 1.5cm 2.65cm 2.2cm,
      clip=true,width=0.47\columnwidth]{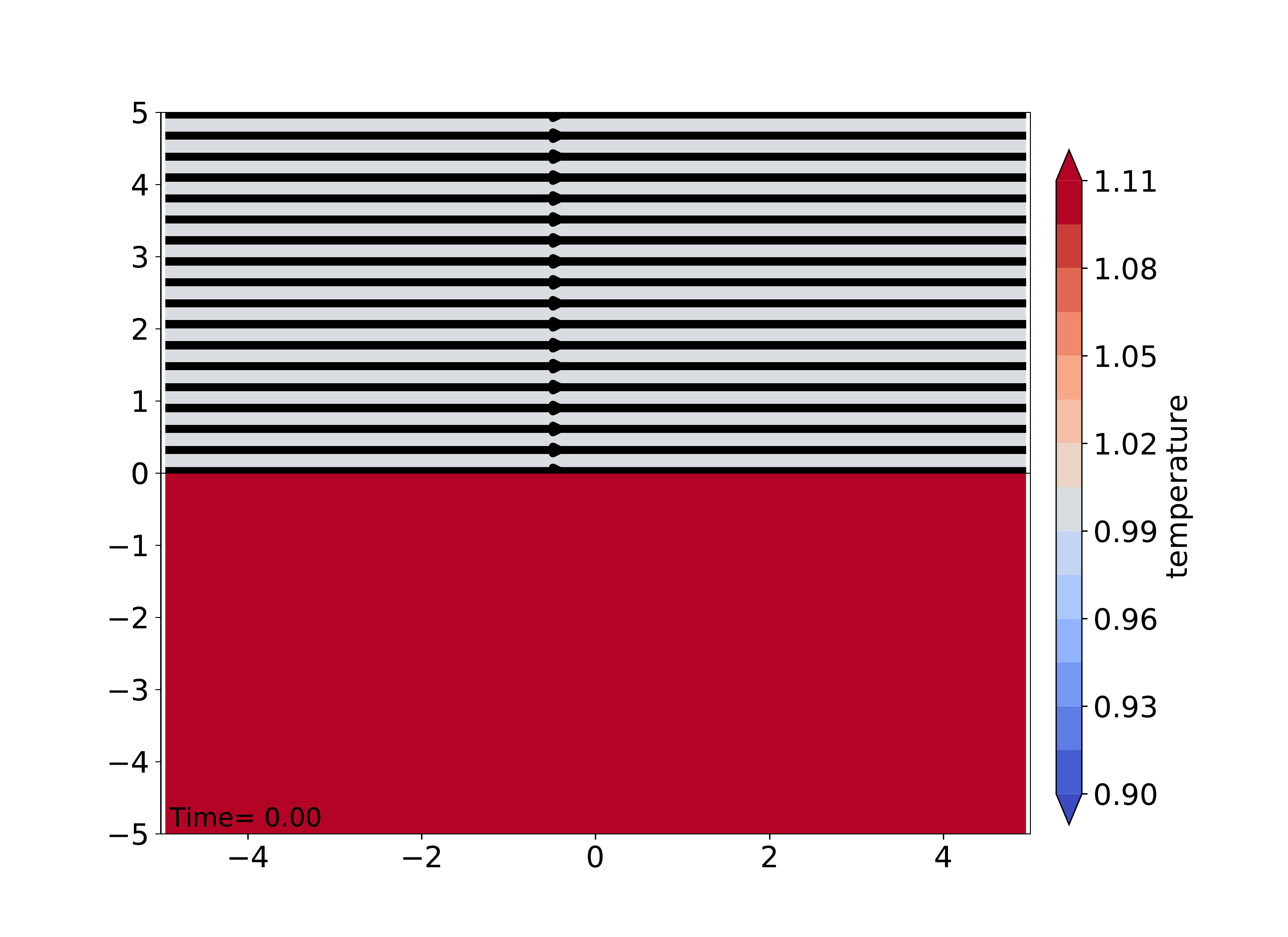}
    \includegraphics[trim=2.5cm 1.5cm 2.65cm 2.2cm,
      clip=true,width=0.47\columnwidth]{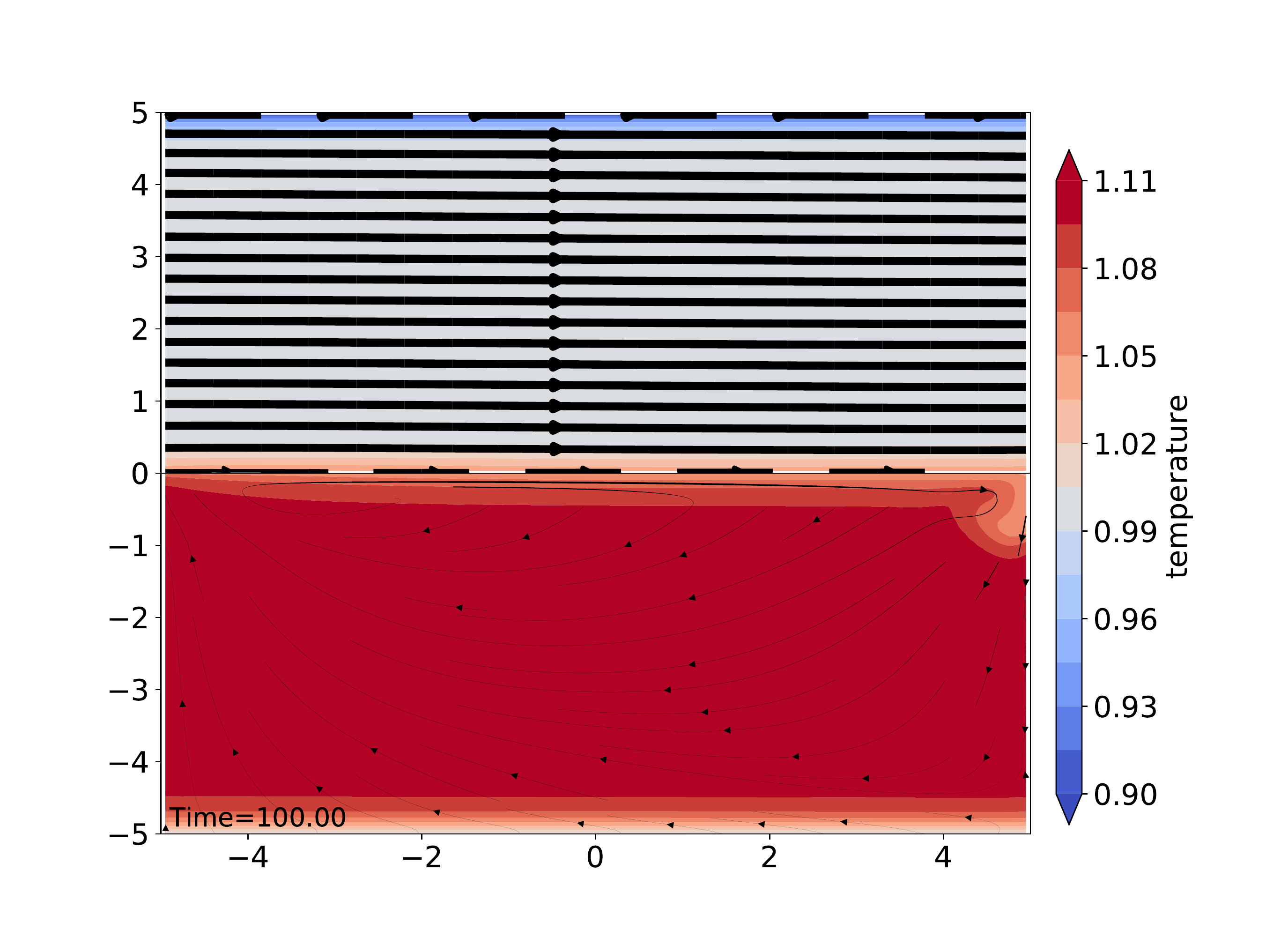}
    \includegraphics[trim=2.5cm 1.5cm 2.65cm 2.2cm,
      clip=true,width=0.47\columnwidth]{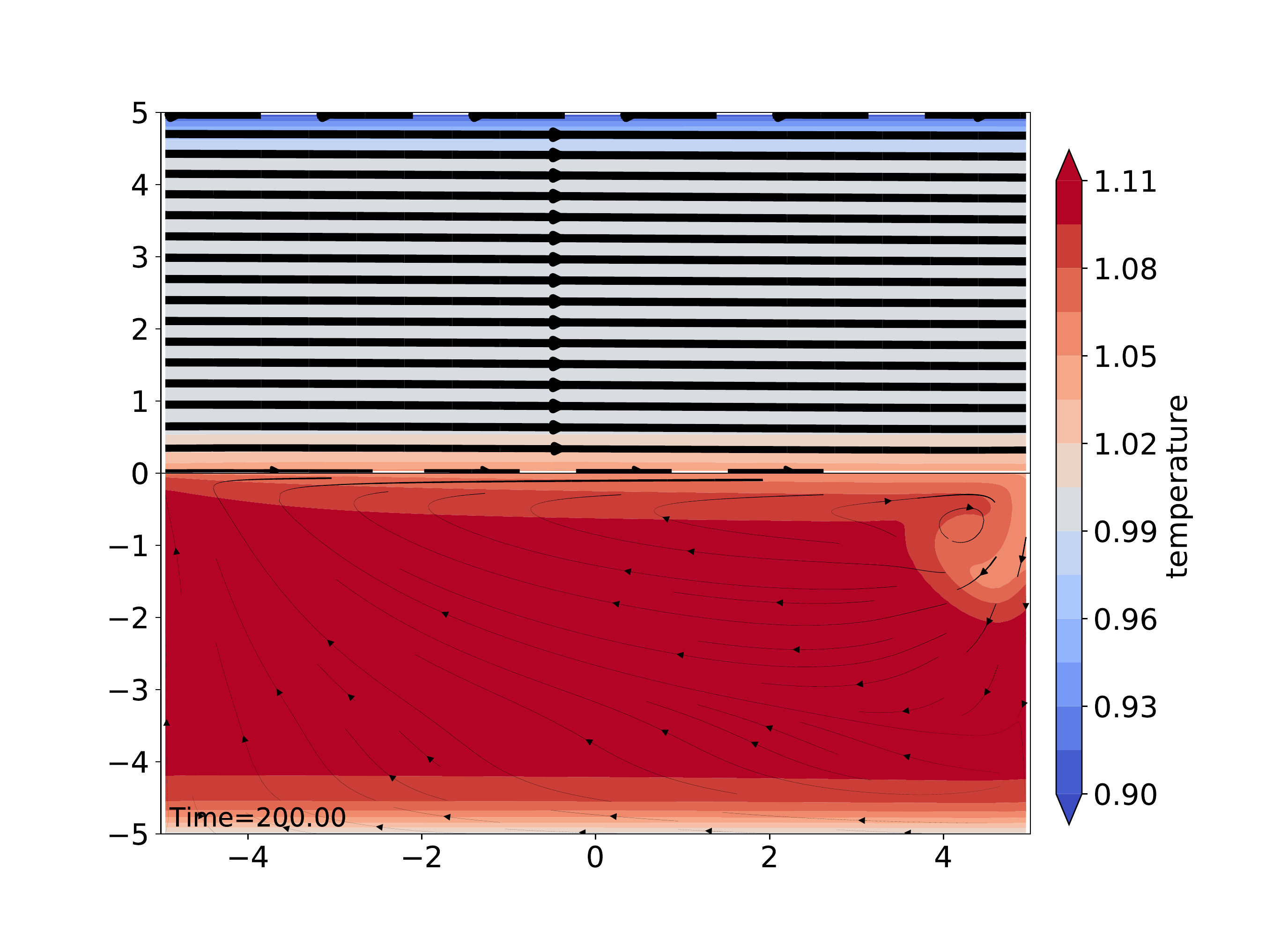}
    \includegraphics[trim=2.5cm 1.5cm 2.65cm 2.2cm,
      clip=true,width=0.47\columnwidth]{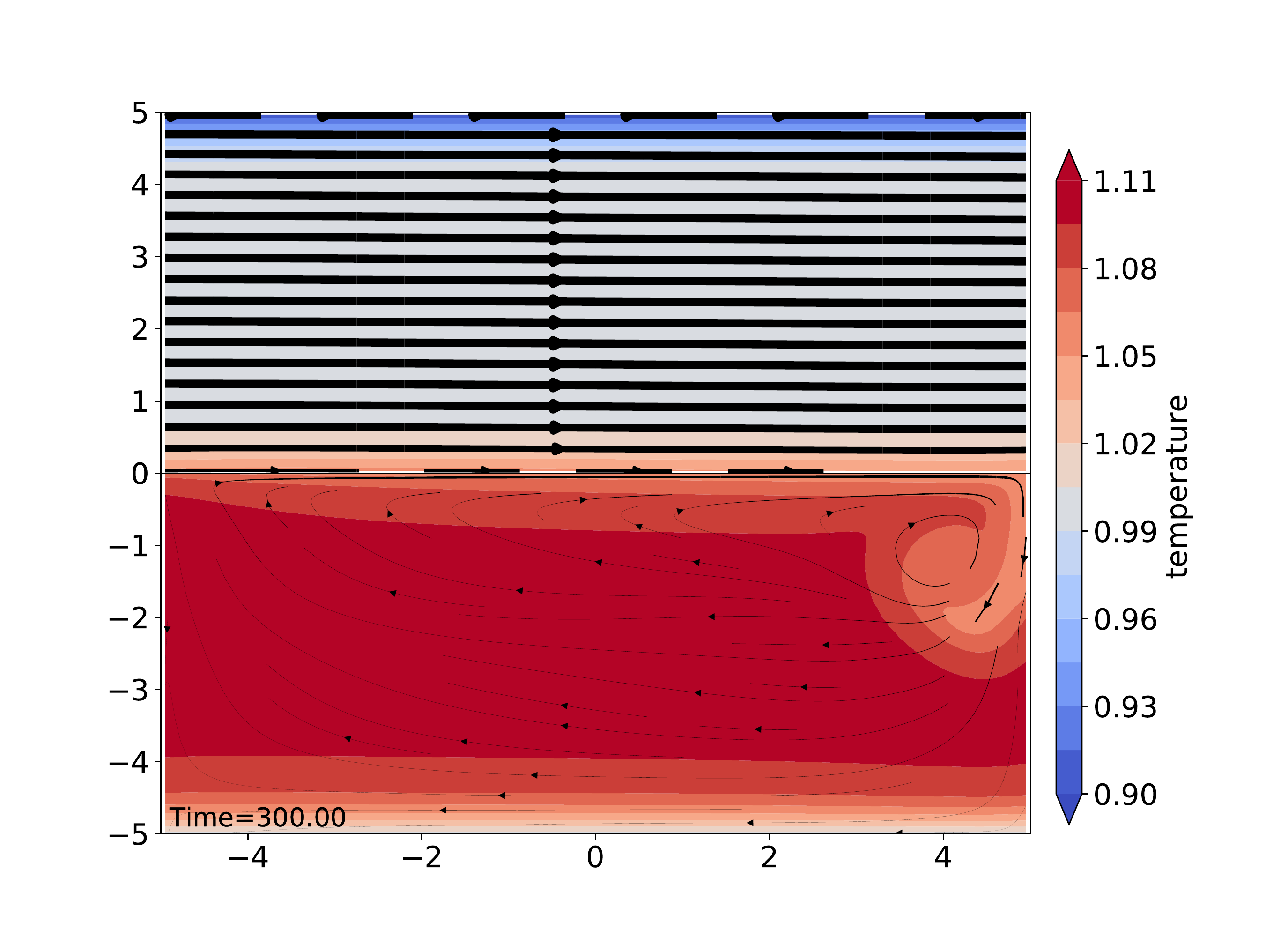}
    \includegraphics[trim=2.5cm 1.5cm 2.65cm 2.2cm,
      clip=true,width=0.47\columnwidth]{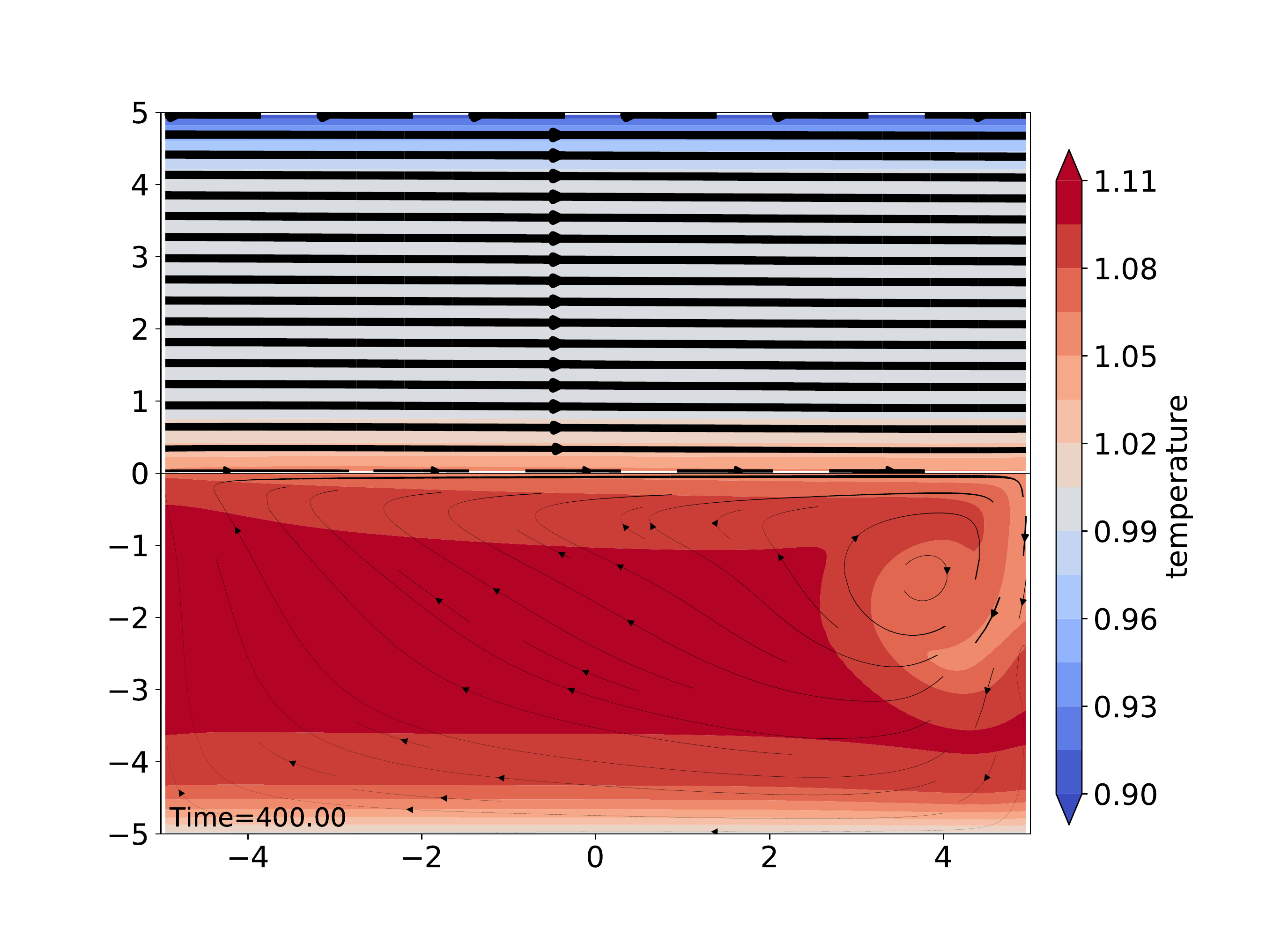}
    \includegraphics[trim=2.5cm 1.5cm 2.65cm 2.2cm,
      clip=true,width=0.47\columnwidth]{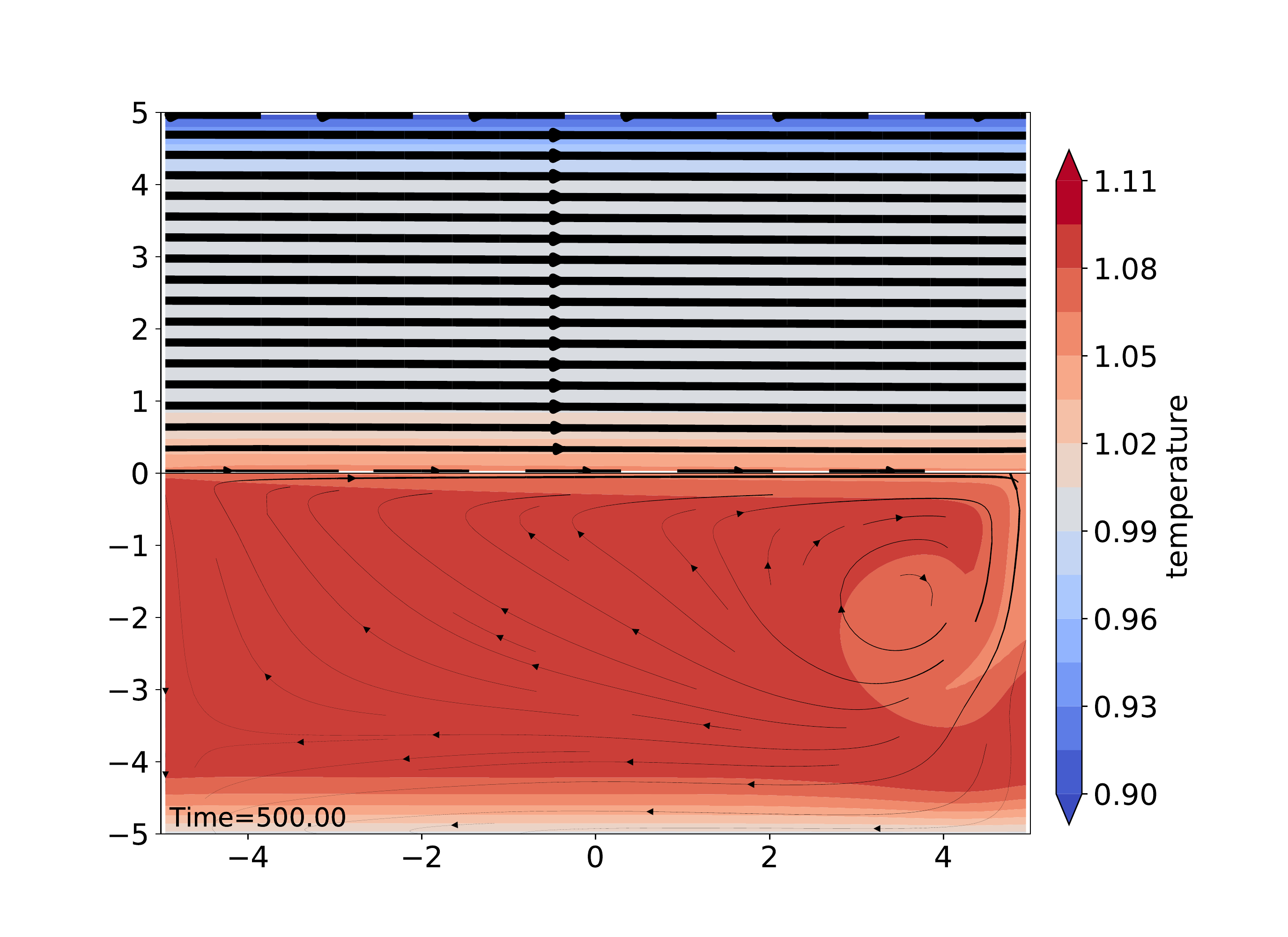}
  \caption{Evolution of the temperature field for the wind-driven flows:
  simulation is conducted with the ARK4 ($\Lb^z$,TC) method over 
  a mesh of $100\times500$ elements on $\Omega_1$ 
  and $100 \times 80$ elements on $\Omega_2$ for $t\in\LRs{0,500}$.
  Black solid lines represent streamline, and the color bar shows the range of temperature from $0.9$ to $1.1$.}
  
  \figlab{wdf-re5000-hevi-ark4-tc-ss}
\end{figure}

\begin{table}[t] 
  \caption{Relative errors and wall-clock times for IMEX coupling methods with respect to the RK4 ($\dt=0.01$) coupling at $t=500$.
  The domain is discretized with $100\times500$ on $\Omega_1$ and $100 \times 80$ on $\Omega_2$.  $Cr_2$ is the Courant number for the entire timestep; the substep number needs to be scaled by 1/2 for SC2 and CC2.
  }
  \tablab{wdfh-re5000-coupling} 
  \begin{center} 
    \begin{tabular}{*{1}{c}|*{1}{c}|*{1}{c}|*{1}{c}|*{1}{c}|*{1}{c}}  
    \hline 
    \multirow{2}{*}{ }
    & \multirow{2}{*}{$ \triangle t (Cr_1,Cr_2)$}
    & \multirow{2}{*}{$ \left\Vert \rho  - \rho_r \right\Vert  $}  
    & \multirow{2}{*}{$ \left\Vert \rho {\bf u} - \rho {\bf u}_r \right\Vert  $}  
    & \multirow{2}{*}{$ \left\Vert \rho E - \rho E_r \right\Vert  $}  
    & \multirow{2}{*}{wc[s]} \tabularnewline 
    & && & &  \tabularnewline 
    \hline\hline 
    RK4                 & 0.01 (1.08,0.18)&              - &              - &              - & 5570.06\tabularnewline
    \multicolumn{6}{c}{} \tabularnewline
    ARK2 ($\Lb^z$,TC)  & 0.04 (4.20,0.70)&       7.435E-05&       2.049E-04&       1.989E-04& 3322.84\tabularnewline
    ARK2 ($\Lb^z$,SC2) & 0.05 (5.24,0.88)&       1.310E-04&       3.075E-04&       3.318E-04& 3280.01\tabularnewline
    ARK2 ($\Lb^z$,CC2) & 0.05 (5.24,0.88)&       1.393E-04&       3.082E-04&       3.323E-04& 3312.13\tabularnewline
    \multicolumn{6}{c}{} \tabularnewline
    ARK3 ($\Lb^z$,TC)   & 0.05 (5.24,0.88)&       7.231E-04&       1.128E-03&       1.959E-03& 5214.52\tabularnewline
    ARK3 ($\Lb^z$,SC2)  & 0.05 (5.24,0.88)&       7.266E-04&       1.128E-03&       1.959E-03& 5400.41\tabularnewline
    ARK3 ($\Lb^z$,CC2)  & 0.05 (5.24,0.88)&       7.336E-04&       1.129E-03&       1.959E-03& 5363.56\tabularnewline
    \multicolumn{6}{c}{} \tabularnewline
    ARK4 ($\Lb^z$,TC)   & 0.05 (5.24,0.88)&       1.468E-04&       1.661E-04&       3.956E-04& 6355.54\tabularnewline
    ARK4 ($\Lb^z$,SC2)  & 0.05 (5.24,0.88)&       1.452E-04&       1.641E-04&       3.961E-04& 7361.94\tabularnewline
    ARK4 ($\Lb^z$,CC2)  & 0.05 (5.24,0.88)&       1.452E-04&       1.660E-04&       3.966E-04& 7867.85\tabularnewline
  \hline\hline 
  \end{tabular} 
  \end{center} 
\end{table}

\begin{figure}[h!t!b!]
  \centering
    
  \subfigure[RK4]{
    \includegraphics[trim=2.5cm 1.5cm 2.65cm 2.2cm,
      clip=true,width=0.47\columnwidth]{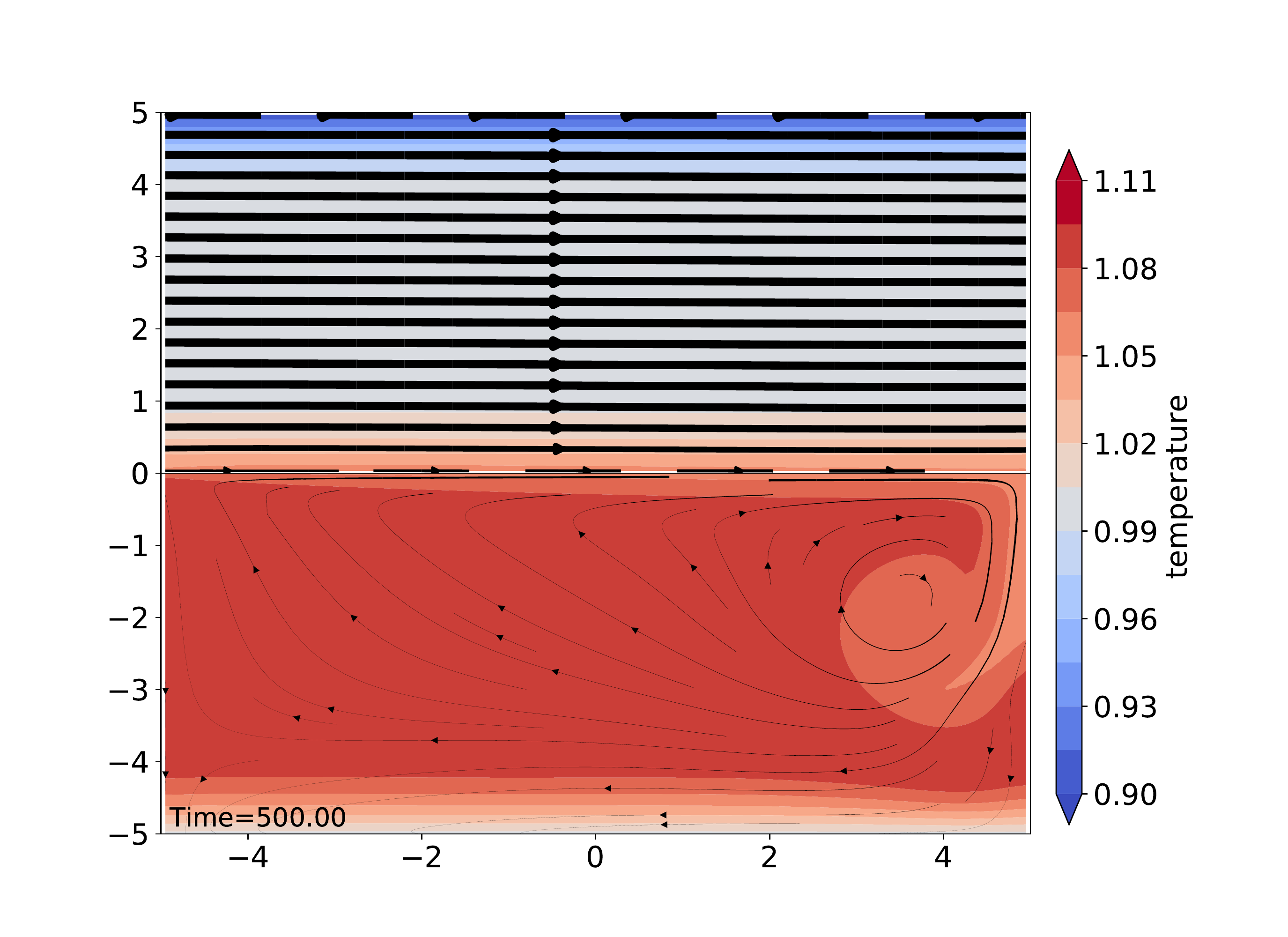}
  }
  \subfigure[Difference]{
    \includegraphics[trim=2.5cm 1.5cm 2.65cm 2.2cm,
      clip=true,width=0.47\columnwidth]{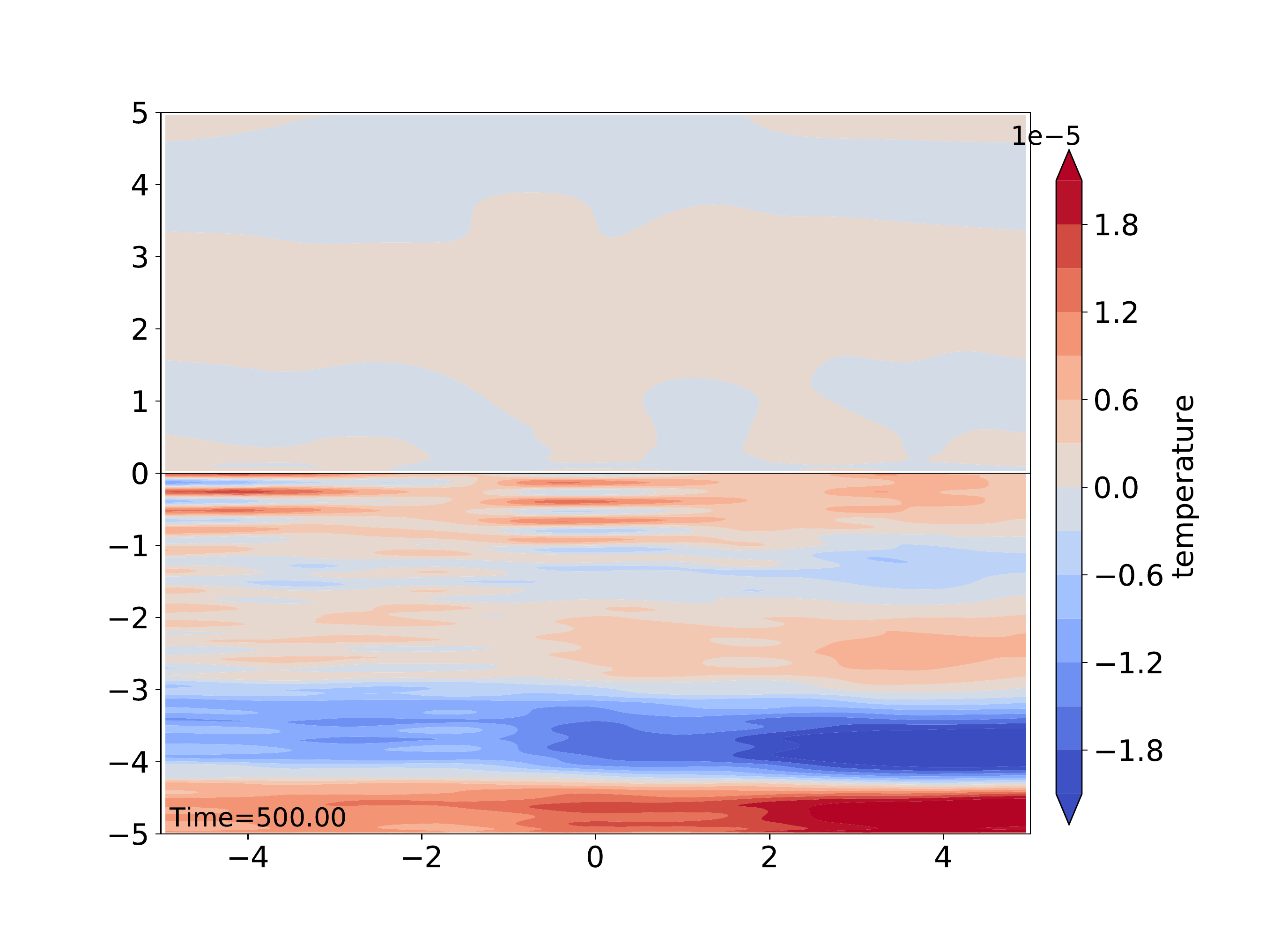}
  }
  \caption{Wind-driven flows at $t=500$: (a) temperature field with streamlines of the RK4 solution  
  and (b) the difference between ARK4 ($\Lb^z$,TC) and RK4. 
  Simulations are performed 
  on the grid with $100\times500$ elements on $\Omega_1$ 
  and $100 \times 80$ elements on $\Omega_2$.
  We take $\dt=0.05$ for ARK4 ($\Lb^z$,TC) and $\dt=0.01$ for RK4.}
  \figlab{wdf-re5000-diff}
\end{figure}

\subsection{CNS2: Wind-driven flows with Kelvin--Helmholtz instability}
\seclab{CNS2-KHI}
 
Kelvin--Helmholtz instability (KHI) is an important mechanism in the development of turbulence in the stratified atmosphere and ocean. 
KHI arises when two fluids have different densities and tangential velocities across the interface. Small disturbances such as waves at the interface grow exponentially, and the interface rolls up into KH rotors \cite{drazin2004hydrodynamic, springel2010pur, lecoanet2016validated}.
To see the nonlinear evolution of KHI, 
we add a jet to $\Omega_2$ in the wind-driven example above.
We also place a vortex on $\Omega_1$.
Boundary conditions are the same as in the wind-driven example. 
The initial conditions are chosen as 
\begin{align*}
  \rho_1 & = \LRp{ 1 - \frac{(\gamma-1)\beta^2}{8 \alpha \gamma \pi^2} e^{\alpha\LRp{1-r^2}} }^{\frac{1}{\gamma-1}},\\
  u_1 & =  \frac{\beta}{2\pi}\tilde{z}_1 e^{\frac{\alpha}{2}\LRp{1-r^2} }, \\
  w_1 & = - \frac{\beta}{2\pi}\tilde{x}_1 e^{\frac{\alpha}{2}\LRp{1-r^2} }, \\
  \pres_1 & = 1.1 \gamma^{-1} \rho_1^\gamma
\end{align*} 
with $r = \norm{\xb - \tilde{\xb}}$,
$\tilde{x}_1=x_1 - x_c$,
$\tilde{z}_1=z_1 - z_c$ for $\Omega_1$, 
\begin{align*}
  \rho_2 &= 1 + \half \LRp{ \tanh\LRp{\frac{z_2 - s_1}{a}} - \tanh\LRp{\frac{z_2 - s_2}{a}} },\\
     u_2 &= 0.1 + \LRp{ \tanh\LRp{\frac{z_2-s_1}{a}} - \tanh\LRp{\frac{z_2-s_2}{a}} - 1 },\\
     w_2 &= A \sin\LRp{2\pi x_2} \LRp{ \exp\LRp{-\frac{(z_2-s_1)^2}{\sigma^2} } 
                                 + \exp\LRp{-\frac{(z_2-s_2)^2}{\sigma^2} }  }, \\
  \pres_2 &= \gamma^{-1}
\end{align*}
for $\Omega_2$. 
Here, we take $a=0.05$, $A=0.01$, $\sigma=0.2$,
 $s_1=2$, $s_2=3$,
$\xb_c=(0,-3)$, $\alpha=5$
and $\beta=0.5$. 
We choose the fluid parameters of $\gamma=1.4$, $Pr=0.72$, $\tilde{c}_p=(\gamma-1)^{-1}$, and $\tilde{\mu}_m=5000^{-1}$.

We conduct the simulation with the ARK4 ($\Lb^z$,TC) method over 
  the mesh of $100\times400$ elements on $\Omega_1$ 
  and $100 \times 80$ elements on $\Omega_2$\footnote{
  We use $10^{-2}$ Krylov tolerance for a linear solver.
} in Figure \figref{khi-re5000-hevi-ark4-tc-ss}.
The evolution of temperature fields is shown for $t\in\LRs{0,500}$.
Kelvin--Helmholtz waves are well developed at $t=100$ and start to diffuse while mixing fluids. 
Meanwhile, because of the heat and horizontal momentum exchange, 
fluid on $\Omega_1$ cools  near the interface and moves along 
the outside of vortex as time passes.  



\begin{figure}[h!t!b!]
  \centering
    \includegraphics[trim=2.5cm 1.5cm 2.65cm 2.2cm,
      clip=true,width=0.47\columnwidth]{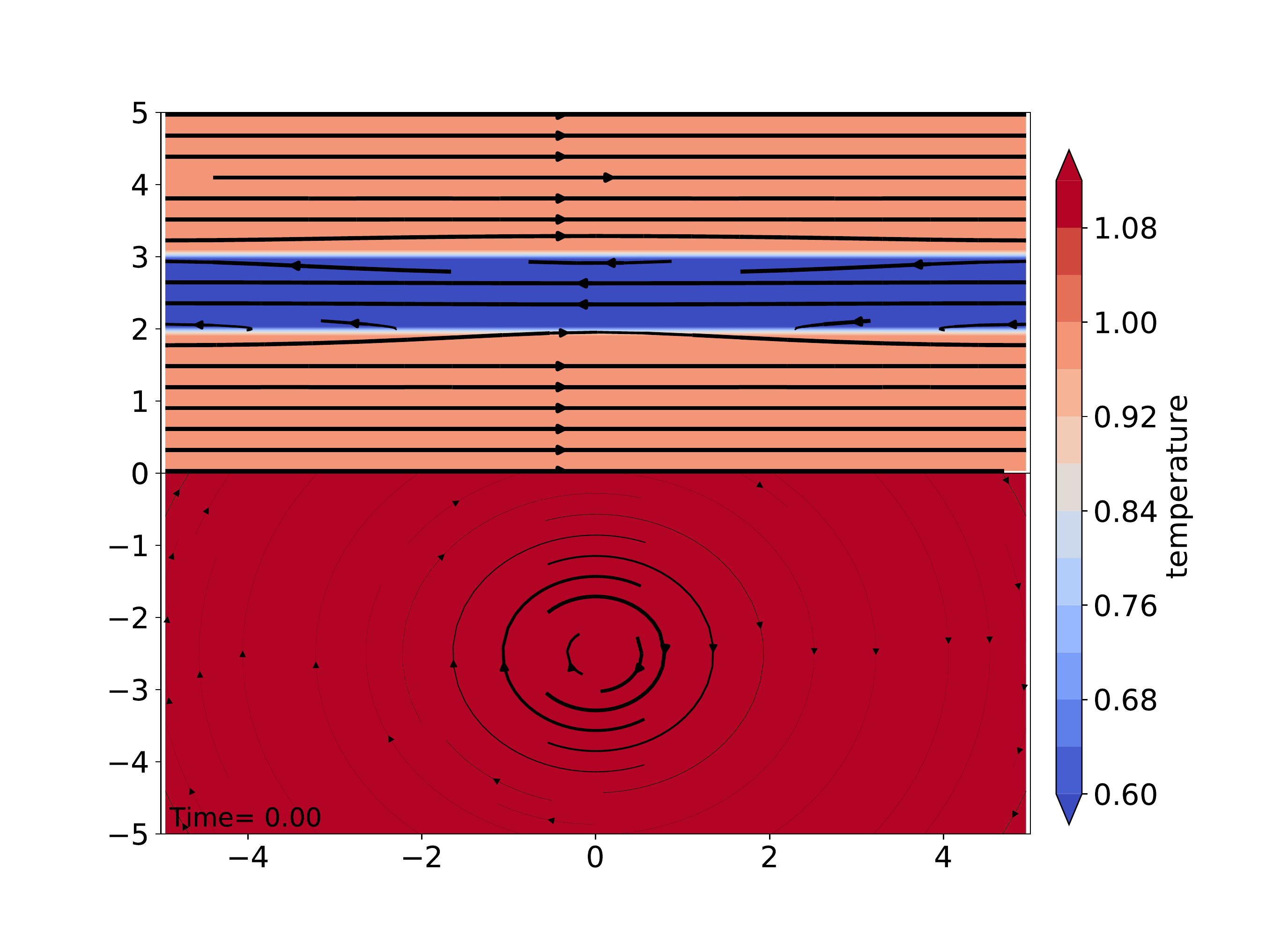}
    \includegraphics[trim=2.5cm 1.5cm 2.65cm 2.2cm,
      clip=true,width=0.47\columnwidth]{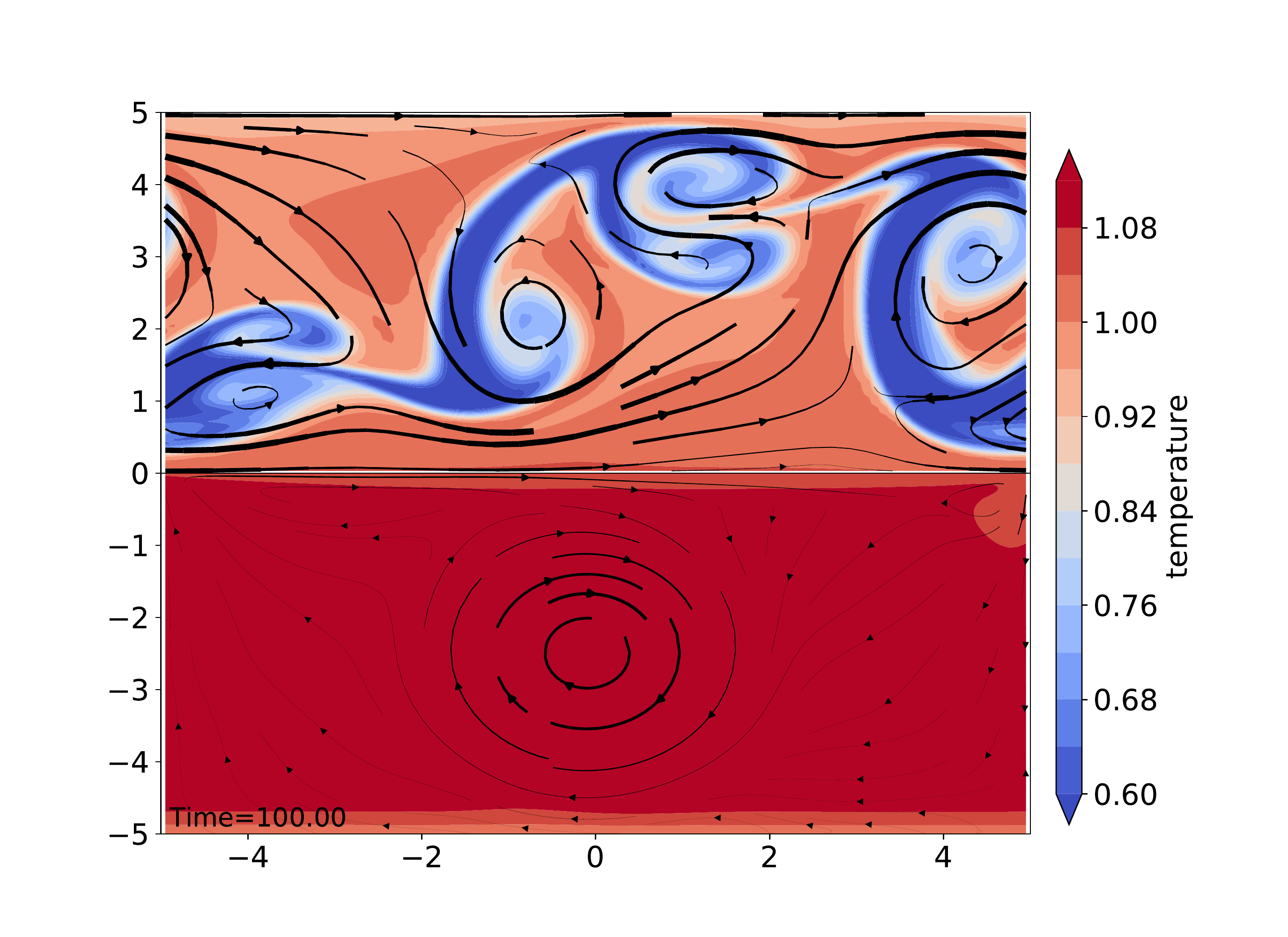}
    \includegraphics[trim=2.5cm 1.5cm 2.65cm 2.2cm,
      clip=true,width=0.47\columnwidth]{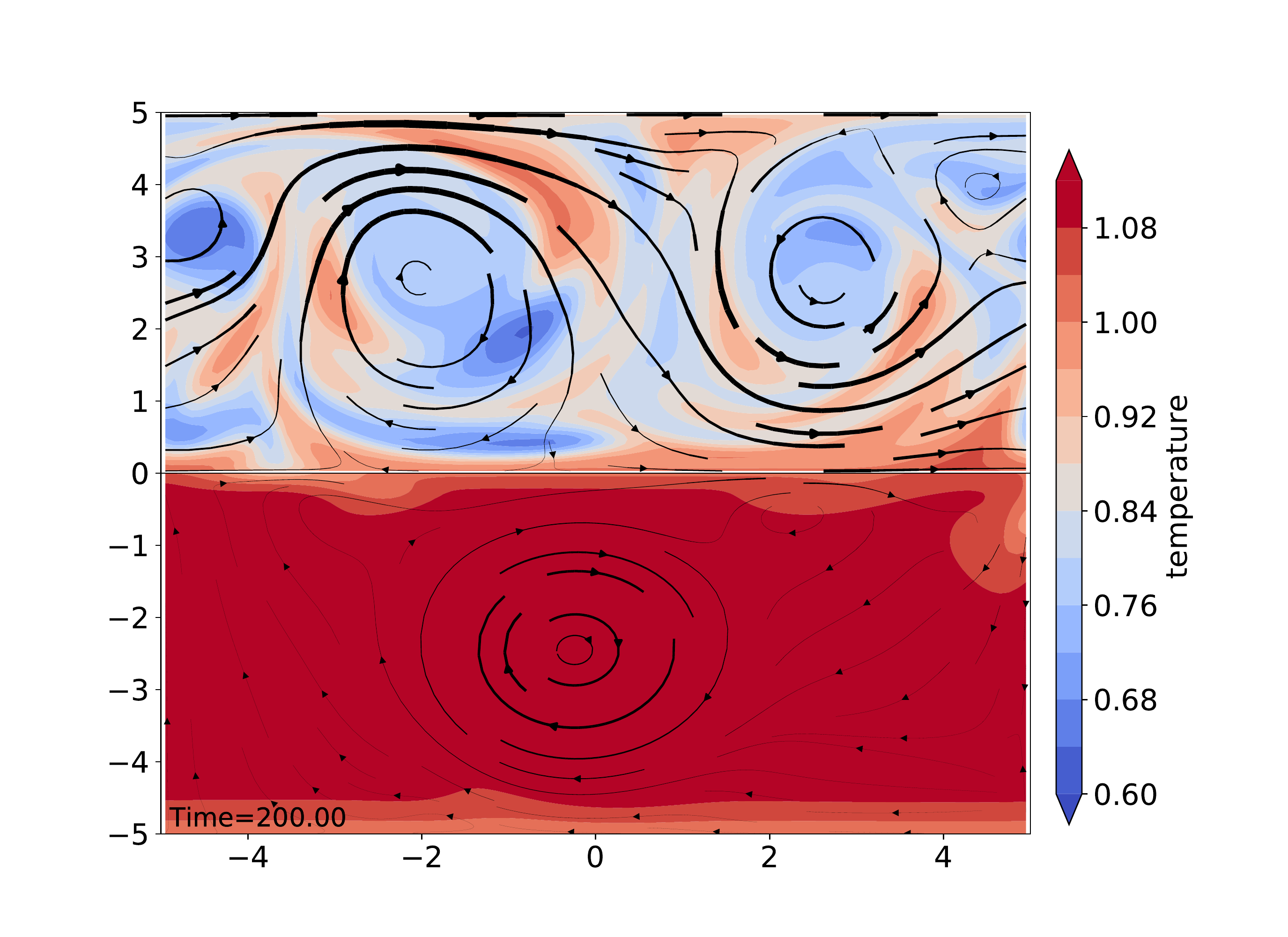}
    \includegraphics[trim=2.5cm 1.5cm 2.65cm 2.2cm,
      clip=true,width=0.47\columnwidth]{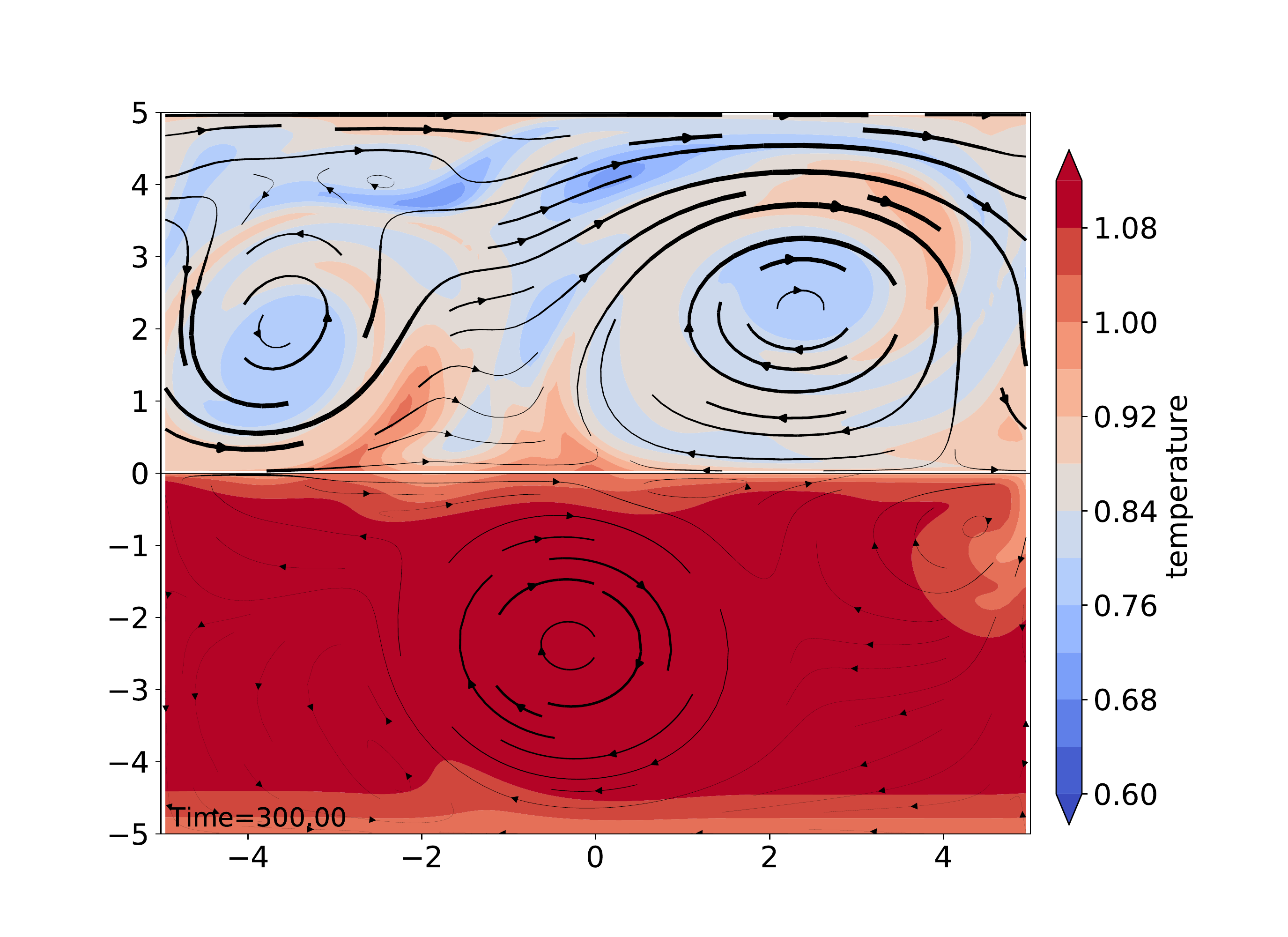}
    \includegraphics[trim=2.5cm 1.5cm 2.65cm 2.2cm,
      clip=true,width=0.47\columnwidth]{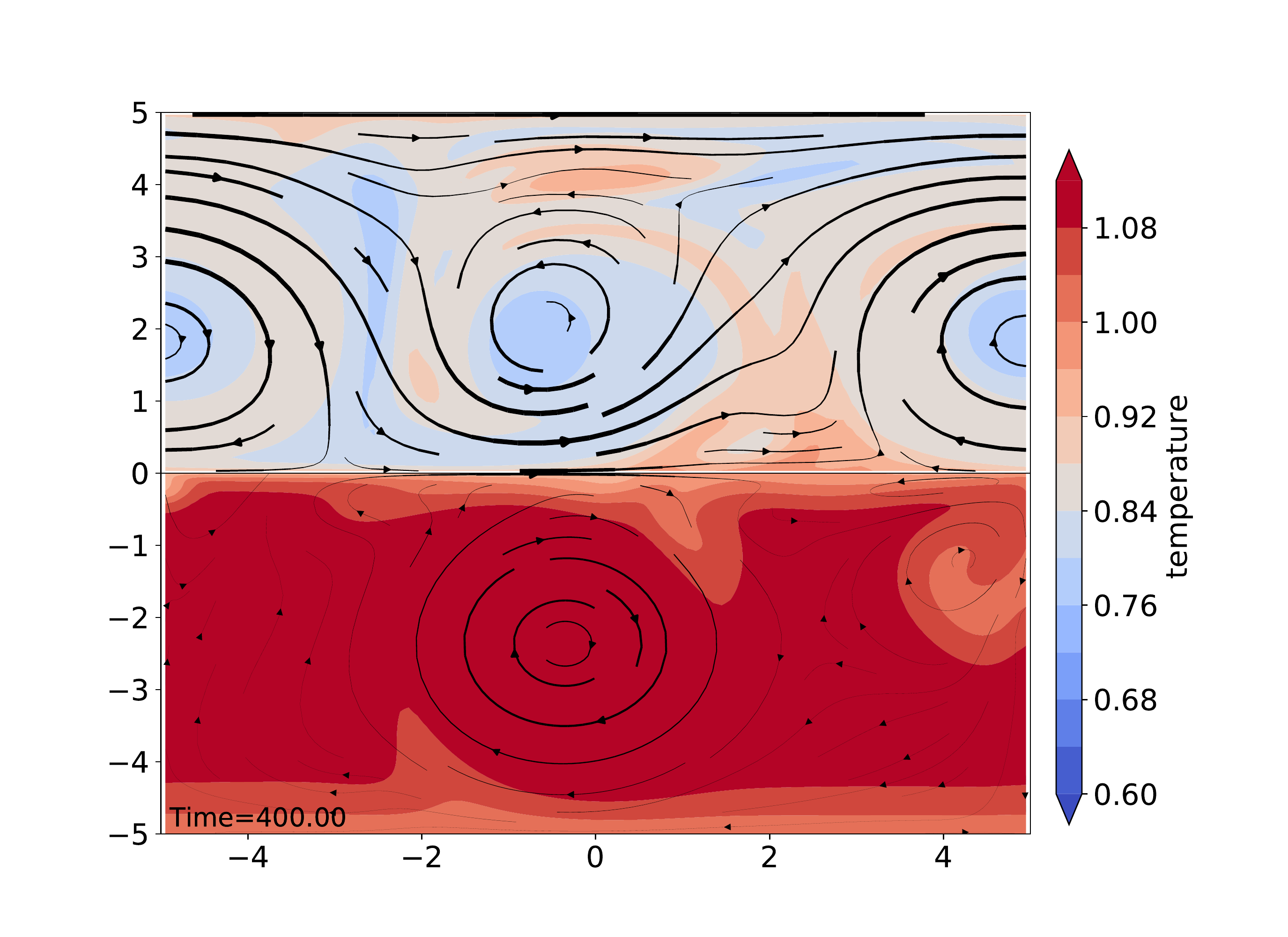}
    \includegraphics[trim=2.5cm 1.5cm 2.65cm 2.2cm,
      clip=true,width=0.47\columnwidth]{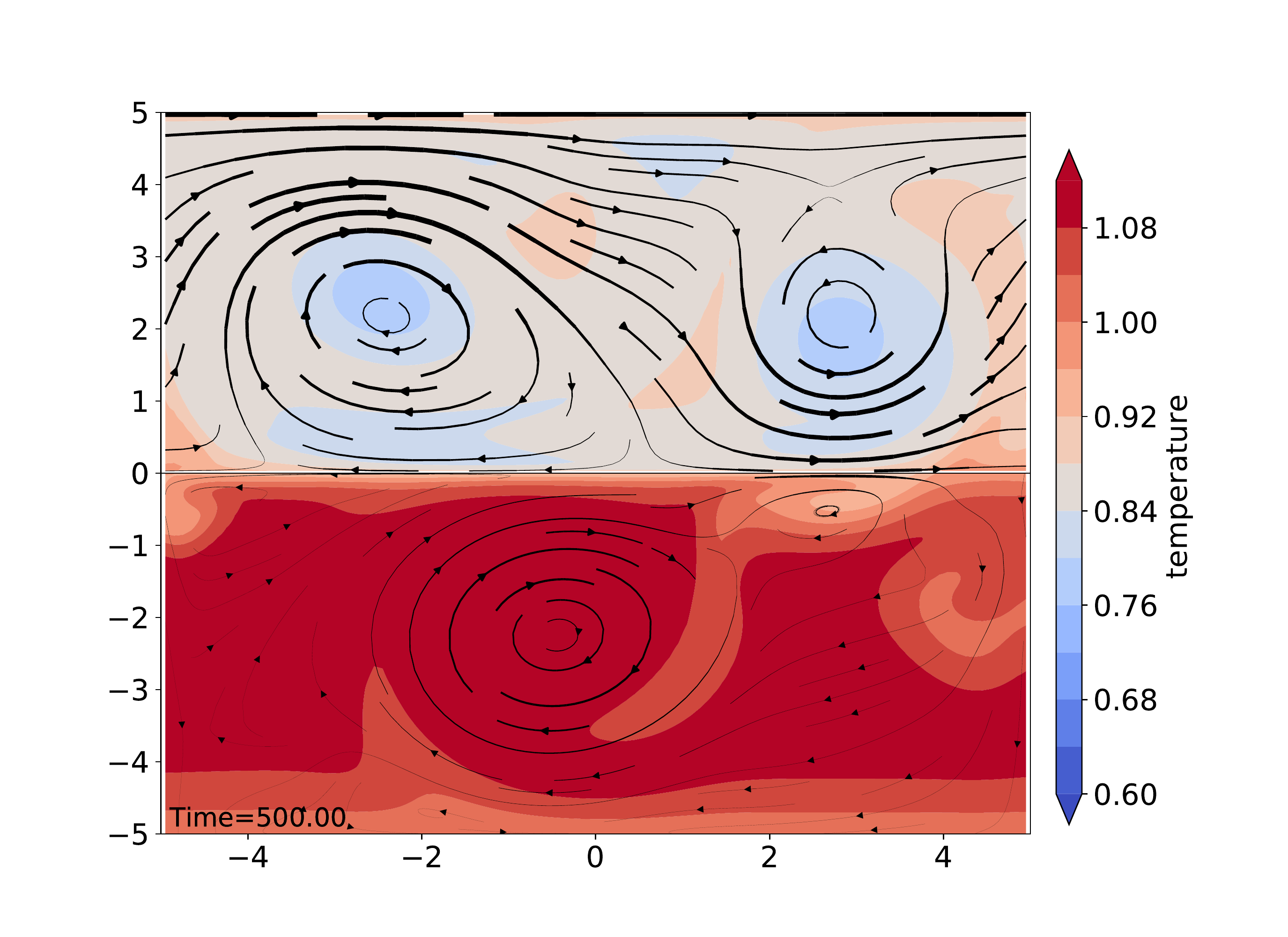}
  \caption{Evolution of the temperature field for the wind-driven flows with Kelvin--Helmholtz instability:
  simulation is conducted with the ARK4 ($\Lb^z$,TC) method over 
 a mesh of $100\times400$ elements on $\Omega_1$ 
  and $100 \times 80$ elements on $\Omega_2$ for $t\in\LRs{0,500}$. }
  
  \figlab{khi-re5000-hevi-ark4-tc-ss}
\end{figure}

Figure \figref{khi-re5000-IMEX-ss} shows temperature fields at $t=500$ for RK4, 
ARK2 ($\Lb^z$,TC),
ARK2 ($\Lb^z$,SC2), and ARK2 ($\Lb^z$,CC2).
We choose the timestep sizes as $\dt=0.01$ for RK4, 
$\dt=0.04$ for ARK2 ($\Lb^z$,TC),
$\dt=0.05$ for ARK2 ($\Lb^z$,SC2) and ARK2 ($\Lb^z$,CC2), and 
$\dt=0.1$ for ARK2 ($\Lb$,SC8) and ARK2 ($\Lb$,CC8). 
In general, all ARK2 solutions demonstrate good agreement with the RK4 solution.


\begin{figure}[h!t!b!]
  \centering
  \subfigure[RK4]{
    \includegraphics[trim=2.5cm 1.5cm 2.65cm 2.2cm,
      clip=true,width=0.47\columnwidth]{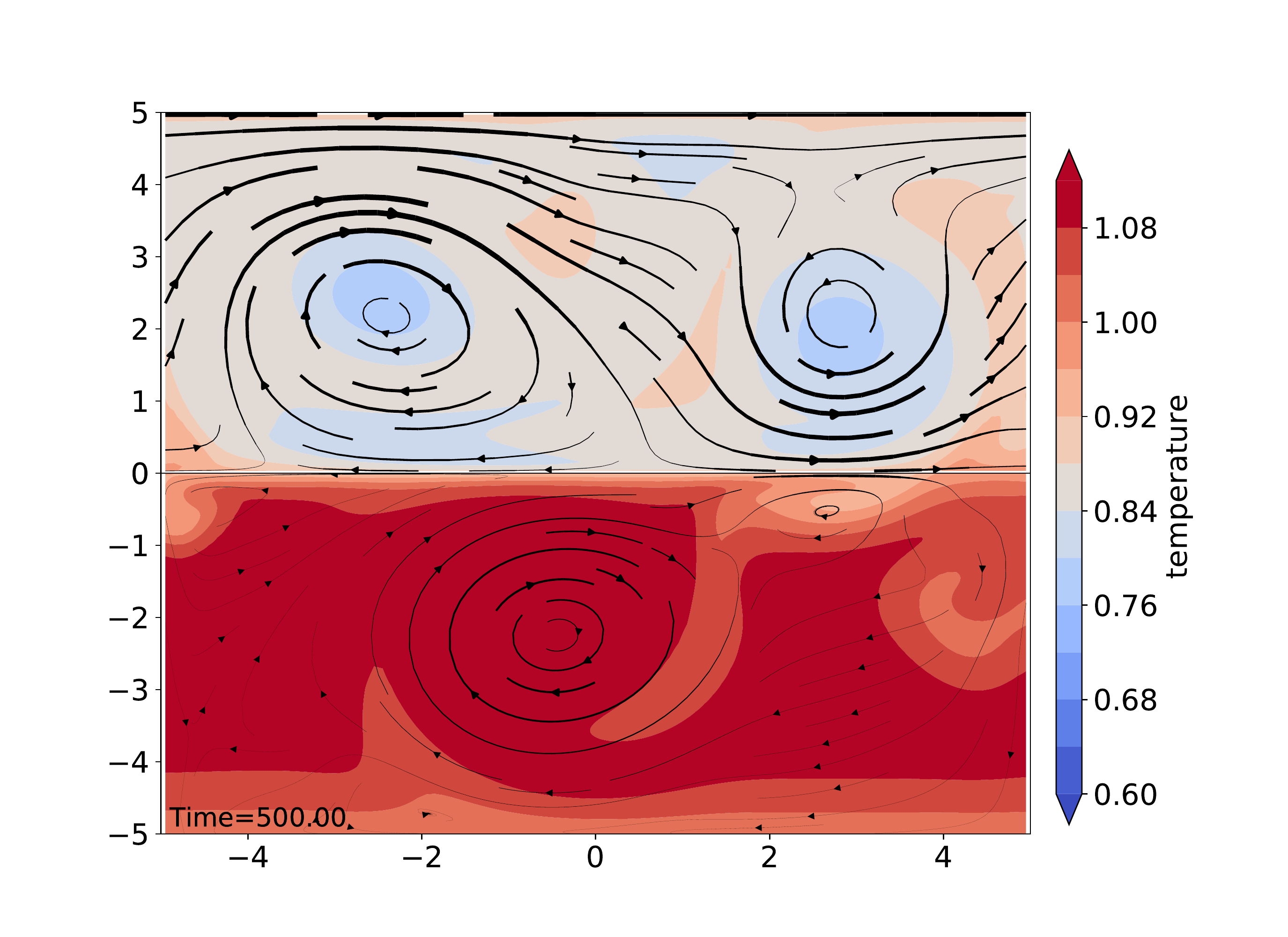}
  }
  \subfigure[ARK2($\Lb^z$,TC)]{
    \includegraphics[trim=2.5cm 1.5cm 2.65cm 2.2cm,
      clip=true,width=0.47\columnwidth]{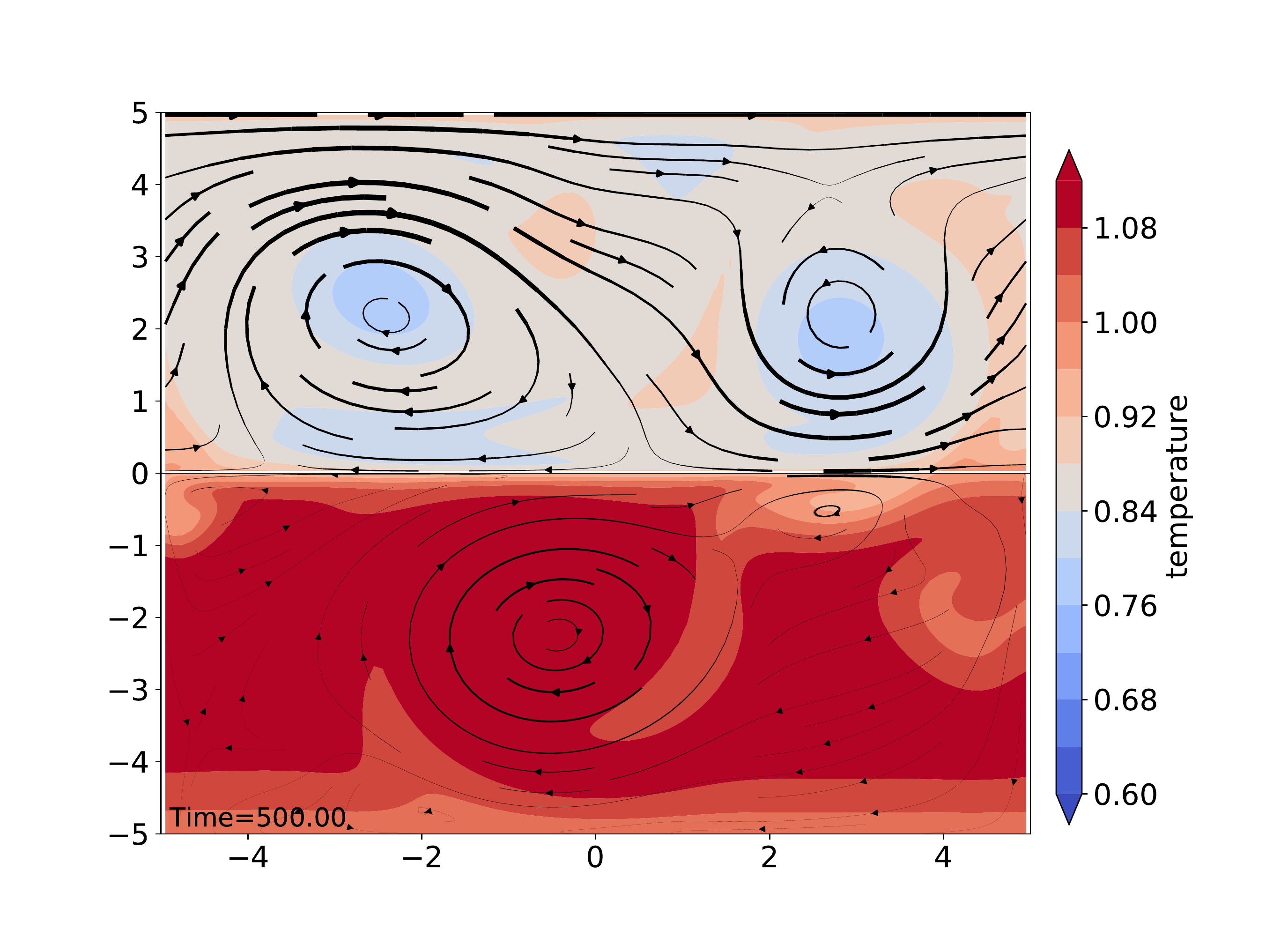}
  }
  \subfigure[ARK2($\Lb^z$,SC2)]{
    \includegraphics[trim=2.5cm 1.5cm 2.65cm 2.2cm,
      clip=true,width=0.47\columnwidth]{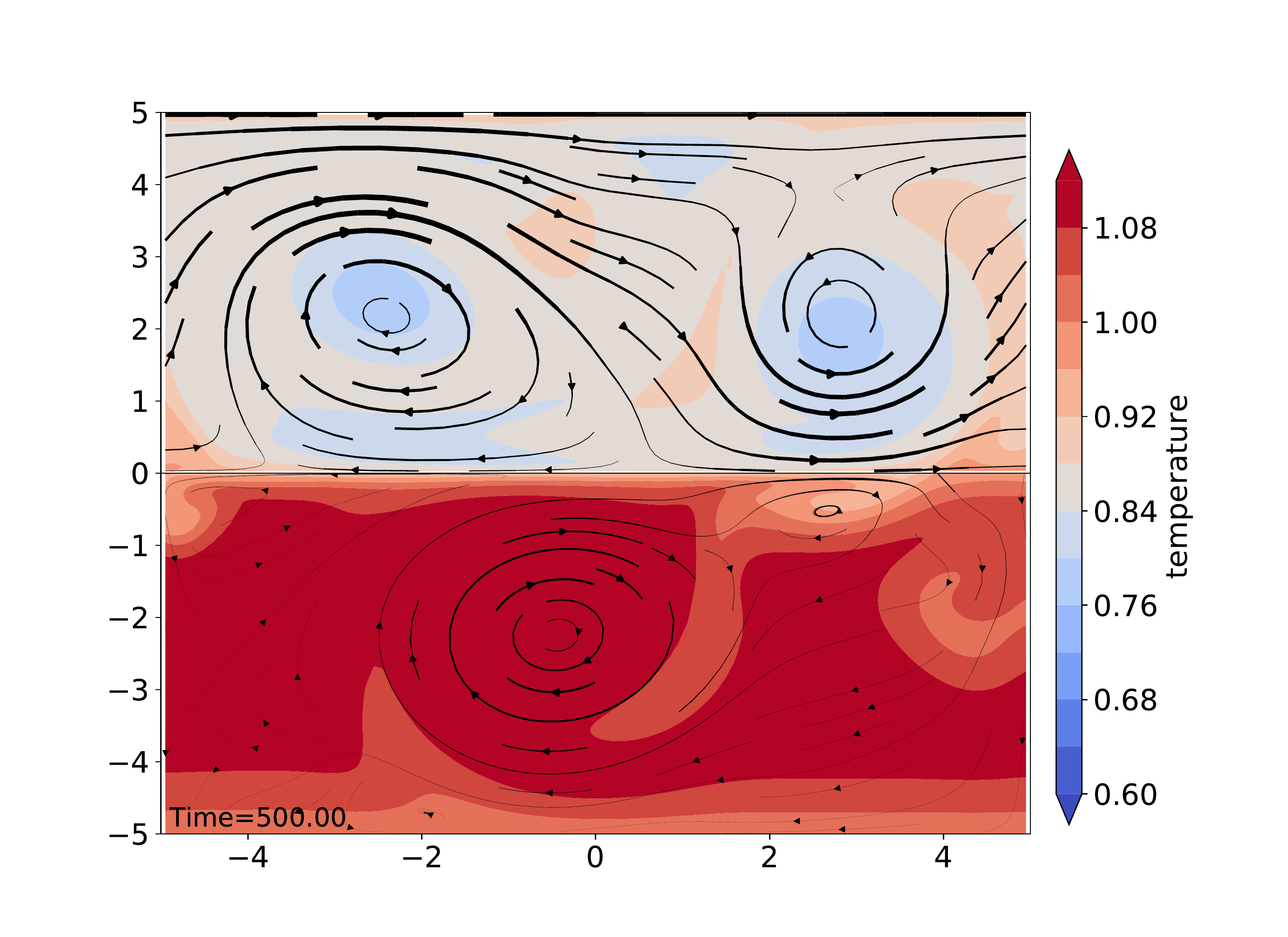}
  }
  \subfigure[ARK2($\Lb^z$,CC2)]{
    \includegraphics[trim=2.5cm 1.5cm 2.65cm 2.2cm,
      clip=true,width=0.47\columnwidth]{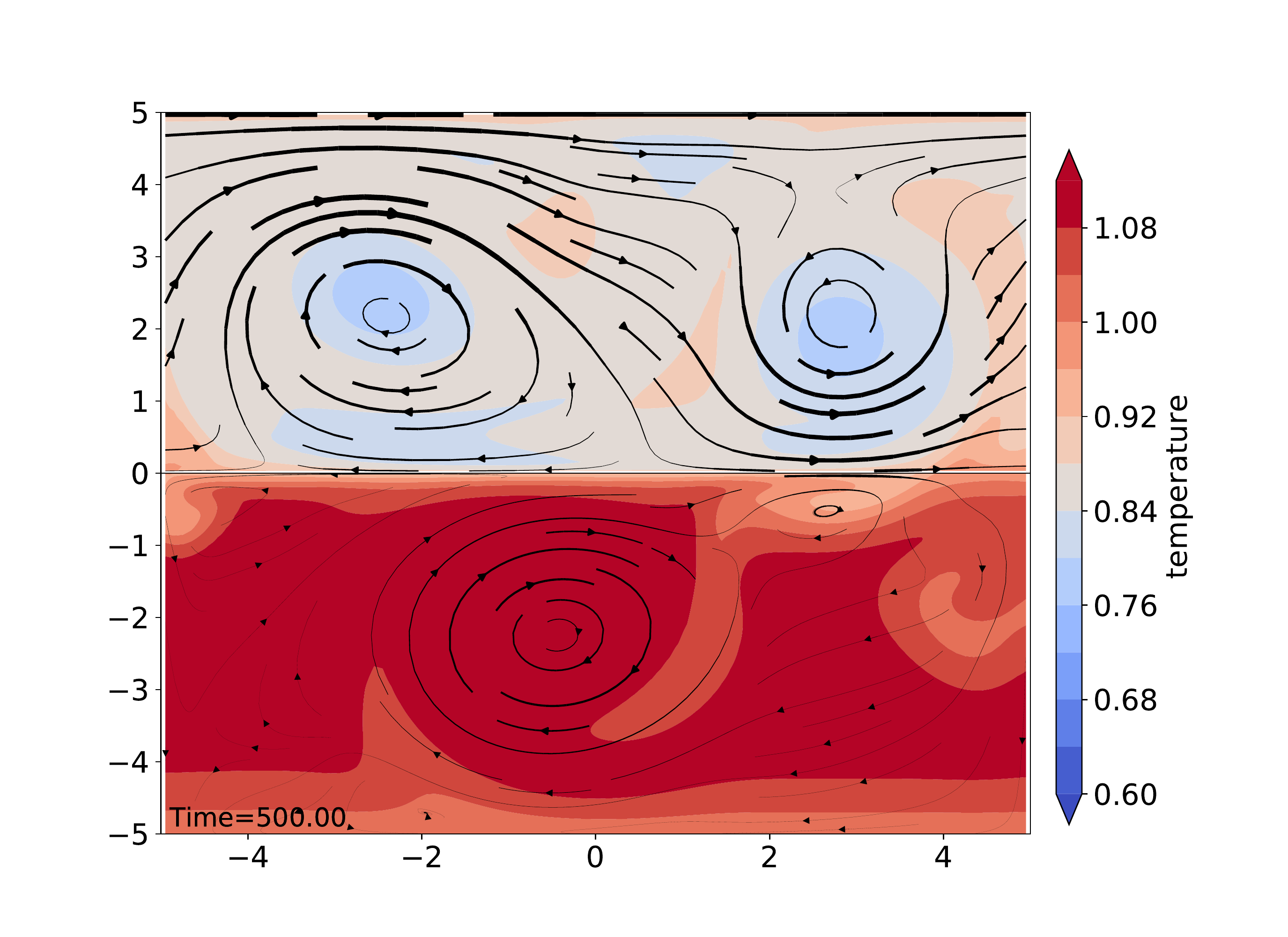}
  }
  \subfigure[ARK2($\Lb$,SC8)]{
    \includegraphics[trim=2.5cm 1.5cm 2.65cm 2.2cm,
      clip=true,width=0.47\columnwidth]{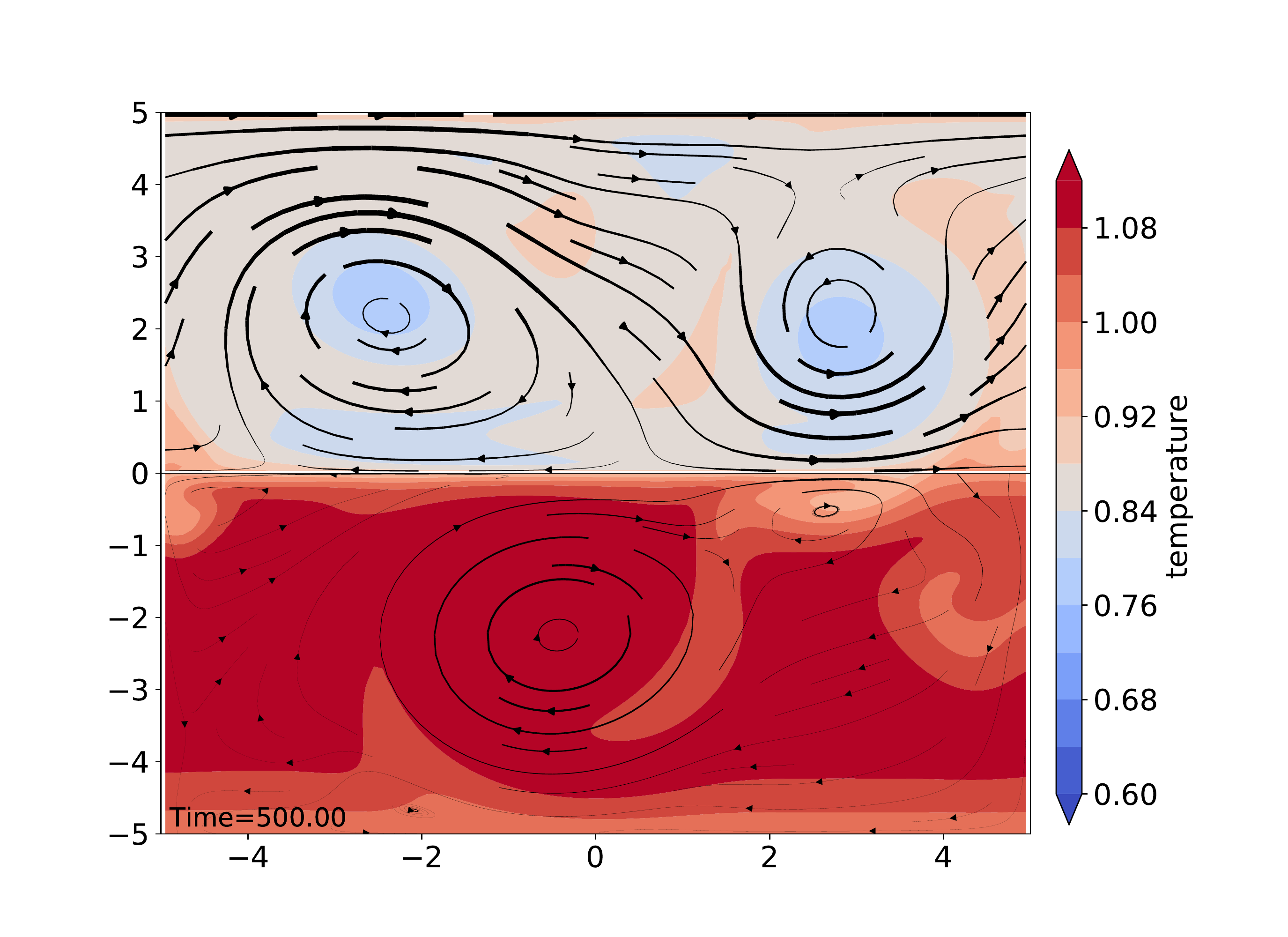}
  }
  \subfigure[ARK2($\Lb$,CC8)]{
    \includegraphics[trim=2.5cm 1.5cm 2.65cm 2.2cm,
      clip=true,width=0.47\columnwidth]{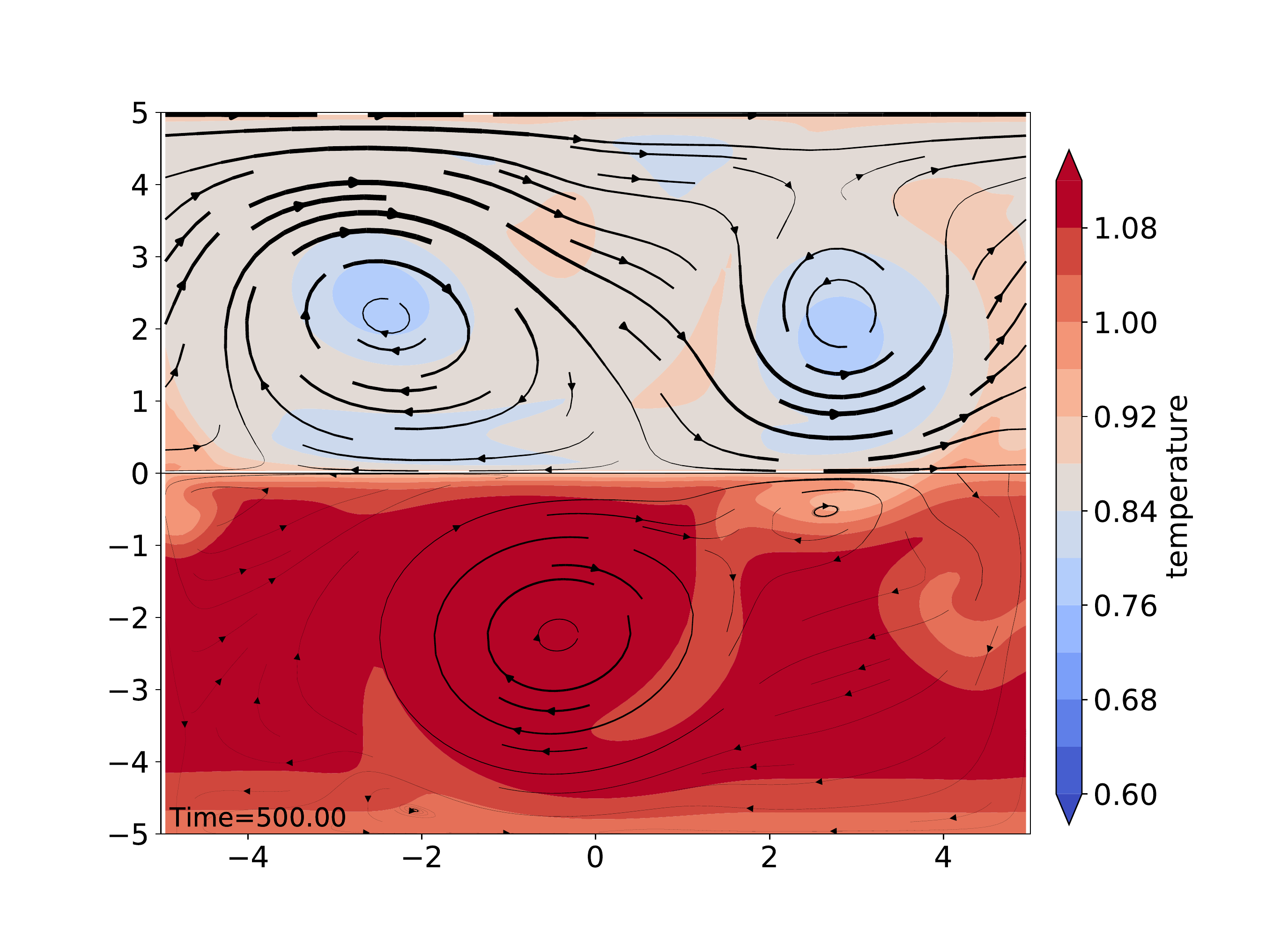}
  }
  \caption{Temperature fields for KHI at $t=500$: (a) RK4, 
  (b) ARK2($\Lb^z$,TC), (c) ARK2($\Lb^z$,SC2), (d) ARK2($\Lb^z$,CC2),
  (e) ARK2($\Lb$,SC8), and (f) ARK2($\Lb$,CC8) coupling methods.
  Simulations are performed 
  on a grid with $100\times400$ elements on $\Omega_1$ 
  and $100 \times 80$ elements on $\Omega_2$.}
  \figlab{khi-re5000-IMEX-ss}
\end{figure}

In Table \tabref{khi-re5000-imex-coupling} we report the relative errors of several IMEX coupling methods with respect to the RK4 solution (with $\dt=0.01$).
Compared with concurrent coupling (CC) methods, sequential coupling (SC) methods show slightly better accuracy.
The relative errors of ARK2 ($\Lb^z$,SC2) and ARK3 ($\Lb^z$,SC2) are smaller than those of the ARK2 ($\Lb^z$,CC2) and ARK3 ($\Lb^z$,CC2) within $\mathcal{O}(10^{-4})$. 
Similarly, ARK2 ($\Lb^z$,SC8) has small relative errors of density, momentum, and total energy compared with those of ARK2 ($\Lb^z$,CC8). 
With the same timestep size, 
the ARK4 ($\Lb^z$,TC) solution is closer to the RK4 solution than that of ARK3 ($\Lb^z$,TC).
With loose coupling (SC2 and CC2) methods, 
ARK4 ($\Lb^z$, SC2/CC2) shows better accuracy than ARK2 ($\Lb^z$, SC2/CC2) for density and total energy, 
and ARK3 ($\Lb^z$, SC2/CC2) for all variables.
However, the
ARK3 ($\Lb^z$) solutions are farther away from the RK4 solution than that of ARK2 ($\Lb^z$).
Nevertheless, with the same timestep, 
the loose coupling (SC2 and CC2) methods  
produce closer solutions to RK4 than do the tight coupling methods 
because SC2 and CC2 have two explicit subcycles. 
We have already observed similar behavior in the wind-driven flow example. 
When we double the timestep size of ARK2 ($\Lb^z$,SC2/CC2), the solutions become unstable.
The linear operator $\Lb^z$ 
is not sufficient to capture the stiff components in the system when $\dt=0.1$. 
However, ARK2 ($\Lb$,SC8/CC8) with $\dt=0.1$ results in stable solutions. 
This result implies that 
as we increase a timestep size beyond a certain Courant number (between $Cr_1=4.56$ and $Cr_1=9.13$),
the viscous part becomes stiff and thus  starts to affect the numerical stability.
Including the viscous part to the linear operator enhances numerical stability but also increase computational cost in this study.\footnote{
  We solve the linear system using Krylov subspace method without any preconditioner. 
  To improve the solver performance, 
  a proper preconditioner needs to be equipped. 
} 
As for wall-clock time, ARK2 ($\Lb^z$) is cheaper than RK4, 
and ARK3 ($\Lb^z$) is comparable to RK4 in this example.
Note that we introduce geometric stiffness vertically to use the HEVI approach.
When a mesh is anisotropic, IMEX coupling methods can be beneficial.


\begin{table}[t] 
  \caption{Relative errors and wall-clock times for IMEX coupling methods with respect to the RK4 coupling at $t=500$.
  The domain is discretized with $100\times400$ on $\Omega_1$ and $100 \times 80$ on $\Omega_2$. $Cr_2$ is the Courant number for the entire timestep; the substep number needs to be scaled by $1/k$ for SC$\{k\}$ and CC$\{k\}$.
  }
  \tablab{khi-re5000-imex-coupling} 
  \begin{center} 
    \begin{tabular}{*{1}{c}|*{1}{c}|*{1}{c}|*{1}{c}|*{1}{c}|*{1}{c}}  
    \hline 
    \multirow{2}{*}{ }
    & \multirow{2}{*}{$ \triangle t (Cr_1,Cr_2)$}
    & \multirow{2}{*}{$ \left\Vert \rho  - \rho_r \right\Vert  $}  
    & \multirow{2}{*}{$ \left\Vert \rho {\bf u} - \rho {\bf u}_r \right\Vert  $}  
    & \multirow{2}{*}{$ \left\Vert \rho E - \rho E_r \right\Vert  $}  
    & \multirow{2}{*}{wc[s]} \tabularnewline 
    & && & &  \tabularnewline 
    \hline\hline 
    RK4 & 0.01 (0.91,0.18) &       -&       -&       -& 3968.209\tabularnewline
    \multicolumn{6}{c}{} \tabularnewline
    ARK2 ($\Lb^z$,TC)  & 0.04(3.65,0.70)&       5.768E-03&       3.534E-03&       7.369E-04& 2407.544\tabularnewline
    ARK2 ($\Lb^z$,SC2) & 0.05(4.56,0.88)&       1.143E-03&       5.181E-04&       6.271E-04& 2356.430\tabularnewline
    ARK2 ($\Lb^z$,CC2) & 0.05(4.56,0.88)&       1.757E-03&       8.697E-04&       6.612E-04& 2240.349\tabularnewline
    ARK2 ($\Lb^z$,SC8) & 0.05(4.56,0.88) &      8.361E-04&       6.511E-04&       6.220E-04& 2929.486\tabularnewline
    ARK2 ($\Lb^z$,CC8) & 0.05(4.56,0.88) &      1.437E-03&       9.995E-04&       6.597E-04& 2999.596\tabularnewline
    ARK2 ($\Lb$,SC8)   & 0.10(9.13,1.76)&       1.631E-03&       1.305E-03&       1.569E-03& 6794.486\tabularnewline    
    ARK2 ($\Lb$,CC8)   & 0.10(9.13,1.76)&       2.916E-03&       1.987E-03&       1.633E-03& 6542.753\tabularnewline
    \multicolumn{6}{c}{} \tabularnewline
    ARK3 ($\Lb^z$,TC)   & 0.05(4.56,0.88)&       1.160E-02&       7.594E-03&       2.373E-03& 3403.247\tabularnewline
    ARK3 ($\Lb^z$,SC2)  & 0.05(4.56,0.88)&       1.171E-03&       1.284E-03&       2.170E-03& 3781.536\tabularnewline
    ARK3 ($\Lb^z$,CC2)  & 0.05(4.56,0.88)&       1.719E-03&       1.605E-03&       2.184E-03& 3890.584\tabularnewline
    \multicolumn{6}{c}{} \tabularnewline
    ARK4 ($\Lb^z$,TC)   & 0.05(4.56,0.88)&       4.452E-03&       2.702E-03&       4.910E-04& 4706.067\tabularnewline
    ARK4 ($\Lb^z$,SC2)  & 0.05(4.56,0.88)&       9.121E-04&       7.120E-04&       3.522E-04& 5402.731\tabularnewline
    ARK4 ($\Lb^z$,CC2)  & 0.05(4.56,0.88)&       1.475E-03&       1.084E-03&       4.261E-04& 5120.688\tabularnewline    
   \hline\hline 
  \end{tabular} 
  \end{center} 
\end{table}

Next we plot the time series of total mass and total energy losses in Figure \figref{khi-re5000-IMEX-history}.
The total mass and total energy losses are defined by
\begin{align*}
  \text{mass loss}   &:= \snor{ \text{mass}(t) - \text{mass}(0)}, \\
  \text{energy loss} &:= \snor{ \text{energy}(t) - \text{energy}(0)}, 
\end{align*}
where 
\begin{align*}
  \text{energy} = \int_{\Omega_1} \rho E d\Omega_1 + \int_{\Omega_2} \rho E d\Omega_2 
  = \sum_{m=1}^2 \sum_{\ell=1}^{\Ne{m}} \overline{\rho E}_{m_\ell} \snor{K_{m_\ell}} 
\end{align*}
and $\overline{\rho E}_{m_\ell} = \snor{K_{m_\ell}}^{-1} \int_{K_{m_\ell}} \rho E dK$.
In Figure \figref{khi-re5000-IMEX-history}(a) 
we numerically observe that total mass is conserved with IMEX coupling methods regardless of tight or loose coupling methods.
The total mass losses for IMEX tight and loose coupling methods are within $\mathcal{O}(10^{-13})$.
This means that the total mass changes with time, but fluates around the initial total mass within $10^{-13}$. 
This result makes sense because we do not exchange the mass across the interface but 
adjust the wall temperature in the interface. 
This rigid-lid condition blocks the vertical motion of the interface: 
no mass flux is allowed, and hence the total mass is conserved. 
As for the total energy loss, a peak is observed near $t=100$ 
(when strong KHI is observed in Figure \figref{khi-re5000-hevi-ark4-tc-ss}), 
and then the total energy loss decreases as time passes.


\begin{figure}[h!t!b!]
  \centering
  \subfigure[Mass loss]{
    \includegraphics[trim=1.0cm 1.0cm 0.5cm 0.2cm,
      clip=true,width=0.47\columnwidth]{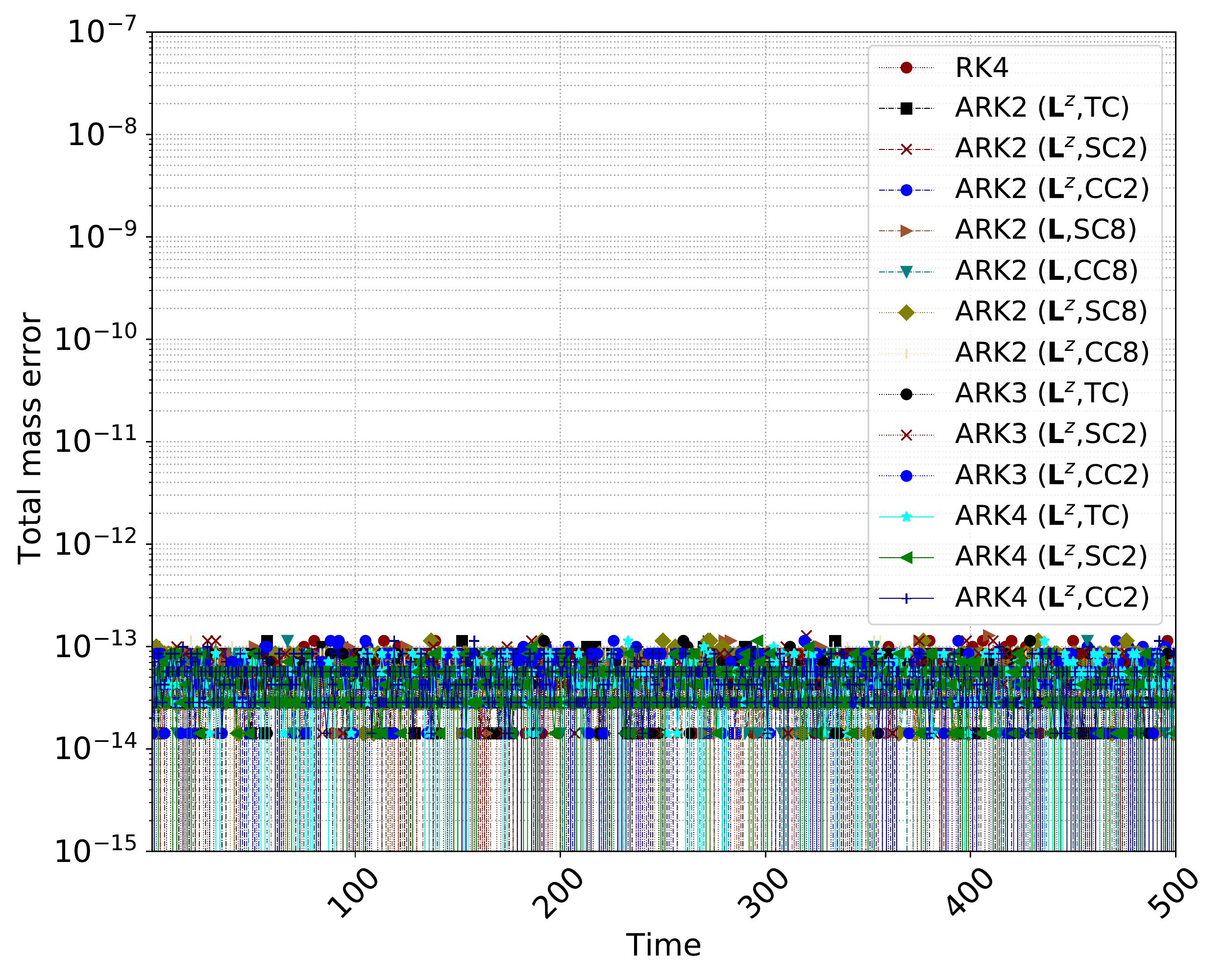}
  }
  \subfigure[Total energy loss]{
    \includegraphics[trim=1.0cm 1.0cm 0.5cm 0.2cm,
      clip=true,width=0.47\columnwidth]{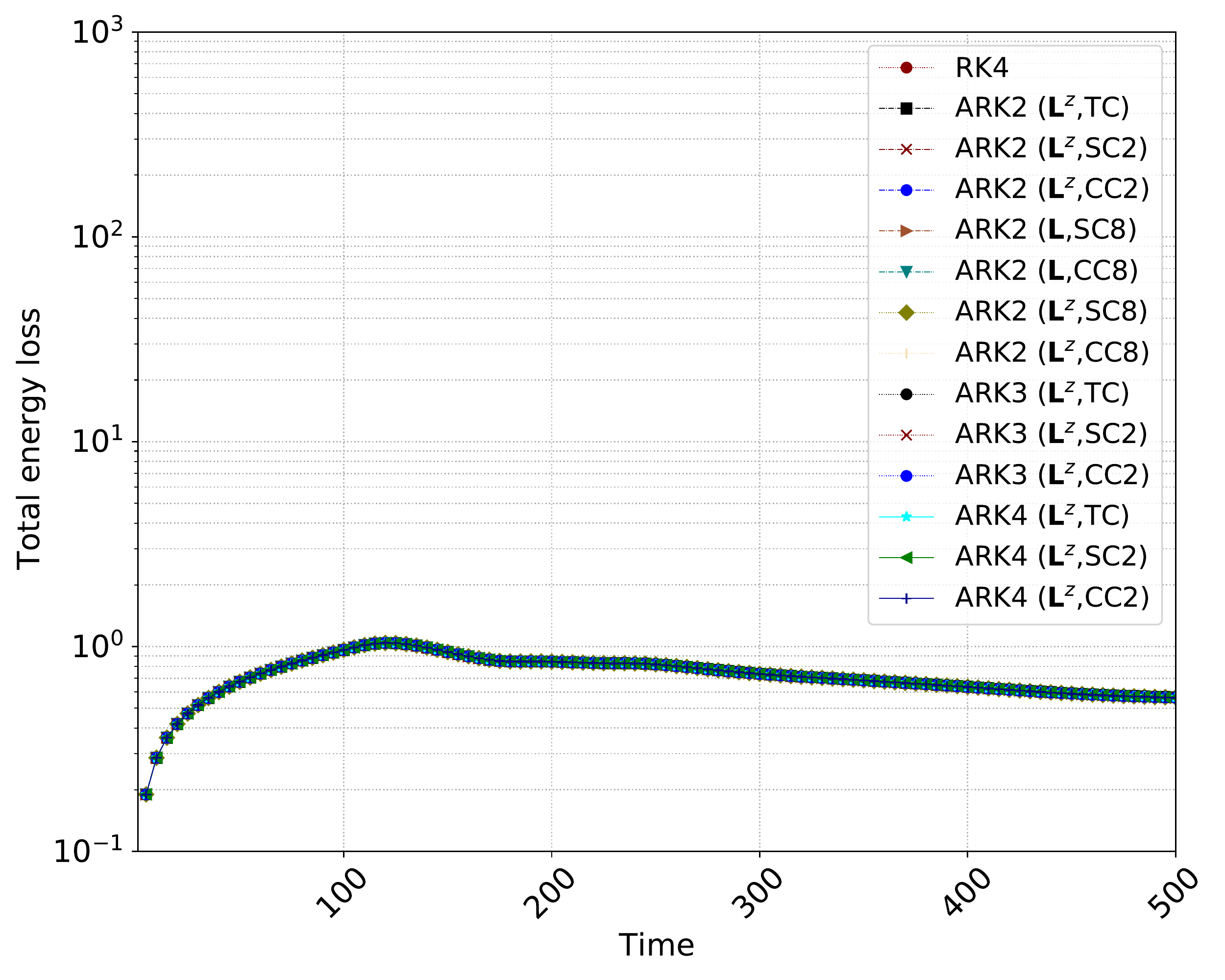}
  }
  
  \caption{Wind-driven flows with KHI: histories of (a) total mass loss and (b) total energy loss for $t\in\LRs{0,500}$.
  Total mass loss is bounded within $\mathcal{O}(10^{-13})$. 
  Total energy loss has a peak near $t=100$ when strong KHI is observed in Figure \figref{khi-re5000-hevi-ark4-tc-ss}. }
  \figlab{khi-re5000-IMEX-history}
\end{figure}

\section{Conclusions}
\seclab{Conclusion}
 
In this paper, we have developed IMEX coupling methods for compressible Navier--Stokes systems 
with the rigid-lid coupling condition arising from the atmosphere and the ocean interaction. 
We compute 
horizontal momentum and heat fluxes across the interface by using the bulk formula, 
from which we estimate wall temperature and horizontal velocity.
These estimated values serve as the isothermal moving wall boundary conditions on each model.
Each model is solved within the IMEX (tight or loose) coupling framework. 
IMEX coupling methods solve one domain (atmosphere) explicitly and the other (ocean) implicitly.
To enhance computation efficiency, we adapt IMEX time integrators, 
which can handle scale-separable stiffness or geometrically induced stiffness, 
as an implicit solver for the ocean model. 
Furthermore, we employ a horizontally explicit and vertically implicit (HEVI) approach 
where solutions are obtained column by column; hence, the resulting linear system 
is significantly reduced 
compared with two-dimensional IMEX methods. 
 
IMEX tight coupling methods 
naturally support two-way coupling at every stage.
Thus, 
the continuity of the heat and the horizontal momentum fluxes 
is guaranteed by construction. 
These methods also facilitate high-order solutions in time
and relax geometrically induced stiffness. 
However, 
IMEX tight coupling methods treat one domain (atmosphere) explicitly, and 
therefore the stiffness in the domain (atmosphere) restricts the maximum timestep size. 
Also, each coupled model advances with the same timestep size. We can alleviate the restriction by loosening the tight coupling condition. 
That is, instead of two-way coupling at each stage, 
we exchange the interface information at a certain time.
Heat and horizontal momentum fluxes are no longer continuous 
across the interface at the stages, but total mass is strictly conserved. 
In the Kelvin--Helmholtz instability example, the total mass loss is less than $\mathcal{O}(10^{-13})$. 
We conjecture that the total mass will still be conserved with other formulations such as a coupled compressible and incompressible Navier--Stokes systems as long as the spatial discretization is conservative. 
As for temporal convergence, 
IMEX loose coupling methods achieve a first-order convergent rate in the two moving-vortex example.
However, the relative errors of the IMEX loose coupling schemes are smaller than those of their IMEX tight coupling counterparts.
The reason  is that the temporal discretization errors are reduced by adding substeps. 

We have investigated two loose coupling strategies: concurrent and sequential. 
Both  exchange interface solutions  
before advancing the ocean model.
In concurrent coupling, each model is run independently, which is attractive for parallel computing. In contrast,  in   
sequential coupling, the ocean model is run first and then the atmospheric model is advanced with a number of substeps. 
In the latter case, because all stage values from the ocean model are available, 
necessary interface information can be interpolated and transferred 
from the ocean model to the atmospheric model at each stage.
This capability slightly improves the accuracy but is not significant in our simulations. 

We also observe that the choice of the linear operator affects numerical stability.
The viscous term is not a dominant source of stiffness with a small timestep size; thus
the linear operator $\Lb^z$ is sufficient to relax stiffness in the ocean model.
 As the timestep size increases, however, $\Lb^z$ is not sufficient  
because viscous terms become stiffer.   
In this case, using the linear $\Lb$ that contains both the inviscid and the viscous parts increases coupling stability. 
As for the computational cost, the
ARK2 ($\Lb^z$) and ARK3 ($\Lb^z$) coupling methods show  wall-clock time comparable to that of the RK4 method
within $\mathcal{O}(10^{-2})$ relative errors.
Note that for compressible Navier--Stokes systems, 
the relative errors of IMEX coupling methods are caused mainly  by suppressing fast acoustic waves. 
The flow patterns of IMEX couplings methods show good agreement with the RK counterpart in our numerical examples.
  
     
In the current study, we  investigated IMEX coupling schemes 
on two ideal gas fluids without gravity. 
For realistic atmosphere and ocean interaction, the
gravity source term and the material properties of the ocean should be taken into account.
Also, coupling incompressible and compressible Navier--Stokes system is also worthwhile for reflecting the current global atmosphere and ocean models. 
Since these may affect the stability of coupling schemes, 
ongoing work focuses on improving our testbed to fit realistic models in modern computing architectures.
Other time integrators such as multirate or exponential integrators may represent more  effective ways for addressing limitations in the couple performance. 

\section*{Acknowledgments}
This material is based upon work supported by the U.S. Department of Energy, Office of Science, Office of Advanced Scientific Computing Research and Office of Biological and Environmental Research, Scientific Discovery through Advanced Computing (SciDAC) program under Contract DE-AC02-06CH11357 through the Coupling Approaches for Next-Generation Architectures (CANGA) Project.

\bibliography{main}

\begin{thebibliography}{10}
\expandafter\ifx\csname url\endcsname\relax
  \def\url#1{\texttt{#1}}\fi
\expandafter\ifx\csname urlprefix\endcsname\relax\def\urlprefix{URL }\fi
\expandafter\ifx\csname href\endcsname\relax
  \def\href#1#2{#2} \def\path#1{#1}\fi

\bibitem{golaz2019doe}
J.-C. Golaz, P.~M. Caldwell, L.~P. Van~Roekel, M.~R. Petersen, Q.~Tang, J.~D.
  Wolfe, G.~Abeshu, V.~Anantharaj, X.~S. Asay-Davis, D.~C. Bader, et~al., The
  {DOE} {E3SM} coupled model version 1: Overview and evaluation at standard
  resolution, Journal of Advances in Modeling Earth Systems 11~(7) (2019)
  2089--2129.

\bibitem{hazeleger2010ec}
W.~Hazeleger, C.~Severijns, T.~Semmler, S.~{\c{S}}tef{\u{a}}nescu, S.~Yang,
  X.~Wang, K.~Wyser, E.~Dutra, J.~M. Baldasano, R.~Bintanja, et~al.,
  {EC}-{E}arth: a seamless earth-system prediction approach in action, Bulletin
  of the American Meteorological Society 91~(10) (2010) 1357--1364.

\bibitem{gentine2019coupling}
P.~Gentine, J.~K. Green, M.~Gu{\'e}rin, V.~Humphrey, S.~I. Seneviratne,
  Y.~Zhang, S.~Zhou, Coupling between the terrestrial carbon and water cycles a
  review, Environmental Research Letters 14~(8) (2019) 083003.

\bibitem{burrows2020doe}
S.~Burrows, M.~Maltrud, X.~Yang, Q.~Zhu, N.~Jeffery, X.~Shi, D.~Ricciuto,
  S.~Wang, G.~Bisht, J.~Tang, et~al., The {DOE} {E3SM} v1. 1 biogeochemistry
  configuration: description and simulated ecosystem-climate responses to
  historical changes in forcing, Journal of Advances in Modeling Earth Systems
  12~(9) (2020) e2019MS001766.

\bibitem{collins2006community}
W.~D. Collins, C.~M. Bitz, M.~L. Blackmon, G.~B. Bonan, C.~S. Bretherton, J.~A.
  Carton, P.~Chang, S.~C. Doney, J.~J. Hack, T.~B. Henderson, et~al., The
  {C}ommunity {C}limate {S}ystem {M}odel version 3 ({CCSM3}), Journal of
  Climate 19~(11) (2006) 2122--2143.

\bibitem{anderson2016co2}
T.~R. Anderson, E.~Hawkins, P.~D. Jones, {CO2}, the greenhouse effect and
  global warming: from the pioneering work of arrhenius and callendar to
  today's {E}arth {S}ystem {M}odels, Endeavour 40~(3) (2016) 178--187.

\bibitem{hoffman2019effect}
M.~J. Hoffman, X.~Asay-Davis, S.~F. Price, J.~Fyke, M.~Perego, Effect of
  subshelf melt variability on sea level rise contribution from thwaites
  glacier, antarctica, Journal of Geophysical Research: Earth Surface.

\bibitem{craig2012new}
A.~P. Craig, M.~Vertenstein, R.~Jacob, A new flexible coupler for earth system
  modeling developed for {CCSM4} and {CESM1}, The International Journal of High
  Performance Computing Applications 26~(1) (2012) 31--42.

\bibitem{jacob2005m}
R.~Jacob, J.~Larson, E.~Ong, M$\times${N} communication and parallel
  interpolation in {C}ommunity {C}limate {S}ystem {M}odel version 3 using the
  {M}odel {C}oupling {T}oolkit, The International Journal of High Performance
  Computing Applications 19~(3) (2005) 293--307.

\bibitem{valcke2013oasis3}
S.~Valcke, The {OASIS3} coupler: A {E}uropean climate modelling community
  software, Geoscientific Model Development 6~(2) (2013) 373.

\bibitem{hallberg122014numerical}
R.~Hallberg, Numerical instabilities of the ice/ocean coupled system, in:
  CLIVAR WGOMD Workshop on high, 2014, p.~38.

\bibitem{lemarie2015analysis}
F.~Lemari{\'e}, E.~Blayo, L.~Debreu, Analysis of ocean-atmosphere coupling
  algorithms: consistency and stability, Procedia Computer Science 51 (2015)
  2066--2075.

\bibitem{beljaars2017numerical}
A.~Beljaars, E.~Dutra, G.~Balsamo, F.~Lemari\'e, On the numerical stability of
  surface--atmosphere coupling in weather and climate models, Geoscientific
  Model Development 10~(2) (2017) 977--989.

\bibitem{zhang2020stability}
H.~Zhang, Z.~Liu, E.~Constantinescu, R.~Jacob, Stability analysis of interface
  conditions for ocean--atmosphere coupling, Journal of Scientific Computing
  84~(3) (2020) 1--25.

\bibitem{peterson2019explicit}
K.~Peterson, P.~Bochev, P.~Kuberry, Explicit synchronous partitioned algorithms
  for interface problems based on {L}agrange multipliers, Computers \&
  Mathematics with Applications 78~(2) (2019) 459--482.

\bibitem{sockwell2020interface}
K.~C. Sockwell, K.~Peterson, P.~Kuberry, P.~Bochev, N.~Trask, Interface flux
  recovery coupling method for the ocean--atmosphere system, Results in Applied
  Mathematics (2020) 100110.

\bibitem{lemarie2014sensitivity}
F.~Lemari{\'e}, P.~Marchesiello, L.~Debreu, E.~Blayo, Sensitivity of
  ocean-atmosphere coupled models to the coupling method: example of tropical
  cyclone {E}rica.

\bibitem{bresch2006operator}
D.~Bresch, J.~Koko, Operator-splitting and {L}agrange multiplier domain
  decomposition methods for numerical simulation of two coupled
  {N}avier--{S}tokes fluids, International Journal of Applied Mathematics and
  Computer Science 16 (2006) 419--429.

\bibitem{formaggia2001coupling}
L.~Formaggia, J.-F. Gerbeau, F.~Nobile, A.~Quarteroni, On the coupling of 3{D}
  and 1{D} {N}avier--{S}tokes equations for flow problems in compliant vessels,
  Computer methods in applied mechanics and engineering 191~(6-7) (2001)
  561--582.

\bibitem{connors2019stability}
J.~M. Connors, R.~D. Dolan, Stability of two conservative, high-order
  fluid-fluid coupling methods, Advances In Applied Mathematics And Mechanics
  11~(6) (2019) 1287--1338.

\bibitem{carpenter2014entropy}
M.~H. Carpenter, T.~C. Fisher, E.~J. Nielsen, S.~H. Frankel, Entropy stable
  spectral collocation schemes for the {N}avier--{S}tokes equations:
  Discontinuous interfaces, SIAM Journal on Scientific Computing 36~(5) (2014)
  B835--B867.

\bibitem{ascher1997implicit}
U.~M. Ascher, S.~J. Ruuth, R.~J. Spiteri, Implicit-explicit {R}unge-{K}utta
  methods for time-dependent partial differential equations, Applied Numerical
  Mathematics 25~(2) (1997) 151--167.

\bibitem{pareschi2005implicit}
L.~Pareschi, G.~Russo, Implicit-explicit {R}unge-{K}utta schemes and
  applications to hyperbolic systems with relaxation, Journal of Scientific
  computing 25~(1-2) (2005) 129--155.

\bibitem{Constantinescu_A2010a}
E.~Constantinescu, A.~Sandu, Extrapolated {IM}plicit--{EX}plicit time stepping,
  SIAM Journal on Scientific Computing 31~(6) (2010) 4452--4477.

\bibitem{Kennedy2003additive}
C.~A. Kennedy, M.~H. Carpenter, Additive {R}unge-{K}utta schemes for
  convection-diffusion-reaction equations, Applied Numerical Mathematics
  44~(1-2) (2003) 139--181.

\bibitem{roldan2013efficient}
T.~Roldan, I.~Higueras, Efficient implicit-explicit {R}unge--{K}utta methods
  with low storage requirements, SciCADE 2013 45~(1) (2013) 174.

\bibitem{boscarino2017unified}
S.~Boscarino, L.~Pareschi, G.~Russo, A unified {IMEX} {R}unge--{K}utta approach
  for hyperbolic systems with multiscale relaxation, SIAM Journal on Numerical
  Analysis 55~(4) (2017) 2085--2109.

\bibitem{kanevsky2007application}
A.~Kanevsky, M.~H. Carpenter, D.~Gottlieb, J.~S. Hesthaven, Application of
  implicit--explicit high--order {R}unge--{K}utta methods to discontinuous
  {G}alerkin schemes, Journal of Computational Physics 225~(2) (2007)
  1753--1781.

\bibitem{kang2019imex}
S.~Kang, F.~X. Giraldo, T.~Bui-Thanh, {IMEX HDG-DG}: A coupled implicit
  hybridized discontinuous {G}alerkin and explicit discontinuous {G}alerkin
  approach for shallow water systems, Journal of Computational Physics (2019)
  109010.

\bibitem{restelli2009conservative}
M.~Restelli, F.~X. Giraldo, A conservative discontinuous {G}alerkin
  semi-implicit formulation for the {N}avier-{S}tokes equations in
  nonhydrostatic mesoscale modeling, SIAM Journal on Scientific Computing
  31~(3) (2009) 2231--2257.

\bibitem{giraldo2010high}
F.~Giraldo, M.~Restelli, High-order semi-implicit time-integrators for a
  triangular discontinuous {G}alerkin oceanic shallow water model,
  International journal for numerical methods in fluids 63~(9) (2010)
  1077--1102.

\bibitem{giraldo2013implicit}
F.~X. Giraldo, J.~F. Kelly, E.~Constantinescu, Implicit-explicit formulations
  of a three-dimensional nonhydrostatic unified model of the atmosphere
  ({NUMA}), SIAM Journal on Scientific Computing 35~(5) (2013) B1162--B1194.

\bibitem{gardner2018implicit}
D.~J. Gardner, J.~E. Guerra, F.~P. Hamon, D.~R. Reynolds, P.~A. Ullrich, C.~S.
  Woodward, Implicit--explicit ({IMEX}) {R}unge--{K}utta methods for
  non-hydrostatic atmospheric models, Geoscientific Model Development 11~(4)
  (2018) 1497--1515.

\bibitem{vogl2019evaluation}
C.~J. Vogl, A.~Steyer, D.~R. Reynolds, P.~A. Ullrich, C.~S. Woodward,
  Evaluation of implicit explicit additive {R}unge--{K}utta integrators for the
  {HOMME}--{NH} dynamical core, Journal of Advances in Modeling Earth Systems
  11~(12) (2019) 4228--4244.

\bibitem{Abdi_2019}
D.~S. Abdi, F.~X. Giraldo, E.~M. Constantinescu, L.~E. Carr, L.~C. Wilcox,
  T.~C. Warburton, Acceleration of the {IM}plicit--{EX}plicit nonhydrostatic
  unified model of the atmosphere on manycore processors, The International
  Journal of High Performance Computing Applications 33~(2) (2019) 242--267.

\bibitem{froehle2014high}
B.~Froehle, P.-O. Persson, A high-order discontinuous {G}alerkin method for
  fluid--structure interaction with efficient implicit--explicit time stepping,
  Journal of Computational Physics 272 (2014) 455--470.

\bibitem{huang2019high}
D.~Z. Huang, P.-O. Persson, M.~J. Zahr, High--order, linearly stable,
  partitioned solvers for general multiphysics problems based on implicit
  explicit {R}unge--{K}utta schemes, Computer Methods in Applied Mechanics and
  Engineering 346 (2019) 674--706.

\bibitem{washington2005introduction}
W.~M. Washington, C.~Parkinson, Introduction to three-dimensional climate
  modeling, University science books, 2005.

\bibitem{muller2006equations}
P.~M{\"u}ller, The equations of oceanic motions, Cambridge University Press,
  2006.

\bibitem{connors2012decoupled}
J.~M. Connors, J.~S. Howell, W.~J. Layton, Decoupled time stepping methods for
  fluid-fluid interaction, SIAM Journal on Numerical Analysis 50~(3) (2012)
  1297--1319.

\bibitem{liu1979bulk}
W.~T. Liu, K.~B. Katsaros, J.~A. Businger, Bulk parameterization of air-sea
  exchanges of heat and water vapor including the molecular constraints at the
  interface, Journal of the Atmospheric Sciences 36~(9) (1979) 1722--1735.

\bibitem{smith1988coefficients}
S.~D. Smith, Coefficients for sea surface wind stress, heat flux, and wind
  profiles as a function of wind speed and temperature, Journal of Geophysical
  Research: Oceans 93~(C12) (1988) 15467--15472.

\bibitem{fairall1996bulk}
C.~W. Fairall, E.~F. Bradley, D.~P. Rogers, J.~B. Edson, G.~S. Young, Bulk
  parameterization of air-sea fluxes for tropical ocean-global atmosphere
  coupled-ocean atmosphere response experiment, Journal of Geophysical
  Research: Oceans 101~(C2) (1996) 3747--3764.

\bibitem{bao2000numerical}
J.~Bao, J.~Wilczak, J.~Choi, L.~Kantha, Numerical simulations of air--sea
  interaction under high wind conditions using a coupled model: A study of
  hurricane development, Monthly Weather Review 128~(7) (2000) 2190--2210.

\bibitem{panosfsky1984atmospheric}
H.~Panosfsky, J.~Dutton, Atmospheric turbulence: Models and methods for
  engineering applications, NewYork: JohnWiley\&Sons.

\bibitem{vickers2015formulation}
D.~Vickers, L.~Mahrt, E.~L. Andreas, Formulation of the sea surface friction
  velocity in terms of the mean wind and bulk stability, Journal of Applied
  Meteorology and Climatology 54~(3) (2015) 691--703.

\bibitem{jacobs2007conservative}
G.~B. Jacobs, D.~A. Kopriva, F.~Mashayek, A conservative isothermal wall
  boundary condition for the compressible {N}avier--{S}tokes equations, Journal
  of Scientific Computing 30~(2) (2007) 177--192.

\bibitem{roe1986characteristic}
P.~Roe, Characteristic-based schemes for the {E}uler equations, Annual review
  of fluid mechanics 18~(1) (1986) 337--365.

\bibitem{toro2013riemann}
E.~F. Toro, Riemann solvers and numerical methods for fluid dynamics: a
  practical introduction, Springer Science \& Business Media, 2013.

\bibitem{nishikawa2011two}
H.~Nishikawa, Two ways to extend diffusion schemes to navier-stokes schemes:
  Gradient formula or upwind flux, in: 20th AIAA Computational Fluid Dynamics
  Conference, 2011, p. 3044.

\bibitem{syrakos2017critical}
A.~Syrakos, S.~Varchanis, Y.~Dimakopoulos, A.~Goulas, J.~Tsamopoulos, A
  critical analysis of some popular methods for the discretisation of the
  gradient operator in finite volume methods, Physics of Fluids 29~(12) (2017)
  127103.

\bibitem{cristini2012some}
P.~Cristini, D.~Komatitsch, Some illustrative examples of the use of a
  spectral-element method in ocean acoustics, The Journal of the Acoustical
  Society of America 131~(3) (2012) EL229--EL235.

\bibitem{constantinescu2007multirate}
E.~M. Constantinescu, A.~Sandu, Multirate timestepping methods for hyperbolic
  conservation laws, Journal of Scientific Computing 33~(3) (2007) 239--278.

\bibitem{schlegel2009multirate}
M.~Schlegel, O.~Knoth, M.~Arnold, R.~Wolke, Multirate {R}unge--{K}utta schemes
  for advection equations, Journal of Computational and Applied Mathematics
  226~(2) (2009) 345--357.

\bibitem{constantinescu2013extrapolated}
E.~M. Constantinescu, A.~Sandu, Extrapolated multirate methods for differential
  equations with multiple time scales, Journal of Scientific Computing 56~(1)
  (2013) 28--44.

\bibitem{seny2013multirate}
B.~Seny, J.~Lambrechts, R.~Comblen, V.~Legat, J.-F. Remacle, Multirate time
  stepping for accelerating explicit discontinuous {G}alerkin computations with
  application to geophysical flows, International Journal for Numerical Methods
  in Fluids 71~(1) (2013) 41--64.

\bibitem{sandu2019class}
A.~Sandu, A class of multirate infinitesimal {GARK} methods, SIAM Journal on
  Numerical Analysis 57~(5) (2019) 2300--2327.

\bibitem{weller2013runge}
H.~Weller, S.-J. Lock, N.~Wood, {R}unge-{K}utta {IMEX} schemes for the
  horizontally explicit/vertically implicit ({HEVI}) solution of wave
  equations, Journal of Computational Physics 252 (2013) 365--381.

\bibitem{thuburn2008some}
J.~Thuburn, Some conservation issues for the dynamical cores of {NWP} and
  climate models, Journal of Computational Physics 227~(7) (2008) 3715--3730.

\bibitem{ghosh2016semi}
D.~Ghosh, E.~M. Constantinescu, Semi-implicit time integration of atmospheric
  flows with characteristic-based flux partitioning, SIAM Journal on Scientific
  Computing 38~(3) (2016) A1848--A1875.

\bibitem{taylor1937mechanism}
G.~I. Taylor, A.~E. Green, Mechanism of the production of small eddies from
  large ones, Proceedings of the Royal Society of London. Series A-Mathematical
  and Physical Sciences 158~(895) (1937) 499--521.

\bibitem{young2010dynamic}
W.~R. Young, Dynamic enthalpy, conservative temperature, and the seawater
  {B}oussinesq approximation, Journal of physical oceanography 40~(2) (2010)
  394--400.

\bibitem{bruneau20062d}
C.-H. Bruneau, M.~Saad, The 2{D} lid-driven cavity problem revisited, Computers
  \& fluids 35~(3) (2006) 326--348.

\bibitem{drazin2004hydrodynamic}
P.~G. Drazin, W.~H. Reid, Hydrodynamic stability, Cambridge university press,
  2004.

\bibitem{springel2010pur}
V.~Springel, E pur si muove: {G}alilean-invariant cosmological hydrodynamical
  simulations on a moving mesh, Monthly Notices of the Royal Astronomical
  Society 401~(2) (2010) 791--851.

\bibitem{lecoanet2016validated}
D.~Lecoanet, M.~McCourt, E.~Quataert, K.~J. Burns, G.~M. Vasil, J.~S. Oishi,
  B.~P. Brown, J.~M. Stone, R.~M. O'Leary, A validated non-linear
  {K}elvin--{H}elmholtz benchmark for numerical hydrodynamics, Monthly Notices
  of the Royal Astronomical Society 455~(4) (2016) 4274--4288.

\end{thebibliography}

 \begin{center}
	\scriptsize \framebox{\parbox{4in}{Government License (will be removed at publication):
			The submitted manuscript has been created by UChicago Argonne, LLC,
			Operator of Argonne National Laboratory (``Argonne").  Argonne, a
			U.S. Department of Energy Office of Science laboratory, is operated
			under Contract No. DE-AC02-06CH11357.  The U.S. Government retains for
			itself, and others acting on its behalf, a paid-up nonexclusive,
			irrevocable worldwide license in said article to reproduce, prepare
			derivative works, distribute copies to the public, and perform
			publicly and display publicly, by or on behalf of the Government. The Department of Energy will provide public access to these results of federally sponsored research in accordance with the DOE Public Access Plan. http://energy.gov/downloads/doe-public-access-plan.
}}
	\normalsize
\end{center}

\end{document}